\documentclass[11pt,a4paper]{article}
\usepackage[dvipsnames]{xcolor}
\usepackage[all,2cell,cmtip]{xy}
\UseTwocells
\entrymodifiers={+!!<0pt,\fontdimen22\textfont2>}

\usepackage{fouriernc}  % Current font.
\usepackage[english]{babel}  % Spacing and hyphenation rules for English.

\usepackage{verbatim} 

\usepackage{amsmath}
\usepackage{amsfonts}
\usepackage{amssymb}
\usepackage{mathrsfs}    %allows \mathscr
\usepackage{amsthm} % Introduces proof environment. Redefines \@endtheorem so that indentation is not suppressed on following paragraph.
\swapnumbers  % Theorem number precedes heading.

\usepackage[style=alphabetic,backend=biber]{biblatex}
\usepackage{csquotes}
\usepackage[colorlinks=true,linkcolor=MidnightBlue,citecolor=OliveGreen]{hyperref}
\usepackage{bookmark}

\theoremstyle{plain}
\newtheorem*{thm0}{Theorem}
\newtheorem*{cor0}{Corollary}

\newtheorem{thm}{Theorem}[subsection]
\newtheorem{prop}[thm]{Proposition}
\newtheorem{cor}[thm]{Corollary}

\newtheorem{lemma}[thm]{Lemma}
\newtheorem*{lemma0}{Lemma}

\newtheorem{hyp}[thm]{Hypothesis}
\theoremstyle{definition}
\newtheorem*{def0}{Definition}

\newtheorem{para}[thm]{}
\newtheorem{defn}[thm]{Definition}

\theoremstyle{remark}
\newtheorem{remark}[thm]{Aside}
\newtheorem{eg}[thm]{Example}
\newtheorem*{eg0}{Example}

\newcommand{\proofof}[1]{\end{#1}\begin{proof}}

\makeatletter
\renewcommand\section{\@startsection {section}{1}{\z@}%
  {-3.5ex \@plus -1ex \@minus -.2ex}{2.3ex \@plus.2ex}%
  {\normalfont\large\bfseries}}
\renewcommand\subsection{\@startsection{subsection}{2}{\z@}%
  {-3.25ex\@plus -1ex \@minus -.2ex}{1.5ex \@plus .2ex}%
  {\normalfont\bfseries}}

\newcommand{\lie}[1]{\mathfrak{#1}}
\newcommand{\sh}[1]{\mathcal{#1}}
\newcommand{\N}{{\mathbb N}}
\newcommand{\Z}{{\mathbb Z}}

\newcommand{\R}{{\mathbb R}}
\newcommand{\C}{{\mathbb C}}

\newcommand{\iden}{\mathrm{id}}

\DeclareMathOperator*\colim{colim}
\DeclareMathAlphabet{\mathrmsl}{OT1}{cmr}{m}{sl}

\newcommand{\rssymb}[2]{\newcommand{#1}{\mathrmsl{#2}} }
\newcommand{\oper}[3][n]{\newcommand{#2}{\mathop{\mathrm{#3}}%
\ifx n#1\nolimits\else\limits\fi} }
\newcommand{\rsoper}[3][n]{\newcommand{#2}{\mathop{\mathrmsl{#3}}%
\ifx n#1\nolimits\else\limits\fi} }

\renewcommand{\bf}[1]{\mathbf{#1}}
\renewcommand{\rm}[1]{\mathrm{#1}}

\oper\Ad{Ad}
\oper\ad{ad}

\oper\val{val}
\oper\coker{coker}
\oper\mult{mult}
\oper\Iso{Iso}
\oper\End{End}
\oper\Nat{Nat}
\oper\Aut{Aut}
\oper\Sub{Sub}
\oper\Alt{Alt}
\oper\Ext{Ext}
\oper\Pic {Pic}
\oper\Sym{Sym}
\oper\Spec{Spec}
\oper\Spf{Spf}
\oper\Sp{Sp}
\oper\Spa{Spa}
\oper\Proj{Proj}
\rsoper\divg{div}
\rsoper{\sym}{sym}
\rsoper{\alt}{alt}
\rsoper\trace{tr}
\rssymb\id{id}

\newcommand{\thismonth}{\ifcase\month\or
  January\or February\or March\or April\or May\or June\or
  July\or August\or September\or October\or November\or December\fi
  \space\number\year}

\usepackage{enumitem}

\newcommand{\PSh}{\mathrm{PSh}}
\newcommand{\Fun}{\mathrm{Fun}}

\newcommand{\Biv}{\mathrm{Biv}}

\newcommand{\CMon}{\mathrm{CMon}}

\newcommand{\Cor}{\mathrm{Corr}}

\newcommand{\Ob}{\mathrm{Ob}}

\newcommand{\Mod}{\mathrm{Mod}}
\newcommand{\Map}{\mathrm{Map}}

\newcommand{\Cat}{\mathbf{Cat}}

\newcommand{\Fin}{\mathbf{Fin}}
\newcommand{\Spc}{\mathbf{Spc}}
\newcommand{\Cart}{\mathrm{Cart}}
\newcommand{\coCart}{\mathrm{coCart}}
\newcommand{\biCart}{\mathrm{biCart}}
\newcommand{\op}{\mathrm{op}}

\newcommand{\Spanext}{\mathrm{Spex}}
\newcommand{\sep}{\!,\,}	% improved comma spacing for scriptstyle
\DeclareMathOperator*{\return}{\mapsto}
\newcommand{\fibre}{\mathrm{fib}}	% Fibre transport functor
\newcommand{\alignhead}[1]{\noalign{\noindent\hspace{20pt}\emph{#1.}}}

\usepackage{calc}
\newcommand{\acong}{\mathrel{\makebox[\widthof{$\Rightarrow$}]{$\cong$}}} % aligned \cong symbol to width of \rightarrow

\newenvironment{labelitems}{
  \begin{list}{}{
    \setlength{\leftmargin}{0pt}
    \setlength{\itemindent}{\labelwidth}  
    \addtolength{\itemindent}{\labelsep}
    
  } 
}{ 
  \end{list} 
}

\title{A bivariant Yoneda lemma and $(\infty,2)$-categories of correspondences\thanks{AMS Math subject classification tags: 
  18G99  ``None of the above, but in this section (derived categories)'';
  18D20  ``Enriched categories'';
  18A40  ``Adjoint functors'';
  55U40  ``Topological categories'';
  55P65  ``Homotopy functors''. \\
  The bulk of this work was supported by World Premier International Research Center Initiative (WPI), MEXT,
Japan, and by JSPS KAKENHI Grant Number JP19K14522.
  Later revisions were supported by JSPS KAKENHI Grant Number JP16H06337.}}
\author{Andrew W. Macpherson}

\begin{document}

\maketitle
\begin{abstract}
A well-known folklore states that if you have a bivariant homology theory satisfying a base change formula, you get a representation of a category of correspondences. For theories in which the covariant and contravariant transfer maps are in mutual adjunction, these data are actually equivalent. In other words, a 2-category of correspondences is the universal way to attach to a given 1-category a set of right adjoints that satisfy a base change formula .

Through a bivariant version of the Yoneda paradigm, I give a definition of correspondences in higher category theory and prove an extension theorem for bivariant functors. Moreover, conditioned on the existence of a 2-dimensional Grothendieck construction, I provide a proof of the aforementioned universal property. The methods, morally speaking, employ the `internal logic' of higher category theory: they make no explicit use of any particular model.
\end{abstract}
\begin{comment}
\begin{abstract}
Through a bivariant Yoneda philosophy, I give a definition of correspondences in higher category theory and prove an extension theorem for bivariant functors. Conditioned on a 2-dimensional Grothendieck construction, I prove a universal property for higher categories of correspondences.
\end{abstract}
\end{comment}

\tableofcontents

\section{Introduction}

Let $D$ be a category with fibre products. We can define a category $\Cor_D$ of \emph{correspondences in $D$} as follows:

\begin{def0}
The category $\Cor_D$ of correspondences in $D$ comprises the following data:
\begin{enumerate}[label=(\arabic*), start=0]
\item Objects of $\Cor_D$ are objects of $D$.
\item A map $X\nrightarrow Y$ in $\Cor_D$ is a \emph{span}
  
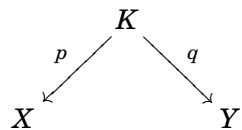
\[\xymatrix{ & K \ar[dl]_p \ar[dr]^q \\ X&&Y }\] 
\item[$(\circ)$] The composite of spans $[X\leftarrow K\rightarrow Y]$ and $[Y\leftarrow K^\prime\rightarrow Z]$ is calculated by the fibre product 
  
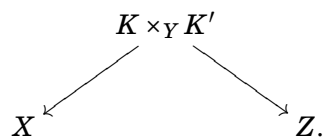
\[\xymatrix{
    & K\times_YK' \ar[dl] \ar[dr] \\ X&&Z.
  }\]
  of $K$ with $K'$ over $Y$.
\end{enumerate}
The associative law for composition can be deduced from that of fibre products. The reader may object that fibre products are defined only up to unique isomorphism and not on the nose: $\Cor_D$ is more properly a $(2,1)$-category whose mapping objects are groupoids. This issue is, in a sense, the crux of the whole paper, and we'll return to it shortly.\end{def0}

%Any correspondence as above decomposes naturally as a composite \[\xymatrix{&K\ar[dl]\ar@{=}[dr]&&K\ar[dr]\ar@{=}[dl] \\ X&&K&&Y }\] of a \emph{wrong-way} or \emph{contravariant} map $X\leftarrow K=K$ followed by a \emph{right-way} or \emph{covariant} map $K=K\rightarrow Y$. That is, the category of correspondences carries a natural \emph{orthogonal factorisation system} $\Cor_D=\langle D\perp  D^\op\rangle$. 
Each map $f:X\rightarrow D$ in $D$ defines two morphisms in $\Cor_D$: 
\begin{itemize}
\item a `right-way' map $[X\tilde\leftarrow X\stackrel{f}{\rightarrow} Y]$ from $X$ to $Y$, written $f_!$;
\item a `wrong-way' map $[Y\stackrel{f}{\leftarrow} X\tilde\rightarrow X]$ from $Y$ to $X$, written $f^!$.
\end{itemize}
These right and wrong way maps satisfy a compatibility condition called the \emph{base change formula}, which says that if 
\[\xymatrix{
  X\times_YZ \ar[r]^-{\bar g} \ar[d]_{\bar f}	& X \ar[d]^f \\
  Y \ar[r]^g 		& Z
}\]
is a fibre product square in $D$, then $g_!f^!=\bar f^!\bar g_!$, since either composite is given by the same correspondence $[X\leftarrow X\times_YZ \rightarrow Y]$. This is in a precise sense the \emph{only} relation: one can show, by explicitly checking relations, that to define a functor out of $\Cor_D$ it is enough to define a covariant and a contravariant functor out of $D$ such that this base change formula  is satisfied.

\paragraph{Correspondences and cohomology}

Students of the homology of manifolds, and its big city cousin, Grothendieck's theory of motives, will be familiar with settings in which data of the above form is available. That is, we are given a contravariant functor $H:D\rightarrow E^\op$ which has, at least for a subclass $S_D$ of morphisms in $D$, also a covariant behaviour in the form of \emph{transfer} or \emph{Gysin} maps. Perhaps after imposing further conditions, these satisfy a base change formula and hence define a functor
\[  \Spanext(H): \Cor(D;S_D) \rightarrow E^\op  \]
where $\Cor(D;S_D)\subseteq \Cor(D)$ is the subcategory where morphisms are those spans whose wrong-way component belongs to $S_D$.

\begin{eg0}[Cohomology]
Let $\bf{Sm}_\C$ denote the category of smooth quasi-projective varieties over $\C$ and let $\rm{pr|sm}$ be the class of projective submersions. De Rham cohomology with complex coefficients defines a functor
  \[ H_\rm{dR}^\bullet(-;\C):\bf{Sm}_\C \rightarrow (\rm{Vect}_\C^\Z)^\op \]
into the category of $\Z$-graded $\C$-vector spaces. For morphisms $f:X\rightarrow Y$ in $\rm{pr|sm}$ there is also a transfer map $f_!:H^\bullet(X;\C)\rightarrow H^{\bullet-\dim_\R(X/Y)}(Y;\C)$ defined using Poincar\'e duality, and this map satisfies the base change compatibility with pullbacks. (The transfer map is actually defined for any projective morphism in $\bf{Sm}_\C$, but we cannot always pull back projective morphisms and keep the domain smooth.) We therefore obtain a functor
  \[ \Spanext(\oplus_\bullet H_\rm{dR}^\bullet):\Cor( \bf{Sm}_\C, \rm{pr|sm} ) \rightarrow \rm{Vect}_\C^\op \]
  where we summed over the grading to avoid having to deal with the dimension shift. This construction parallels that of $H_\rm{dR}$ as a functor on Grothendieck's category of Chow motives.
  \end{eg0}

\begin{eg0}[Cochains]
The \emph{homotopy coherence} issue one encounters when attempting to construct a chain level version of $H_\rm{dR}$ is familiar. Although this can be solved using a sufficiently robust model for de Rham cochains, the modern approach is to find a native $\infty$-categorical construction.

Introducing Gysin maps and base change, the coherence issue is further compounded: the $\infty$-functor lift of $\Spanext(H_\rm{dR})$ must somehow encode coherences between base change homotopies in arbitrary composite grids of base change squares, in a 2-dimensional version of that classic problem of higher category theory. (The 2-cells in the $(2,1)$-category $\Cor_D$ would also have to come into play to `see' these base change homotopies.)

%Let us now try to construct a chain level version of $\rm{Span}(H_\rm{dR})$, that is, to lift it to an $\infty$-functor into into the derived $\infty$-category $D(\C)$ of $\C$. Let us moreover, for the sake of illustration, deny ourselves the luxury of invoking the remarkable properties of $C^\infty$ differential forms; our construction will work over any field of characteristic zero.
%The classic approach is to try to lift still further to a functor valued in the Abelian category of complexes. A functorial resolution of the algebraic de Rham complex, such as that of Godement, can be used to model the contravariant behaviour of de Rham cochains. When we come to the Gysin maps, however, things become subtle, as Poincar\'e duality is only a duality in $D(\C)$ and not of complexes on the nose. 
\end{eg0}

\begin{eg0}[Derived source]
If we want to generalise the domain to include objects of a homotopical nature, such as \emph{derived} manifolds or orbifolds --- and we do actually want to do this in applications to Gromov-Witten theory \cite{Mann_Robalo} --- then the source category is also an $\infty$-category; now we have to explain how to construct an $\infty$-category of correspondences from an input $\infty$-category, and how to construct $\infty$-functors out of it.\footnote{Another, completely separate, difficulty is entailed in this case, which is that Poincar\'e duality fails even at the level of homology. We don't address this here, but see \cite{Schurg}.
}
\end{eg0}

In the first instance, by-hand methods are quite sufficient for constructing a functor from the correspondences category. Subsequent examples, however, highlight two questions that we must answer before correspondences can be made available in a homotopical context:
\begin{enumerate}
\item When constructing functors from $\Cor_D$ into a higher category, how can we express the compatibility needed between the covariant and contravariant parts?

\item When $D$ is an $\infty$-category, the given \textbf{Definition} cannot be regarded as defining an $\infty$-category; it only specifies the 0- and 1-simplices, and the inner face map from two to one. How can we express the full associative structure of $\Cor_D$?
\end{enumerate}
There is an obvious way to attack problem ii) by extending the \textbf{Definition} to define a complete Segal space $\rm{Span}_{D,\bullet}$ whose associated $\infty$-category deserves to be called $\Cor_D$. This is the approach taken, for example, in \cite[Chap.\ 10]{HigherSegal}. However, when it comes to problem i), we are still left with trying to manipulate an infinite hierarchy of compatibilities.

While it is surely possible with enough effort to chase down these compatibilities directly, a more satisfying strategy is to try to define a `machine' that churns out the necessary relations automatically. For the present example, we need a machine that manufactures relations between the covariant and contravariant aspects of de Rham cohomology. To construct this machine, we will pass into the world of $(\infty,2)$-category theory.

\paragraph{Categorification}

We can extend the preceding \textbf{Definition} and enhance $\Cor_D$ to a 2-category of correspondences in $D$ by defining 2-cells as follows:
\begin{itemize}
  \item[(2)] A morphism from $[X\leftarrow K\rightarrow Y]$ to $[X\leftarrow K' \rightarrow Y]$ is a commutative diagram
  
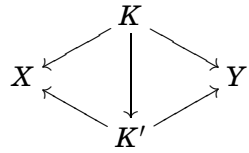
\[\xymatrix@R=2ex{
    & K \ar[dl] \ar[dd] \ar[dr] \\
    X && Y \\
    & K' \ar[ul] \ar[ur]
  }\]  
  in $D$.  
\end{itemize}
(There are other interesting possibilities for 2-cells in $\Cor_D$ \cite{GR, Haugseng}, but we will not pursue them in this paper.)

With this class of 2-cells, each `right-way' map is \emph{left adjoint}, in the sense internal to the 2-category $\Cor_D$, to the corresponding `wrong-way' map. The unit, respectively counit of the adjunction $f_!\dashv f^!$ is depicted 
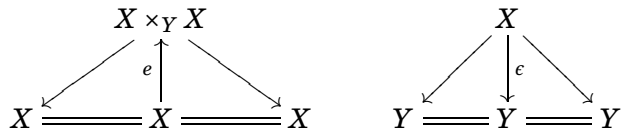
\[\xymatrix{
  & X\times_YX\ar[dr]\ar[dl] &&& X\ar[dr]\ar[d]^\epsilon\ar[dl] 
  \\ 
  X\ar@{=}[r] & X\ar[u]^e & X\ar@{=}[l] & Y\ar@{=}[r] & Y\ar@{=}[r] & Y 
}\]
Conversely, any adjoint pair of maps in $\Cor_D$ is of this form. See \cite[Lemma 12.3]{Haugseng}.

This is a much tighter relationship between the right-way and wrong-way maps than the base change formula: it says that either of these two classes `completely determines' the other.%\footnote{In \emph{op.\ cit}.\ the result is proved in an enlarged context (spans of spans, rather than only morphisms of spans, are allowed), but the preceding claims easily follow from the same constructions.}

\begin{eg0}[Derived categories]
Passing through the HKR isomorphism, de Rham cohomology can be categorified to the derived category of coherent sheaves:
\begin{itemize}
\item Taking the bounded derived category of coherent sheaves defines an $\infty$-functor 
\[ D^b:\bf{Sm}_\C\rightarrow\bf{dgCat}_\C^\op. \] 
\item Periodic cyclic chains defines an $\infty$-functor from the category of small pre-triangulated dg-categories over $\C$ and exact functors into the derived category $D(\C[t^{\pm 1}])$ of complexes of $\C[t^{\pm 1}]$-modules. % \cite{Keller}
\item Ungraded homology defines a functor $D(\C[t^{\pm 1}]) \rightarrow \Mod_{\C[t^{\pm 1}]}$.
\item The Hochschild-Kostant-Rosenberg theorem identifies the composite of these three $\infty$-functors with $H_\rm{dR}[t^{\pm1}]$.\footnote{This version of HKR requires putting together several results from the literature. Briefly, and omitting several intermediate steps: by \cite{keller1999cyclic} cyclic chains of the derived category can be computed by Loday cyclic homology, that is, cyclic chains of the structure sheaf; \cite[Prop.~1.1]{Ben_Zvi_2012} then provides a natural isomorphism of this with de Rham cohomology.}
\end{itemize}
Now, this functor $D^b$ has an additional structure over its complex-level shadow $C_\rm{dR}$: the (derived) pullback and pushforward $\infty$-functors are adjoint to one another. Using this adjunction, one can associate to any morphism of spans
  \[\xymatrix@R=2ex{
    & K \ar[dl]_p \ar[dd]^h \ar[dr]^q \\
    X && Y \\
    & K' \ar[ul]^{p'} \ar[ur]_{q'}
  }\]  
a natural transformation $(q')_*(p')^*  \rightarrow  q_*h^*(p')^* \simeq q_*p^*$ of $\infty$-functors $D^b(X)\rightarrow D^b(Y)$. One would be forgiven for imagining that $D^b$ actually extends to an $(\infty,2)$-functor
  \[ D^b:\Cor(\bf{Sm}_\C,\;\rm{pr|sm}) \rightarrow \bf{dgCat}_\C^{\op_1} \]
into a suitably defined $(\infty,2)$-category of dg-categories, dg-functors, and dg-natural transformations. Since pushforwards are `determined', being adjoints, by pullbacks, one might further imagine this extension to be unique.
\end{eg0}

\paragraph{Bivariant Yoneda philosophy}

A functor from an $(\infty, 1)$-category into an $(\infty, 2)$-category that --- like the inclusion of $D$ into $\Cor_D$ --- takes every map in (a certain marked subset of) the source to a left adjoint is called \emph{bivariant}. (For simplicity, in this introduction we suppress further reference to the marked subset.) Bivariant functors from $D$ into the $(\infty,2)$-category of $(\infty,1)$-categories that satisfy the base change formula themselves form an $(\infty,2)$-category $\Biv_D$. The machine we will construct in this paper is a `Yoneda API' for $\Biv_D$.

To each object $x$ of $D$, we assign a bivariant theory $\Cor_D(x,-)$ of correspondences into $x$: this defines a faithful \emph{Yoneda embedding} $\rm y^\mp:D\rightarrow\Biv_D^\op$ which is itself bivariant. It is not full, but its essential image can be described using the bivariant version of the \emph{Yoneda lemma}:
\begin{thm0}[\ref{BIV_YON_LEMMA}]
Let $F$ be a bivariant functor of $D$. The  bivariant evaluation mapping \[ \rm{ev}:\Biv_D\left[\Cor_D(x,-),F(-)\right]\rightarrow F(x) \] is an equivalence of categories, where $\Biv_D[*,*]$ denotes bivariant natural transformations.
\end{thm0}
\begin{proof}[Sketch of proof.]Considered via the Grothendieck construction as a bi-Cartesian fibration, the representable bivariant functor $\Cor_D(x,-)$ is the free co-Cartesian fibration on the free Cartesian fibration on the singleton $\{x\}\rightarrow D$.
\end{proof}

\noindent It follows that the \emph{bivariant Yoneda image} $\rm y^\mp(D)$ of $D$ --- that is, the full $(\infty,2)$-subcategory of $\Biv_D$ spanned by representable bivariant functors --- is itself a good candidate for an $(\infty,2)$-category of correspondences.

\begin{cor0}The bivariant Yoneda image of $D$ is an $(\infty,2)$-category of correspondences for $D$ in the sense that its objects are the objects of $D$, and for any $x,y:D$, the category of maps between $x$ and $y$ in $\Biv_D$ is the category $\Cor_D(x,y)$ of spans between $x$ and $y$. \hfill\eqref{BIV_LOC_IMAGE}

Moreover, under this identification, the composition map in $\rm y^\mp(D)$ is isomorphic to the composite of correspondences computed by the fibre product.\hfill\eqref{BIV_COMP_PROP}
\end{cor0}

In \S\ref{BIV_EXT}, we apply a kind of `mixed Kan extension' construction that takes in bivariant functors of $D$ and outputs 2-functors of the bivariant Yoneda image of $D$. 

\begin{thm0}[\ref{BIV_EXT_THM}]

Let $K$ be an $(\infty,2)$-category. There is an $(\infty,2)$-functor 
\[ \Spanext:\Biv(D,K) \longrightarrow 2\Fun(\rm y^\mp(D),K) \]
of \emph{span extension} to the Yoneda image. As an $(\infty,1)$-functor, it is a section of the functor of restriction along $D\subseteq \rm y^\mp(D)$.
\end{thm0}
\noindent Finally, in \S\ref{UNIV}, we apply a 2-dimensional Grothendieck construction to put the extension functor of Theorem \ref{BIV_EXT_THM} in families, hence recovering a universal property. The Grothendieck construction we invoke is not quite justified in the literature at time of writing, so these results are best considered as conditional --- see the disclaimer below.
\begin{thm0}[\dag\hskip.5pt) (\ref{UNIV_EXT_THM}]
The inclusion $i:D\rightarrow\Cor_D$ is universal among bivariant functors of $D$. That is, for any $(\infty,2)$-category $K$, restriction along $i$ yields an equivalence of $(\infty,2)$-categories \[ \Biv(D,K)\cong2\Fun(\Cor_D,K). \]
Moreover, this equivalence is a \emph{natural} equivalence of $(\infty,3)$-functors of $K$.
\end{thm0}
\noindent We also obtain a monoidal version:
\begin{thm0}[\dag\hskip.5pt) (\ref{UNIV_MON_THM}]
Let $(D,\otimes)$ be a symmetric monoidal $\infty$-category with pullbacks, and suppose that pullbacks are preserved by $\otimes$. There is a unique symmetric monoidal structure on $\Cor_D$ making the inclusion of $D$ strongly monoidal.

The inclusion $i:(D,\otimes), \rightarrow (\Cor_D,\otimes)$ is universal among symmetric monoidal bivariant functors of $D$. That is, for any symmetric monoidal $(\infty,2)$-category $(K, \otimes)$, restriction along $i$ yields an equivalence of $(\infty,2)$-categories 
\[ \Biv^\otimes(D,K)\cong2\Fun^\otimes(\Cor_D,K). \]
Moreover, this equivalence is a natural equivalence of $(\infty,3)$-functors of the symmetric monoidal 2-category $(K,\otimes)$.
\end{thm0}
\noindent These results provide the machine we were looking for to produce functors out of correspondence categories: to be a functor of $\Cor_D$ is simply a \emph{property} of a functor of $D$. In particular, since categories of correspondences are characterised by a universal property, this machine allows us to bypass the need to make the coherently associative structure of $\Cor_D$ explicit.

\paragraph{Disclaimer}The main result of this paper already appears in a published work as \cite[Thm.\ 8.1.1.9]{GR}. In comparison with \emph{op.~cit}, we offer two advantages:
\begin{itemize}
\item Conceptual: the present paper avoids explicit manipulation of simplices in the Segal space model. Instead, we develop much of what might be termed the `internal logic' of $(\infty,2)$-category theory and couch our arguments in this language. The proof of the extension theorem fits into a larger context of `Yoneda type operations' in higher category theory. This, I believe, makes it both more transparent and more amenable to further generalisation. 

Beyond its interest in its own right, the methods used to obtain Theorem \ref{UNIV_EXT_THM} can also serve as a guide to the future development of $(\infty,2)$ (and higher) category theory.

\item Unconditional: All of the results of \emph{op.~cit}.\ are conditioned on a list of `unproved statements' \cite[(10.0.4.2)]{GR}.\footnote{More precisely, the authors do not distinguish between results that are conditional on this list and those that are independent. From my own inspection it seems to me that even the most basic foundations of $(\infty,2)$-category theory enumerated in \cite[\S10]{GR} are so conditioned.} Conversely, at least the first half of our results, that is, the extension theorem \ref{BIV_EXT_THM} and most of the example constructions of \S\ref{EX}, are established unconditionally.
\end{itemize}
The universal property of correspondences and its symmetric monoidal variant is conditional on a restricted form of $(\infty,2)$-categorical Grothendieck construction for functors from an $(\infty,1)$-category into $2\Cat$. Results that are conditioned in this way are flagged with a (\dag) (for example, Theorem \ref{UNIV_EXT_THM} quoted above). 
For remarks on the epistemological status of the $(\infty,2)$-categorical Grothendieck construction, see \ref{UNIV_GROT_EPIST}.

\paragraph{Acknowledgements}I thank Dennis Gaitsgory, whose questions about this work drove me to think more carefully about the theoretical background upon which it is based. Thanks are also due to David Gepner, Aaron Mazel-Gee, and David Carchedi for helping me with the literature and sharing their perspectives on higher categories; I thank Rune Haugseng for drawing my attention to the work of Hinich on the enriched Yoneda lemma and for many educational discussions.

\section{Higher category theory}\label{CAT}

The purpose of this section is to assemble the basic properties and construction that we need from $(\infty,n)$-category theory before leaping into the study of bivariant functors, collating references to the literature and plugging some gaps.

\subsection{Homotopy invariance in higher category theory}
\label{CAT_INV}

The language used in this paper is, so far as possible, \emph{homotopy invariant}. This means that every statement --- including every hypothesis and each intermediate step in every proof --- is a construction of an $\infty$-functor of $\infty$-categories. We will usually try to make such statements `as functorial as possible' so that they not only respect the notion of equivalence in the domain of discourse --- hence `homotopy invariant' --- but also some notion of morphism.

\begin{para}[Homotopy invariant definitions and constructions]
\label{CAT_INV_DEF}
A \emph{homotopy-invariant definition} is a specification of a set of connected components of some simplicial set, called the domain of discourse, provided by the context. A \emph{homotopy-invariant construction} is a homotopy-invariant definition of a contractible simplicial set, i.e.~of a homotopically unique vertex of the domain of discourse. We use the English definite article to refer to any vertex of a contractible domain of discourse, so that, for example, in the homotopy theory of Kan complexes, ``the one-point space'' actually means ``a contractible Kan complex''.

The specification itself, including the domain of discourse, should ideally not invoke any notions which are not themselves homotopy-invariant. However, in the foundations we are using, at some fundamental level we need to invoke an explicit model to be able to write anything down. Of course, it does not matter which specific models we pick. 

For concreteness, let us use as a basic domain of discourse Lurie's quasi-category $\Cat_\infty$ \cite[Def.~3.0.0.1]{HTT}, for the category of $\infty$-categories, and Lurie's construction of simplicial mapping sets in $\Cat_\infty$ for spaces of $\infty$-functors \cite[\S2.2]{HTT}. We discuss the associative composition law for $\infty$-categories in \S\ref{CAT_TENS}. We are free to replace these simplicial sets with any equipped with an equivalence with Lurie's models, and this will sometimes be necessary.

In the sequel, we will drop the prefix $\infty$ and refer to the objects (vertices) of $\Cat_\infty$, which we rechristen $1\Cat$, as 1-categories. Functor categories are denoted $1\Fun(-,-)$. 
\end{para}

\begin{para}[Using \cite{HTT}]
\label{CAT_INV_HTT}
Our requirement that every sentence be homotopy invariant necessitates that we diverge slightly from the approach of \cite{HTT}. To wit, some of our definitions --- notably, that of Cartesian fibration --- differ slightly from Lurie's treatment in that ours relax all non homotopy invariant conditions --- in the case in point, to be an inner fibration.

Despite this discrepancy, we freely invoke the results and constructions of \cite{HTT}, even if the methods used are \emph{not} homotopy-invariant (as is often the case), provided that the output can be given a homotopy-invariant characterisation. 
\end{para}

\begin{remark}
The constructions of \cite{HTT} often produce simplicial sets, and not merely contractible spaces of simplicial sets. Therefore, many of the constructions of this paper, if interpreted using these rigidified constructions, themselves output on-the-nose quasi-categories. However, this does not apply to constructions that invoke Kan extensions, which are only ever defined up to a contractible choice.
\end{remark}

\begin{para}[Formulas in higher category theory]
\label{CAT_TYPE}
%This paper employs some notational conventions adapted from type theory, about which what little I know I learned from the IAS volume \cite{HoTT} (especially appendix A.2). However, I make no claims that the paper is founded on a type theory. 

Some of the arguments of this paper are expressed by manipulating formulas, also known as \emph{anonymous functions}, using the $\return$ symbol. This practice is typical in other fields of mathematics but rare in homotopy theory, essentially because a map of spaces or $\infty$-categories is not at all determined by its action on objects, which is all that this notation usually makes explicit. Nonetheless, with enough care it can be made rigorous in our setting as well. The basic format is explained below.

This notation is convenient when dealing with multivariate expressions, particularly when they express different bivariance behaviour in each factor and when we need to pass through multiple hom-tensor adjunctions. This is the situation encountered in \S\ref{BIV_EXT}. 

%We also make explicit the level of functoriality of each formula by subscripting the category in which it is a mapping. Homotopy invariance requires that minimum level of functoriality is as a map of spaces.

\begin{labelitems}
\item[Object declaration] Objects of categories (either variables or constants) are declared with a \emph{colon}, i.e.\ $a:A$ states that $a$ is an object of $A$ (hence is equivalent to the usual notation $a\in A$). As usual, we may also write $f:A\rightarrow B$ or $g:C\cong D$ to declare a function or isomorphism; the category in which this morphism or isomorphism lives, if unclear, can be underset.

\item[Functor declaration] We will make definitions via formulas of the form 
\begin{align*} 
\alignhead{Generic $n$-functor}
 a_0:A_0,\ldots,\ a_k:A_k\quad \return_{n\Cat} \quad \ b:B
\end{align*}
 where $n:\N$. Such a formula declares $n$-categories $A_i$, $B$, and an $n$-functor $\prod_0^kA_i \rightarrow B$ (the particular value of $n$ specified by the underset to $\mapsto$); hence, $b$ must be an $n$-functorial formula for an object of $B$ in terms of the free variables $a_i$. The minimum level of functoriality for a formula is as a functor of spaces. (We explain the definitions of $n$-category and $n$-functor for $n>1$ below in \S\ref{CAT_MODELS}.)
 
If the domain types $A_i$ are defined \emph{a priori} as $m$-categories, where $m>n$, then an expression in the form \emph{Generic $n$-functor} should be interpreted as defining a function on the $n$-core \eqref{CAT_FURTHER_CORE}.

\item[Composition of functors] 
  Composites of functors expressed in this way can be written using the usual method of variable substitution.

\item[Natural transformation declaration] 
When an arrow or isomorphism symbol appears to the right of $\mapsto$, it is interpreted as a natural transformation of the variables to the left. Thus we write:
\begin{align*}
\alignhead{Generic $n$-natural transformation}
  a:A\quad & \return_{n\Cat} \quad f:B\underset{D}{\rightarrow} C \\
\alignhead{Generic $n$-natural isomorphism}
  a:A\quad & \displaystyle\return_{n\Cat} \quad f:B\underset{D}{\cong} C
\end{align*}
for a natural transformation of functors $B(a)\rightarrow C(a)$ from $A$ to $D$.

When the target category is $m\Cat$, we can make the natural transformation itself anonymous by nesting formulas, as in:
\begin{align*} 
a:A\quad  \return_{n\Cat} \quad (b:B\quad \return_{m\Cat} \quad c:C) 
\end{align*}
(brackets, added for clarification here, may be omitted). This notation can be converted into the format \emph{Generic $n$-natural transformation} by naming the bracketed function: $f(b)=c$.

\item[Currying] Passing through the  Cartesian closure equivalence $n\Fun(A\times B,C)\cong n\Fun(A,n\Fun(B,C))$ is called `currying'. It gives us the inference rules
\begin{align*}
  \quad a:A,\ b:B \quad  
    & \mapsto\quad c:C \\
  \Rightarrow\hspace{6em} b:B\quad
    & \return\quad f:A\rightarrow C. 
\intertext{in the named format, or}
  \quad a:A,\ b:B \quad  
    & \return\quad c:C \\
  \Rightarrow\hspace{6em} b:B\quad
    & \return\quad a:A \quad\return\quad  c:C. 
\end{align*}
in the anonymous format.
\end{labelitems}
\end{para}

\subsection{Models of $n$-category theory}\label{CAT_MODELS}

The underlying philosophy of this paper is to use only `model-independent' constructions. However, since we do not have to hand a full definition of the syntax of $(\infty,n)$-category theory, we are forced at some point to found our constructions on models. In this section we discuss and fix equivalences between some known models for $(\infty,n)$-category theory so that when necessary we may freely switch between them without creating ambiguity.

%\begin{remark}
%\label{CAT_MODELS_TYPE}
%There is some hope that the intensional type theory discussed in \cite{HoTT, Riehl_Shulman} can provide the axiomatisation needed to make sense of `model-independence'. As it stands, that theory is not ready-to-use because it is not known to be equivalent to the ZFC set theory on which \cite{HTT}, and consequently this paper, is founded. However, I have no doubt that every statement in this paper is valid without modification in a suitable type theory.
%\end{remark}

\begin{para}[Approaches to $n$-categories]\label{CAT_MODELS_LIST}There are several equivalent definitions of the homotopy theory of $(\infty,n)$-categories (henceforth: $n$-categories).  
\begin{itemize}\item The original approach of C.\ Barwick \cite{Barwick_thesis,GR} defines them to be complete $n$-fold Segal spaces.
\item As localisations of the category of presheaves on some predefined category of generators \cite{Rezk, BS-P}.
\item Inductively as enriched categories \cite{Simpson,Gepner_Haugseng,Hin}.\end{itemize}
Each of the above works define a $1$-category of $n$-categories.\end{para}

\begin{para}[Uniqueness of the theory of $n$-categories]
\label{CAT_MODELS_UNIQUE}

The results of \cite{BS-P} provide a framework for comparing the $1$-categories produced by the approaches enumerated in \ref{CAT_MODELS_LIST}. It defines a theory of $n$-categories to be a presentable 1-category $n\Cat$ equipped with an embedding of a certain fixed category of generators $\mathbb G_n=\langle C_k\rangle_{k=0}^n$ consisting of `$n$-cells' and satisfying certain additional conditions \cite[Def.\ 7.1]{BS-P}. Let us call the embedding $\mathbb G_n\hookrightarrow n\Cat$ the \emph{orientation}.

The authors then compute \cite[\S4]{BS-P} that
\[\Aut(n\Cat)=\Aut(\mathbb G_n)=(\Z/2\Z)^n\]
with named generators $\op_k$ defined on $\mathbb G_n$ so that they `reverse $k$-cells'. In particular, the space of \emph{oriented} $n$-category theories is contractible.

As a special case, by fixing the embedding of $\langle\Delta^0,\Delta^1\rangle$, the various approaches outlined above applied in the case $n=1$ are identified with the category of $1$-categories described in \cite{HTT} (and on which all the above are based).\end{para}

\begin{remark}[Bergner-Rezk hierarchy]In \cite{BS-P} the category of cells is defined using a prior notion of strict $n$-category. This can be avoided by basing the definitions instead on the Bergner-Rezk hierarchy $\{\Theta_n\}_{n:\N}$ \cite{Berger, Rezk}, a variation that is allowed for in \cite[\S11]{BS-P}. The $(1,1)$-categories $\Theta_n$ are defined inductively via a combinatorial procedure called the categorical wreath product. Their objects can be interpreted as strict $n$-categories obtained as a lattice-like concatenation of $n$-dimensional cells. 

%It is not hard to compute $\Aut(\Theta_n)=\Aut(\mathbb G_n)$ from the arguments of \cite[\S4]{BS-P} and the fact that $\Theta_n$ is generated by $\mathbb G_n$.
\end{remark}

\begin{para}[Orientations of the models]The models listed in \ref{CAT_MODELS_LIST} are oriented as follows:
\begin{itemize}\item An $n$-category theory produced by \cite[Thm.~11.2]{BS-P} comes equipped with an orientatation by (R.4) of \emph{loc.~cit}. 

Thus, Rezk $\Theta$-spaces and $n$-fold complete Segal spaces are oriented by the construction of \cite[\S13, 14]{BS-P};
\item $(n-1)\Cat$-enriched categories are oriented by transfer along the functor from the enriched model to the iterated Segal space model constructed in \cite[\S7]{Rune_rectification}.
\item In the case $n=1$, the original category of quasi-categories is oriented by the tautological inclusion $\{\Delta^0,\Delta^1\}\subset1\Cat$. The standard identification $C \return [[n]\return \Ob(C^{\Delta^n})]$ of $1\Cat$ with the category of complete Segal spaces is compatible with the stated orientations.
\end{itemize}
This fixes equivalences between the models and ensures that any combination thereof must commute.\end{para}

\begin{para}[Where do we use the models?]

Apart from the basic properties of the $1$-category of $n$-categories provided by \cite{BS-P}, we explicitly invoke models for three types of statements:
\begin{itemize}
\item We need the enriched categories model to be able to apply the methods of \cite{Hin}, in particular, his Yoneda lemma.

\item I prove the recognition principle for limit cones in $n\Cat$ in terms of mapping categories using the enriched categories model.

\item I invoke the Segal space model to construct 2-subcategories fitting a specification.

%in two specific locations: to construct 2-subcategories fitting a specification, and to prove a `general' Yoneda lemma in \S\ref{CAT_YON}.
\end{itemize}
\end{para}

\subsection{Mapping objects in higher categories}
\label{CAT_FURTHER}
\label{CAT_ENRICH}

Category theory comprises more than just the concept of category and functor; consequently there is more to it than can be extracted directly from the uniqueness theorem \ref{CAT_MODELS_UNIQUE}. In this section we gather some basic facts and constructions that can be defined using the notion of mapping object between two objects of an $n$-category.

%I have left the arguments here in a fairly terse form because, firstly, they are elementary, and secondly, because a full development will appear in a separate paper \cite{Mac_enrich}.

\begin{para}[Enrichment]

One of the main things about $n$-categories is that they are enriched in $(n-1)$-categories. That is, each $n$-category $K$ has an underlying space $\Ob(K)$, and for each pair of objects $x,y:K$ there is a \emph{mapping $(n-1)$-category} $K(x,y):(n-1)\Cat$. These $(n-1)$-categories are collectively equipped with an associative composition law. The precise meaning of `associative composition law' is captured by the notion of enriched $\infty$-category \cite{Gepner_Haugseng,Hin}. For the purposes of this paper, however, we will not need to unpack it any further.
\end{para}

\begin{remark}[A look ahead]

There is also a more general notion of enrichment, defined in \cite{Hin}, where $K(-,-)$ is a bifunctor on an underlying 1-category, rather than just a space. This will be important when we discuss enrichments via tensoring or powering in \S\ref{CAT_TENS}. We also discuss there the comparison, in the case $n=1$, of this mapping bifunctor with the fundamental one associated to the 1-categorical Yoneda embedding \cite[\S5.1.3]{HTT}.
\end{remark}

\begin{para}[Specification of subcategories]
\label{CAT_FURTHER_SPECDEF}

A basic way of constructing new $n$-categories is to \emph{specify} them as subcategories of a given $n$-category by specifying the objects, morphisms, and $k$-cells for $k$ up to $n$. I will make this precise in the case $n=2$. A \emph{specification} $S$ on a 2-category $K$ comprises the following data:
\begin{enumerate}[label=\arabic*), start=0]
\item A subset $S_0\subseteq\pi_0\Ob(K)$.
\item A subset $S_1\subseteq\pi_0\Ob(K^{\Delta^1})$, all of whose members have source and target in $S_0$, and closed under compostion.
\item A subset $S_2\subseteq\pi_0\Ob(K^{c_2})$ of the set of isomorphism classes of 2-cells in $K$, all of whose members have source and target in $S_1$, and closed under composition (both of 2-cells and with 1-cells in $S_1$).
\end{enumerate}
A 2-functor $F:K^\prime\rightarrow K$ is said to \emph{fit} the specification $S$ if it maps all objects, morphisms, and 2-cells of $K^\prime$ into $S_0,\,S_1,\,S_2$. If $F$ is monic --- that is, induces monomorphisms on spaces of objects, morphisms, and 2-cells --- say that $F$ \emph{exactly} fits the specification if  it identifies the isomorphism classes of objects, morphisms, and 2-cells of $K^\prime$ with $S_0,\,S_1,\,S_2$.
\end{para}

\begin{lemma}[Existence of specified 2-subcategories]
\label{CAT_FURTHER_SPEC}

Let $K:I\rightarrow 2\Cat$ be a 1-functor and, for each $i:I$, let $S_i$ be a specification on $K_i$. Suppose that for each map $\phi:i\rightarrow j$ in $I$, $\phi:K_i\rightarrow K_j$ maps $S_i$ into $S_j$. Then there is a unique subfunctor $K^\prime$ of $K$ such that for each $i:I$, $K_i^\prime\subseteq K_i$ exactly fits the specification $S_i$. 

Moreover, a natural transformation $L\rightarrow K$ of functors $I\rightarrow 2\Cat$ factors through $K^\prime$ (necessarily uniquely) if and only if it fits the specification $S_i$.
\end{lemma}
\begin{proof}
Via the Segal space model, interpret $K$ as a bifunctor $\tilde K:I\times\Delta\times\Delta\rightarrow\bf{Spc}$. The required subcategories can be constructed as 1-subcategories of the Grothendieck integral $\int_{I\times\Delta\times\Delta}\tilde K$.
\end{proof}

\begin{para}[Core]
\label{CAT_FURTHER_CORE}
The inclusion $m\Cat\hookrightarrow n\Cat$ for $m\leq n$ has a right adjoint called the $m$-\emph{core}. The $m$-core of an $n$-category $K$ is the subcategory with the specification
\begin{itemize}
\item $S_k = $ all $k$-cells of $K$, $0\leq k\leq m$;
\item $S_k=$ invertible $k$-cells on members of $S_{k-1}$, $m<k$.
\end{itemize}
That this indeed defines the right adjoint can be seen from the mapping property of subcategories defined by specification, Lemma \ref{CAT_FURTHER_SPEC}. Where we cannot avoid an express notation, we write $\lie c_mK$ for the $m$-core of $K$. Here is an example of a situation where it can be avoided: if $K,~K^\prime$ are $n$-categories, we say an $m$-\emph{functor $K\rightarrow K^\prime$} as an abbreviation for an $m$-\emph{functor $\lie c_mK\rightarrow K^\prime$}.

The 0-core of a $n$-category $K$ is the same as its space $\Ob(K)$ of objects. When $n=1$ and we regard $K$ as a quasi-category, this 0-core is modelled by the largest Kan complex contained in $K$: cf.~\cite[Prop.~1.2.5.3]{HTT} for the appropriate mapping property.
\end{para}

\begin{lemma}[Recognition of limit cones]
\label{CAT_FURTHER_LIMIT}

A diagram $C:I^\triangleleft\rightarrow n\Cat$ of $n$-categories is a limit cone if and only if:
\begin{enumerate}
\item $\Ob(C):I^\triangleleft \rightarrow \bf{Spc}$ is a limit cone of spaces;
\item for each $x,y:C_0$ \emph{(where $0$ is the cone point)}, $C(x,y):I^\triangleleft \rightarrow (n-1)\Cat$ is a limit cone.
\end{enumerate}
\end{lemma}
\begin{proof}
Clear from the description of $n\Cat$ as a (limit-closed subcategory of a) Cartesian fibration over $\bf{Spc}$ with fibre the category of $(n-1)\Cat$-enrichments of that space.
\end{proof}

\subsection{Enrichment from tensoring}
\label{CAT_TENS}

The work \cite[\S6]{Hin} provides a method to upgrade $1$-functors into a $\sh V$-tensored category to $\sh V$-enriched functors. In this section, we review the construction and compare it to the structures already described. I freely employ terminology from \emph{op.~cit}.. At the end of the section, we discuss a construction that uses a powering rather than a tensoring.

\begin{remark}[Enrichments on a space versus enrichments on a 1-category]
\label{CAT_TENS_EMBED}

What we are calling the `usual' embedding of the category $\sh V\Cat$ of $\sh V$-enriched categories into the category of $\sh V$-precategories, constructed in \cite{Gepner_Haugseng}, founds every $\sh V$-category on an underlying \emph{space} of objects. In this section, however, we construct functors --- including the Yoneda functor --- into a $\sh V$-category based on a 1-category of objects. The reader may worry that our constructions do not, then, actually end up inside $\sh V\Cat$.

For the sake of this paper, we proceed as follows: any functors constructed using the method of this section should in the end be restricted to the enrichment of the underlying space of objects, where the theory of \cite{Gepner_Haugseng} applies. This replacement preserves the notion of $\sh V$-fully faithful functor, so the $\sh V$-Yoneda lemma continues to hold when $\sh V\PSh(-)$ is considered as a $\sh V$-enriched category based on a space.

%(There is an alternative embedding of $\sh V\Cat$ into $\sh V$-precategories which gives each $\sh V$-category an underlying 1-category and has the constructions of this section in its domain, but a proof of this fact is beyond the scope of the present paper.)
\end{remark}

\begin{para}[Adjoint enrichment]
\label{CAT_TENS_ENRICH}

Let $(\sh V,\otimes)$ be a presentable monoidal category, $C$ a $\sh V$-enriched precategory (possibly based on a 1-category), and $D$ a left $\sh V$-module. Hinich defines $\Fun_\sh V(C,D)$ to be the 1-category of $C(-,-)$-modules in $1\Fun(C,D)$ with respect to a certain tensoring of the latter over a monoidal category whose underlying 1-category is $1\Fun(C^\op\times C,\sh V)$ \cite[(6.1.3)]{Hin}. The category $\Fun_\sh V(-,-)$ is contravariant in the $\sh V$-category $C$ and covariant in the left $\sh V$-module $D$ \cite[(6.1.4)]{Hin}.

Suppose that for each $x,y:D$ the functor $D(-\otimes x,y)$ is a representable presheaf on $\sh V$. Then \cite[Prop.~6.3.1]{Hin}) tells us that the endomorphism algebra of $\iden_D$ (again with respect to the action $1\Fun(D^\op\times D,\sh V)\curvearrowright 1\Fun(D, D)$) is representable. We view $D$ as a $\sh V$-enriched precategory by equipping it with this algebra; it is called the \emph{adjoint enrichment}.

Thus, the fibre over  of the forgetful functor $\Fun_\sh V(C,D)\rightarrow 1\Fun(C,D)$ over a fixed $f:1\Fun(C,D)$ is the space of algebra homomorphisms $C(-,-)\rightarrow\End(f)$. This identifies Hinich's $\Fun_\sh V(C,D)$ with $\sh V\Fun(C,D)$, where $D$ carries the adjoint enrichment.
\end{para}

\begin{para}[Hom-tensor adjunction]
\label{CAT_TENS_HOM}

We must check that an $n$-category of $n$-functors obtained by the method of \ref{CAT_TENS_ENRICH} is the same as the one provided by Cartesian closure of $n\Cat$.

We have an external Hom-tensor adjunction for $\sh V$-categories in the form of \cite[Cor.~6.1.8]{Hin}, which states:
\[ \Fun_{\sh V\otimes\sh V^\prime}(C\boxtimes C^\prime,D) \cong \Fun_\sh V(C,\Fun_{\sh V^\prime}(C^\prime,D)) \]
where $C':\sh V^\prime\Cat$ and $D$ is tensored over $\sh V\otimes \sh V'$, and the functor category $\Fun_{\sh V'}(C',D)$ is a $\sh V$-module via its action on $D$, $\sh V\otimes\sh V^\prime$ is computed using the tensor product of presentable categories, and $C\boxtimes C^\prime$ is a $\sh V\otimes\sh V^\prime$-enrichment of $C\times C^\prime$.

When $\sh V=\sh V^\prime$ is symmetric monoidal, and the left and right $\sh V$-module structures on $D$ agree, we can compose the enrichment of $C\times C^\prime$ with the monoidal functor $\sh V\otimes\sh V\rightarrow\sh V$ to obtain an internal Hom-tensor adjunction
\[ \Fun_\sh V(C\otimes C^\prime,D) \cong \Fun_\sh V(C,\Fun_\sh V(C^\prime,D)). \]
By \ref{CAT_TENS_ENRICH}, if $D$ admits the adjoint enrichment over $\sh V\otimes\sh V'$, then the preceding statements also hold with $\Fun_*$ (functors from an enriched category into a module) replaced by $*\Fun$ (enriched functors), where $*\in\{\sh V,\sh V',\sh V\otimes \sh V'\}$.

Specialising to the case $\sh V=(n-1)\Cat$ --- so that $\otimes=\times$ is just Cartesian product --- this shows that the adjoint enrichment of $n\Fun(-,-)$ satisfies the universal property of an exponential.
\end{para}

\begin{para}[Enriched mapping bifunctor]
\label{CAT_TENS_MAP}

Hinich's mapping $\sh V\otimes\sh V$-bifunctor for $C:\sh V\Cat$ is defined, within the framework of \ref{CAT_TENS_ENRICH}, to be the tautological bimodule structure on $C(-,-):1\Fun(C^\op \times C,\sh V)$ over itself (with respect to the $1\Fun(C^\op\times C,\sh V)$-bitensoring on itself) \cite[(6.2.4)]{Hin}. The $\sh V$-enriched Yoneda embedding is obtained by passing this through the Hom-tensor adjunction \ref{CAT_TENS_HOM}. It is $\sh V$-fully faithful \cite[Cor.\ 6.2.7]{Hin}.

In the symmetric monoidal case, we can push this again through $\sh V\otimes\sh V\rightarrow\sh V$ to obtain a $\sh V$-enriched functor, so that, in the special case $\sh V=(n-1)\Cat$ we obtain a formula:
\begin{align*}
\alignhead{Mapping object, $n$-natural}
x:C^\op,\ y:C \quad \return_{n\Cat} \quad C(x,y):(n-1)\Cat 
\end{align*}
which enhances the 1-functorial mapping object given as part of the data of enrichment. The $n$-Yoneda embedding then has the form
\begin{align*}
\alignhead{$n$-Yoneda embedding}
  y:C \quad \return_{n\Cat} \quad C(-,y):(n-1)\PSh(C) 
\end{align*}
where $(n-1)\PSh(C)$ denotes the $n$-category of $n$-functors $C\rightarrow(n-1)\Cat$.

I record here the particular application of this result that we will invoke in \S\ref{BIV_EXT}:
\begin{lemma0}[2-Yoneda lemma]
\label{CAT_TENS_YON}

The 2-Yoneda embedding is 2-fully faithful.
\end{lemma0}
\end{para}

\begin{para}[Action of $n$-functors on mapping objects]
\label{CAT_TENS_YON_NAT}
The $\sh V$-enriched mapping objects are intertwined by $\sh V$-functors $F:C\rightarrow D$. Indeed, $F$ comprises the data of a functor $\Ob(F):\Ob(C)\rightarrow\Ob(D)$ and a homomorphism $C(-,-)\rightarrow\Ob(F)^*D(-,-)$ of algebras in $1\Fun(\Ob(C)\times \Ob(C),\sh V)$. The corresponding map of $C(-,-)$-$C(-,-)$-bimodules gives a $\sh V$-natural transformation
\[ x:C^\op,\ y:C \quad\return_{\sh V\Cat}\quad C(x,y) \underset{\sh V}{\rightarrow} D(Fx,Fy) \]
that enhances the natural transformation of functions of $\Ob(C)\times\Ob(C)$ that comprises the underlying map of $\sh V$-quivers. In particular, for $\sh V=(n-1)\Cat$, 
\[ x:C^\op,\ y:C \quad\return_{n\Cat}\quad C(x,y) \underset{(n-1)\Cat}{\rightarrow} D(Fx,Fy) \]
This transformation is invoked in \ref{BIV_EXT_LOC}. 
\end{para}

\begin{remark}[1-categories as $\Spc$-enriched categories]
\label{CAT_TENS_YON_COMPARE}

For the special case $\sh V=\Spc$, \cite{Rune_rectification} gives us an identification $\Phi:1\Cat\tilde\rightarrow\Spc\Cat$. 
We will use this throughout to regard 1-categories as enriched categories.

In particular, the construction \ref{CAT_TENS_MAP} gives us a canonical mapping bifunctor in $\Spc\Cat(C\times C^\op, \Spc)\cong 1\Fun(C\times C^\op,\Spc)$.
All references to the mapping space bifunctor for 1-categories, the 1-Yoneda embedding, and the twisted arrow category (\S\ref{CAT_EVAL}) should be taken from this construction. (We do not address the comparison of this bifunctor with other approaches, such as \cite[316]{HTT}.)

\end{remark}

\subsection{Bootstrapping $n$-natural isomorphisms}
\label{CAT_BOOT}

Here we prove some results that reduce the task of showing certain $n$-natural transformations are invertible to something that can be checked `pointwise' on spaces of objects.

%Lemma \ref{CAT_BOOT_LEMMA} is used in \S\ref{UNIV} to upgrade the 1-universal property of correspondences --- an equivalence of mapping spaces --- to a 3-universal property giving an equivalence of 2-categories.

\begin{lemma}[Extensionality of $n$-natural transformations]
\label{CAT_TENS_EXT}

Let $\eta:F\rightarrow G$ be an $n$-natural transformation of $n$-functors. Suppose that for each $X$ in the domain the induced map $\eta_X:FX\rightarrow GX$ is invertible. Then $\eta$ is an $n$-natural isomorphism.
\end{lemma}
\begin{proof}
By the $\sh V$-Yoneda lemma it is enough to prove this for presheaves. That is, we must establish that for $D:\sh V\Cat$ the forgetful functor $\sh V\Fun(D^\op,\sh V)\rightarrow 1\Fun(D^\op,\sh V)$ is conservative. But $\sh V$-enriched presheaves on $D$ are nothing but modules in $1\Fun(D^\op,\sh V)$ for the monad induced by the enrichment, so this follows from monadicity.
\end{proof}

\begin{lemma}[Bootstrapping natural isomorphisms]
\label{CAT_BOOT_LEMMA}

Let $k:\N$, $I:k\Cat$,  $F_0,F_1:I\rightrightarrows n\Cat$ two $k$-functors, and $\phi:F_0\rightarrow F_1$ a natural transformation. Let the 1-core of $I$ be powered over $n\Cat$, and let the 1-cores of $F_0$ and $F_1$ be compatible with this powering and the tautological self-powering of $n\Cat$.  

Suppose that for each $i:I$, $\phi_i$ induces an equivalence $\Ob(F_0(i))\rightarrow\Ob(F_1(i))$. Then $\phi$ is a $k$-natural equivalence.\end{lemma}

\begin{remark}
Note that we don't ask for any compatibility between the powering of $I$ and its enrichment.
\end{remark}

\begin{proof}
By compatibility with powering, the essential image of
\[ I\rightarrow 0\PSh(n\Cat),\quad i \return [J\return \Ob(F_0(i^J))]\]
consists of representable presheaves (respresented by $F(i)^J$). The hypothesis plus Yoneda gives us that $\phi$ induces an $n$-equivalence $F_0(i)\rightarrow F_1(i)$ for each $i:I$. Now use extensionality of $k$-natural transformations, Lemma \ref{CAT_TENS_EXT}.
\end{proof}

\begin{para}[Mapping $n$-categories]
\label{CAT_BOOT_EX}

An $n$-category $K:n\Cat$ is said to admit an \emph{$n$-powering} over $(n-1)\Cat$ if the contravariant 1-functor
\[
  K^\op \rightarrow (n-1)\Cat, \quad x \return K(x,y)^I
\]
is representable for all $y:K$ and $I:(n-1)\Cat$. 
Then for $\psi:x\rightarrow y$ a morphism in $K$, the induced (by the $n$-Yoneda embedding) $n$-natural transformation $K(y,-)\rightarrow K(x,-)$ of $n$-functors $K\rightarrow (n-1)\Cat$ is compatible with powering over $(n-1)\Cat$, and we are in the situation of Lemma \ref{CAT_BOOT_LEMMA}.

For example, this setup holds for the case $K=(n-1)\Cat$ via the following chain of isomorphisms --- 1-natural in $X$ --- in $0\PSh(n\Cat)$:
\begin{align*}
n\Fun(-,n\Fun(X,Y^I)) &\cong n\Fun(- \times X, Y^I) \cong n\Fun(- \times X \times I, Y)  \\
&\cong n\Fun(- \times I, n\Fun(X, Y)) \cong n\Fun(-, n\Fun(X, Y)^I). 
\end{align*}
\end{para}

\begin{lemma}[Bootstrapping natural isomorphisms for subrepresentable functors]
\label{CAT_BOOT_SUBREP}

  Suppose $K:n\Cat$ admits an $n$-powering over $(n-1)\Cat$.
  Let $\phi:x\rightarrow y$ be a morphism in $K$, and let $F_1$, resp.~$F_0$ be a $k$-subfunctor ($k\leq n$) of $K(x,-)$, resp.~$K(y,-)$, preserved by precomposition with $\phi$.
  
  If $\phi$ induces an equivalence $\Ob(F_0(z)) \tilde\rightarrow \Ob(F_1(z))$ for all $z:K$, then $\phi^*:F_0\rightarrow F_1$ is a $k$-natural equivalence of $(n-1)\Cat$-valued functors.
  
\end{lemma}
\begin{proof}

Apply Lemma \ref{CAT_BOOT_LEMMA} with $I=K$.
\end{proof}

\begin{remark}

Note that this statement is trivial in the case $F_1=K(x,-)$ and $F_0=K(y,-)$, since then by the 1-Yoneda lemma $x\cong y$ and so the conclusion is immediate.

\end{remark}

\begin{eg}[Categorical epimorphism]

A morphism $\psi:x\rightarrow y$ in $K$ is said to be a \emph{$k$-categorical epimorphism} if $\lie c_{k-1}(K(y,-))\rightarrow\lie c_{k-1}(K(x,-))$ is a monomorphism in $(k-1)\Cat$. Suppose $K$ satisfies the hypotheses of \ref{CAT_BOOT_EX} and that $\psi$ is a 1-categorical epimorphism. Then applying Lemma \ref{CAT_BOOT_LEMMA} to the map from $K(y,-)$ to its image in $K(x,-)$ tells us that $\psi$ is also a $k$-categorical epimorphism. This is essentially the setting of Theorems \ref{UNIV_EXT_THM} and \ref{UNIV_MON_THM}.
\end{eg}

\subsection{Fibrations}
\label{CAT_FIB}

For a \emph{marking} $S\subseteq\Ob(D^{\Delta^1})$ of a 1-category $D$, there is an accompanying theory of fibrations admitting Cartesian lifts of members of $S$. In this section, we discuss the basic objects of this theory and construct free $S$-Cartesian fibrations. Just as the free Cartesian fibration on a functor $E\rightarrow D$ is given by a slice category $D\downarrow(-)$ --- see \cite{GHN} ---  the full subcategory $D\downarrow^\sharp(-)$ \eqref{CAT_FIB_SLICE} thereof spanned by the \emph{marked} arrows defines a free marked $S_D$-Cartesian fibration.
The reader may recognise some nomenclature and notation from \cite[\S3.1]{HTT}.

\begin{defn}[Marking]
\label{CAT_FIB_MARK}
\label{CAT_MARK_DEF}

A \emph{marking} of a 1-category $D$ is a collection of morphisms $S_D\subseteq \pi_0(D^{\Delta^1})$, stable under composition (including containing all identities). Equivalently, we may formulate $S_D$ as a subcategory of $D$ whose inclusion induces an equivalence on the space of objects. In this way, a marking is just a type of 1-functor. 

We abuse the same letter to denote both the class of arrows and the subcategory. Hence, the mildly abusive notation $f\in S_D$ means that $f$ is a morphism that belongs to the subcategory $S$; since $\rm{Ob}(S_D)=\rm{Ob}(D)$, this will not cause confusion. Say the pair $(D,S_D)$ is a \emph{marked 1-category}.

Define $1\Cat^+$ as the full 2-subcategory of $1\Cat^{\Delta^1}$ spanned by the marked 1-categories. The morphisms in $1\Cat^+$ are called \emph{marked functors}. I usually abbreviate an object $(D,S_D)$ of $1\Cat^+$ to just $D$, understanding that the marking will be denoted by an $S$ subscripted with the name of the object.
\end{defn}

\begin{para}[Canonical markings]
\label{CAT_MARK_TRIV}

Every 1-category can be canonically thought of as a marked category in two ways, via the \emph{trivial} marking consisting of only isomorphisms, and the maximal marking with all morphisms marked. These constitute a right and left adjoint, respectively, to the forgetful functor.
\[\xymatrix{  1\Cat \ar@/^/[r]^\sharp \ar@/_/[r]_\flat & 1\Cat^+ \ar[l] }\]
By default, equip all categories with the trivial marking. We only need to invoke the maximal marking in the case $\Delta^1$; the marked category $(\Delta^1,\Delta^1)$ is denoted $\Delta^\sharp$.
\end{para}

\begin{comment}
\begin{remark}[Enrichment of the category of marked categories]
\label{CAT_MARK_ENRICH}

The given 2-category structure on $1\Cat^+$ is adjoint to the tensoring over $1\Cat$ induced by the trivial marking \ref{CAT_MARK_TRIV}; this can be proved by the same method as used in \ref{CAT_FIB_CART_2CAT}.
\end{remark}
\end{comment}

\begin{eg}[Restricted and unrestricted arrow category]
\label{CAT_MARK_ARROWS}

Recall that $\Delta^\sharp$ is the 1-simplex category with all morphisms marked. The category $D^{\Delta^\sharp}=1\Fun^+(\Delta^\sharp, D)$ is the full subcategory of $D^{\Delta^1}$ spanned by the elements of $S_D$. It is itself marked by the set of commuting squares in $S_D$.
\end{eg}

\begin{para}[Comma category]
\label{CAT_FIB_SLICE} 

Let $D:1\Cat^+$ and let $p:E\rightarrow D$ be a 1-functor. The \emph{marked comma category} $D\downarrow^\sharp E$ is defined by the pullback
\[\xymatrix{
D\downarrow^\sharp E\ar[d]\ar[r] & D^{\Delta^\sharp} \ar[d]^{\rm{tar}} \\
E\ar[r]^p & D.
}\]
The identity section $D\rightarrow D^{\Delta^\sharp}$ induces a section $E\rightarrow D\downarrow^\sharp E$ that commutes with projection to $D$.

Thus, an object of $D\downarrow^\sharp E$ is the data of an object $e:E$ together with a map $f:x\rightarrow pe$ belonging to $S_D$, and a morphism in $S_D\downarrow_DE$ is a map $e\rightarrow e^\prime$ in $E$ plus a commuting square
\[\xymatrix{
x\ar[d]\ar[r] & x^\prime\ar[d] \\
pe \ar[r] & pe^\prime
}\]
whose top edge does need not belong to $S_D$. The section takes $e:E$ to $(e, \iden:pe\rightarrow pe)$.

%The strict path category commutes with taking 1-cores. Indeed, this follows from the adjointness relation $1\Fun(\Delta^1,\lie c_1\rm D)=\lie c_12\Fun(\Delta^1,\rm D)$ plus that $\lie c_1$ preserves fibre products. Of course, the construction does not commute with 0-cores.

When $E$ is contractible, mapping to an object $d: D$, I write also $ D\downarrow d$. 
\end{para}

\begin{remark}[Slicing in \cite{HTT}]
\label{CAT_FIB_SLICE_HTT}
In the case of slicing over a single object, our definitions compare to that of the \emph{overcategory} $ D_{/d}$ of \cite[\S1.2.9]{HTT}. More precisely, the latter is defined as the value on $(D,d)$ of the right adjoint from an adjunction between $\bf{sSet}$ and $\mathrm{pt}\downarrow\bf{sSet}$. By \cite[Props.~1.2.9.3 and~4.2.1.3]{HTT}, both adjoints preserve categorical equivalences and hence descend to an adjunction between the localised $\infty$-categories; these are $1\Cat$ and $\mathrm{pt}\downarrow1\Cat$ respectively (in the latter case because we sliced under a cofibrant object).
\end{remark}

\begin{defn}[Marked Cartesian fibrations]
\label{CAT_FIB_CART_DEF}

Let $D:1\Cat^+$. A functor $p:E\rightarrow D$ is said to be $S_D$-\emph{Cartesian} if for each $f\in S_D$ and object $x:E_{\rm{tar}(f)}$ there is a Cartesian lift of $f$ ending at $x$. A morphism in $1\Cat\downarrow D$ between $S_D$-Cartesian fibrations is an $S_D$-\emph{Cartesian transformation} if it preserves the set of arrows Cartesian over $S_D$. 

The 1-subcategory of $1\Cat\downarrow D$ spanned by the $S_D$-Cartesian fibrations and $S_D$-Cartesian transformations is denoted $S\Cart_D$. By pullback of fibrations, it is a contravariant 1-functor of $D:1\Cat^+$. When $S_D=D$, we write also $1\Cart_D$. 
\end{defn}

\begin{remark}[Cartesian fibrations in \cite{HTT}]
\label{CAT_FIB_CART_HTT}

The definitions of $p$-Cartesian arrow and Cartesian fibration provided in \cite{HTT} invoke fibrancy conditions for the Joyal model structure, and as such they are not homotopy-invariant. We adopt here the homotopy-invariant versions from \cite[\S2]{MG_Cart}. These sets of definitions are equivalent when applied to inner fibrations between quasi-categories; see \cite[Thm.~3.3,~3.4]{MG_Cart} or \cite[Cor.\ 4.1.24]{RV_Yon}.
\end{remark}

\begin{para}[Free fibrations]
\label{CAT_FIB_FREEDEF}

An $S$-Cartesian fibration $F\rightarrow D$ is said to be \emph{$n$-free}, where $n:\{1,2\}$, on a functor $E\rightarrow F$ if for each $S$-Cartesian fibration $G$ the restriction
\[  S\Cart_D(F,G)\rightarrow 1\Fun_D(E,G) \]
is an equivalence of $(n-1)$-categories. (We will see shortly that this implies that $E$ is a subcategory of $F$.) By extensionality, it is equally an equivalence of functors $1\Cart_D\rightarrow(n-1)\Cat$. We will need to establish the existence and naturality of free fibrations over a 1-category for when we argue for the bivariant version in \S\ref{BIV_BICART}.
\end{para}

\begin{lemma}[Recognition of $S_D$-Cartesian transformations]
\label{CAT_FIB_LEMMA}

Let $D:1\Cat^+$, $p:E\rightarrow D$, and $F:S\Cart_D$. Let $\phi:D\downarrow^\sharp E\rightarrow F$ be a functor over $D$. The following are equivalent:
\begin{enumerate}
\item $\phi$ is a $p$-right Kan extension of its restriction to $E$ along the diagonal section $E\rightarrow D\downarrow^\sharp E$;
\item $\phi$ is an $S$-Cartesian transformation.
\end{enumerate}
Moreover, any functor $\psi:E\rightarrow F$ admits a relative left Kan extension to $D\downarrow^\sharp E$ whose restriction to $E$ is $\psi$.
\end{lemma}
\begin{proof}
A $p$-right Kan extension of $\phi|_E$ to $D\downarrow^\sharp E$ takes an object $(d\rightarrow pe)$ to the value on the vertex of a $p$-limit cone
\[\xymatrix{
((d\rightarrow pe)\downarrow E)^\triangleleft \ar[dr]_{p^\triangleleft}\ar[rr]^-{\phi_E} && F \ar[dl] \\
& D 
}\] 
where $p^\triangleleft=p\circ\rm{tar}$ on $(d\rightarrow pe)\downarrow E$ and sends the cone point to $d$. The index category of this $p$-limit may be replaced by its initial object $[(d\rightarrow pe)\rightarrow e]$, in which case the notion of $p$-limit coincides with that of Cartesian lift.

It follows that if $F$ is $S$-Cartesian, the $p$-right Kan extension always exists. The last statement holds because $E\rightarrow D\downarrow^\sharp E$ is fully faithful.
\end{proof}

\begin{remark}[Relative limits and Kan extensions]
The notion of $p$-limit we are using in the proof of Lemma \ref{CAT_FIB_LEMMA} is, like in the special case of $p$-Cartesian morphisms, a homotopy-invariant version of the notion defined in \cite[Def.\ 4.3.1.1]{HTT} which applies in the case $F\rightarrow D$ is (modelled by) an inner fibration of simplicial sets. See \ref{CAT_FIB_CART_HTT}. The same applies to that of $p$-Kan extension, \cite[Def.\ 4.3.2.2]{HTT}.
\end{remark}

\begin{prop}[Comma categories are free fibrations]
\label{CAT_FIB_FREE}

The inclusion $E\rightarrow D\downarrow^\sharp E$ exhibits the marked comma category as a 2-free $S_D$-Cartesian fibration.
\end{prop}
\begin{proof}
By Lemma \ref{CAT_FIB_LEMMA} and \cite[Prop.~4.3.2.17]{HTT}, the restriction functor $1\Fun_D(D\downarrow^\sharp E,G)\rightarrow 1\Fun_D(E,G)$ has a fully faithful right adjoint $i_*$ whose image is the category of $S_D$-Cartesian transformations. (Compare \cite[Prop.~4.11]{GHN}.)
\end{proof}

\subsection{Grothendieck construction}
\label{CAT_GROT}

The $\infty$-categorical Grothendieck construction is indispensable in higher category theory. However, its actual form as presented in \cite[\S3.2]{HTT} is rather opaque and difficult to use directly. It will be helpful for us to be quite precise about what the Grothendieck construction actually does --- particularly when we come to hypothesise its 2-categorical version in \S\ref{UNIV}.

\begin{para}[Grothendieck integration]
\label{CAT_GROT_INT}

The Grothendieck construction takes the form of a natural equivalence
\begin{align*}
\alignhead{Grothendieck construction --- 1-natural 2-equivalence (\dag)}
 D:1\Cat^\op \quad\underset{1\Cat}{\return} \quad \textstyle\int_D:1\PSh(D)\underset{2\Cat}{\cong}1\Cart_D:\fibre_D
\end{align*}
I also refer to the operator $\int_D$ as \emph{Grothendieck integration} and to $\fibre_DE$ as the \emph{fibre transport functor} of the fibration $E$ (which is supposed to evoke the analogy with parallel transport of connections). Every property of the Grothendieck construction we use is deduced from this sentence. (In fact, its form determines it uniquely, but we won't need to use this fact.)\end{para}

\begin{para}[Epistemological status of Grothendieck integration]
\label{CAT_GROT_STATUS}

A Grothendieck construction having the form 
\begin{align*}
\alignhead{Grothendieck construction --- 1-natural 1-equivalence}
  D:1\Cat^\op \quad\underset{1\Cat}{\return} \quad \textstyle\int:1\PSh(D)\underset{1\Cat}{\cong}1\Cart_D:\fibre
\end{align*}
whose restriction over $D=\rm{pt}$ is the identity functor of $1\Cat$ is established in \cite[Cor.\ A.31]{GHN}. 
For fixed $D:1\Cat$, we can use Hinich's theory of adjoint enrichments outlined in \ref{CAT_TENS_ENRICH} to upgrade this to a 2-categorical version
\begin{align*}
\alignhead{Grothendieck construction --- 0-natural 2-equivalence}
  D:1\Cat^\op \quad\underset{\Spc}{\return} \quad \textstyle\int:1\PSh(D)\underset{2\Cat}{\cong}1\Cart_D:\fibre,
\end{align*}
where $1\Cart$ is considered as a 2-category via its adjoint enrichment. This comes into play in the proof of the bivariant Yoneda lemma. Since we have not established the full functoriality of adjoint enrichments, this does not quite recover the fully fledged formula \ref{CAT_GROT_INT} (whence the (\dag) flag), but it will suffice for our purposes.
\end{para}

\begin{para}[Basic features of Grothendieck integration]
\label{CAT_GROT_PROPS}

In this paper, we will only need to know the specifics of Grothendieck integration through the medium of the following identifications:
\begin{labelitems}
\item[Objects] By base change to a point, for any $x:\rm D$ we obtain an equivalence $(\int F)_x\cong F(x)$ natural in $F:1\PSh D$. That is, the integral does what we expect on fibres.
\item[Morphisms] For any $f:x\rightarrow y$ in $D$, the functor $f^!:Fy\rightarrow Fx$ takes each $e:Fy$ to the source of a Cartesian arrow $f^e:f^!e\rightarrow e$; this can be seen by naturality applied to the unique map of Cartesian fibrations
\[\xymatrix{
\Delta^1\ar@{=}[d] \ar[r] & \textstyle\int_DF\ar[d] \\
\Delta^1\ar[r] & D
}\]
covering $f$ in the base and sending the target in $\Delta^1$ to $e$.

\item[Free fibration] The integral of the functor represented by $x:D$ has the following property:
\begin{align*}
  1\Cart_D\left(\int_D\rm y_Dx, E\right) \quad & 
    \cong \quad 1\PSh(\rm y_Dx,\fibre_DE)& \text{Grot.~construction.} \\
  & \cong \quad (\fibre_DE)_x & \text{Yoneda lemma} \\
  & \cong \quad E_x & \text{By \emph{Objects}.}
\end{align*}
Expanding the definitions, the composite equivalence is simply evaluation on $\iden_x:\int_D\rm y_Dx$. In other words, $\int_D\rm y_Dx$ is a free Cartesian fibration on $\{x\}$.
\end{labelitems}
\end{para}

\begin{comment}
\begin{para}[Yoneda compatibility]
\label{CAT_GROT_YON}

For $x:\rm D$ the representable fibration $\int_D\rm y^1x$ has fibre $D(w,x)$ over $w$. In particular, there is a tautological element $\iden_x:\int_D\rm y^1x$ in the fibre over $x$. Passing via the bottom of the square
\[\xymatrix{
	\Cart_\rm D(\int_D\rm y^1x, \int_DF) \ar@{=}[r]^-{\rm{ev}_{\iden_x}} 
		&(\int_DF)_x \ar@{=}[d] \\
	\Map(\rm y^1x,F) \ar@{=}[r]^-{\rm{ev}_{\iden_x}} \ar@{=}[u]	& F(x)
 } \]
we discover that evaluation at this element is a natural (in $F:1\PSh(D)$) isomorphism. That is, the image of $x$ under the functor $\int\rm y^1:D\rightarrow 0\Cart_D$ is a free Cartesian fibration on $\{\iden_x\}\subseteq \int\rm y^1x$.
\end{para}
\end{comment}

\subsection{Evaluation map}
\label{CAT_EVAL}

The purpose of this section is to construct a universal evaluation 1-functor
\begin{align*}
\alignhead{Universal evaluation map}
\rm{ev}:\left(\int_{X:1\Cat,\,Y:1\Cat}1\Fun(X,Y)\right) \times_{1\Cat} \left(\int_{X:1\Cat}X\right) \quad \underset{1\Cat}{\longrightarrow}\quad \int_{Y:1\Cat}Y
\end{align*}
which, for fixed $X$, recovers the evaluation map $\rm{ev}_X:1\Fun(X,-)\times X\rightarrow (-)$ associated to the Cartesian closed structure on $1\Cat$.
The subtlety in making this natural in $X$ is that the domain is a product of a Cartesian with a co-Cartesian fibration, so this map cannot quite be realised as a natural transformation of functors into $1\Cat$. The domain is, however, exponentiable in $X$, so we may work in a completed setting in which both factors are bi-Cartesian.

\begin{para}[Universal trace map]
\label{CAT_EVAL_TRACE}

Our first objective is to construct a completed evaluation map 
\[ \rm{tr}_Y:\PSh(X\times X^\op \times Y) \rightarrow \PSh(Y), \] 
natural in $X$ and $Y$, based on the duality between $\PSh(X)$ and $\PSh(X^\op)$ with respect to the presentable tensor product. For fixed $X$ and $Y$, this map is computed as a coend over the $X$ variable, i.e.~by transfer of presheaves along the span
\[\xymatrix{
& \rm{Tw}(X)^\op \times Y \ar[dl]\ar[dr]^{\rm{pr}_Y} \\
X\times X^\op\times Y && Y
}\]
by pulling back then taking a left Kan extension; here $\rm{Tw}(X)^\op:=\int_{(x,y):X^\op\times X}X(x,y)$ is the twisted diagonal (opposite to twisted arrow category). %(Note that $\rm{Tw}(X)=\rm{Tw}(X^\op)$ with the projections exchanged.)

Taking presheaves, this span itself fits into a universal family:
\[\xymatrix{
& \int_{X,\,Y:1\Cat}\PSh(\rm{Tw}(X)^\op \times Y) \\
\int_{X,\,Y:1\Cat}\PSh(X\times X^\op\times Y) \ar[ur] && \int_{X,\,Y:1\Cat}\PSh(Y)\ar[ul]^{\int_Y\rm{pr}_Y^{-1}}
}\]
where all terms are both Cartesian and co-Cartesian fibrations over $1\Cat\times 1\Cat$. Since left Kan extension $(\rm{pr}_Y)_!$ defines a left adjoint to $(\int_Y\rm{pr}_Y)^{-1}$ on each fibre, the latter admits a global left adjoint $\int_Y\rm{pr}_{Y!}$. Composing these, we obtain:
\begin{align*}
\alignhead{Universal trace map}
  \widetilde{\rm{ev}}:\int_{X:1\Cat,\,Y:1\Cat}\PSh(X\times X^\op \times Y) \quad\underset{1\Cat}{\longrightarrow}\quad \int_{Y:1\Cat}\PSh(Y).
\end{align*}
\end{para}

\begin{remark}
For fixed $X$ this trace map should be compared to the composition  law
\[ \bf{Corr}(\rm{pt},X) \otimes \bf{Corr}(X,Y) \rightarrow \bf{Corr}(\rm{pt},Y) \]
in the adjoint enrichment over presentable categories of the category $\bf{Corr}$ studied in \cite{AF_fibrations}.
\end{remark}

\begin{para}[Universal evaluation map]
\label{CAT_EVAL_EVAL}

We will recover the evaluation map by restricting the universal trace map \eqref{CAT_EVAL_TRACE} to a category of representable objects, defined as follows:
\begin{itemize}
\item Let $X$, $Y:1\Cat$. Twice currying the composition law $X\times 1\Fun(X,Y)\rightarrow Y \subset \PSh(Y)$ provides us with a 1-functor
\[
  1\Fun(X,Y) \rightarrow \PSh(X^\op\times Y)
\]
which sends a 1-functor $f:X\rightarrow Y$ to the presheaf $Y(-,f-)$. %Applying the Yoneda lemma twice tells us that functor is fully faithful.

\item Taking exterior product similarly defines a functor
\[
  X\times1\Fun(X,Y) \rightarrow \PSh(X\times X^\op\times Y), \quad 
    (x,f) \mapsto X(-,x)\times Y(-,f-).
\]
Note that this is not generally fully faithful --- for example, if $Y=\mathrm{pt}$ then its action on mapping spaces has the form $\rm{diag}:X(x,x')\rightarrow X(x,x')^X$.\footnote{Thanks to an anonymous referee for pointing this out.}
\end{itemize}
We will see that the composite of $\widetilde{\rm{ev}}$ with this functor takes values in representable presheaves on $Y$. More precisely, we will identify $\widetilde{\rm{ev}}(x,f)$ with (the Yoneda image of) $fx$, universally in $f$ and $x$. This shows that our universal evaluation map does indeed specialise, for fixed $X$, to the counit of the Cartesian closed structure on $1\Cat$.

To do this, we first extract a composition law from the $\Spc$-enrichment of $X$ (for more detail on the origin of these formulas, compare \cite[Def.~4.2.4]{Gepner_Haugseng} or \cite[\#4.5]{macpherson2019operad}). Recall that this consists of a map of planar operads $\mathtt{Ass}_{X} \rightarrow (\Spc,\times)$. Binary operations in $\mathtt{Ass}_X$ have the form
\[
  \mathtt{Ass}_{X}((w,a),(b,x);\,(y,z)) = X(y,w)\times X(a,b) \times X(x,z).
\]
and so the enrichment gives us a map
\[
  X(y,w)\times X(a,b) \times X(x,z) \rightarrow \Map(X(w,a)\times X(b,x),\,X(y,z)).
\]
natural in $w,a,b,x,y,z$.

Now, fixing $w=y$ and $x=z$ and integrating it as a contravariant functor of $(w,a,b,x):X^\op\times X \times X^\op\times X$, we find the law:
\begin{align*}  
  w:X^\op,\ x:X,\ [a\rightarrow b]:\rm{Tw}(X)^\op \quad &\return_{1\Cat} \quad X(b,x) \times X(w,a) \rightarrow X(w,x)
\end{align*}
where we used an identification $\rm{Tw}(X)^\op\cong \int_{(a,b):X^\op\times X}X(a,b)$.
Throwing an auxiliary 1-functor into the mix, we further obtain
\begin{align*} 
  f:X\rightarrow Y,\ y:Y^\op,\ x:X,\ [a\rightarrow b]:\rm{Tw}(X)^\op \quad & \return_{1\Cat} \quad X(b,x) \times Y(y,fa) \rightarrow Y(y,fx)
\end{align*}
and hence, passing to the colimit over $\rm{Tw}(X)^\op$,
\[
  f:X\rightarrow Y,\ y:Y^\op,\ x:X \quad\return\quad 
    \colim_{[a\rightarrow b]:\rm{Tw}(X)^\op}X(b,x)\times Y(y,fa) \stackrel{\alpha}{\rightarrow} Y(y,fx)  
\]
    which we will show is an equivalence.

\begin{lemma0}
The map $\alpha$ is an equivalence for any $x:X$, $y:Y$.
\end{lemma0}
\begin{proof}
We must show that for fixed $x,y$, this map exhibits $Y(y,fx)$ as a classifying space (groupoid completion) for $\int{[a\rightarrow b]:\rm{Tw}(X)^\op}X(b,x)\times Y(y,fa)$. An object of this category has the form of a diagram
\[ \xymatrix{
& a\ar[r]\ar@{|->}[d]  & b \ar[r] & x \\
y\ar[r] & fa
}\]
and a morphism is a twisted map of diagrams (going the wrong way on $a$). This reflects onto the full subcategory for which $b\rightarrow x$ is invertible, which in turn coreflects onto the subcategory for which $a\rightarrow b$ is also invertible. This defines a two step deformation retraction of the classifying space onto $Y(y,fx)$.
\end{proof}

\noindent Hence, restricting the [Universal trace map] to $X\times 1\Fun(X,Y)$ provides the evaluation map promised at the top of this section.
\end{para}

\begin{para}[Relative evaluation map]
\label{CAT_EVAL_REL}
Let $D:1\Cat$. We relativise the universal evaluation map over $D$ as follows:
\begin{enumerate}
\item The universal evaluation map is a morphism of co-Cartesian fibrations over $1\Cat$ (projecting on the $Y$ variable); applying the Grothendieck construction, consider it as a natural transformation
\[
  \left(\int_{X:1\Cat}1\Fun(X,-)\right) \times_{1\Cat}\left(\int_{X:1\Cat}X\right) \rightarrow \iden
\]
of endofunctors of $1\Cat$. Precomposing with the 2-cell
\[\xymatrix{  1\Cat\downarrow D \ar@/^/[r]^-{\rm{src}}\ar@/_/[r]_-D \ar@{}[r]|-\Downarrow & 1\Cat }\]
yields a commuting square
\[\xymatrix{
  \left(\int_{X:1\Cat}1\Fun(X,-)\right)\times_{1\Cat}\left(\int_{X:1\Cat}X\right) \ar[d]\ar[r]^-{\rm{ev}}
    &  \rm{src} \ar[d] \\
   \left(\int_{X:1\Cat}1\Fun(X,D)\right)\times_{1\Cat}\left(\int_{X:1\Cat}X\right) \ar[r]^-{\rm{ev}\times\iden}
    & D
}\]
in $1\Fun((1\Cat\downarrow D),1\Cat)$ (with the bottom row consisting of constant functors). Integrating again, this takes the form:
\[\xymatrix{
  \left(\int_{X:1\Cat,\,Y:1\Cat\downarrow D}1\Fun(X,Y)\right)\times_{1\Cat}\left(\int_{X:1\Cat}X\right) \ar[d]\ar[r]^-{\rm{ev}}
    &  \int_{Y:1\Cat\downarrow D}Y \ar[d] \\
   \left(\int_{X:1\Cat}1\Fun(X,D)\right)\times_{1\Cat}\left(\int_{X:1\Cat}X\right)\times(1\Cat\downarrow D) \ar[r]^-{\rm{ev}}
    & D \times (1\Cat\downarrow D)
}\]

\item Meanwhile, the natural inclusion $1\Fun_D(X,Y)\subseteq 1\Fun(X,Y)$ induces, integrating over $Y:1\Cat\downarrow D$, a commuting diagram
  \[\xymatrix{
    \int_{X,\,Y:1\Cat\downarrow D}1\Fun_D(X,Y) \ar[r] \ar[d]_X & 
      \int_{X:1\Cat,\,Y:1\Cat\downarrow D}1\Fun(X,Y) \ar[d] \ar[r] &
      \int_{X,\,Y:1\Cat}1\Fun(X,Y) \ar[d]^X \\
    1\Cat\downarrow D \ar[r] & \int_{X:1\Cat}1\Fun(X,D) \ar[r] & 1\Cat
  }\]
  (the right hand square being a pullback). Hence, composing with i),
  \[\xymatrix{
    \left(\int_{X,\,Y:1\Cat\downarrow D}1\Fun_D(X,Y)\right)\times_{1\Cat}
      \left(\int_{X:1\Cat}X\right) \ar[r] \ar[d]_{(X,\mathrm{pr}_2,Y)} & 
      \int_{Y:1\Cat\downarrow D}Y \ar[d] \\
    (1\Cat\downarrow D)\times_{1\Cat}\left(\int_{X:1\Cat}X\right)\times
      (1\Cat\downarrow D) \ar[r]^-{\rm{ev}\times\iden} & 
      D \times (1\Cat\downarrow D)
  }\]

\begin{comment}
\item Pulling back the universal evaluation map along the 2-cell
\[\xymatrix{  1\Cat\downarrow D \ar@/^/[r]^-{\rm{src}}\ar@/_/[r]_-D \ar@{}[r]|-\Downarrow & 1\Cat }\]
yields a commuting square

\[\xymatrix{
  \left(\int_{X,\,Y:1\Cat\downarrow D}1\Fun(X,Y)\right)\times_{1\Cat\downarrow D}\left(\int_{X:1\Cat\downarrow D}X\right) \ar[d]\ar[r]^-{\rm{ev}}
    &  \int_{Y:1\Cat\downarrow D}Y \ar[d] \\
   (1\Cat\downarrow D) \times \left(\int_{X,\,Y:1\Cat\downarrow D}1\Fun(X,Y)\right)\times_{1\Cat\downarrow D}\left(\int_{X:1\Cat\downarrow D}X\right) \ar[r]^-{\rm{ev}}
    & (1\Cat\downarrow D) \times D
}\]
where we used the independence of the bottom row of $Y$ to take the base $1\Cat\downarrow D$ outside the integral.

\item Now pull the left-hand vertical arrow of i) back along the tautological Cartesian section $1\Cat\downarrow D\rightarrow\int_{X:1\Cat}1\Fun(X,D)$ (inserted on the middle factor) to obtain:
\small
\[\xymatrix{
\left(\int_{X,\,Y:1\Cat\downarrow D}1\Fun_D(X,Y)\right) \times_{1\Cat\downarrow D}\left(\int_{X:1\Cat\downarrow D} X\right) \ar[r]\ar[d]
  & \left(\int_{X,\,Y:1\Cat\downarrow D}1\Fun(X,Y)\right)\times\cdots \ar[d] \\
(1\Cat\downarrow D)\times (1\Cat\downarrow D) \times D\ar[r] 
  & (1\Cat\downarrow D) \times \left(\int_{X,\,Y:1\Cat\downarrow D}1\Fun(X,Y)\right)\times\cdots
}\]
\end{comment}
\end{enumerate}
Finally, precomposing with the base change of the top-left corner of ii) along $\rm{src}:1\Cat\downarrow D\rightarrow 1\Cat$ (inserted at the base of the fibre product), we find a triangle:
\begin{align*}
\alignhead{Universal evaluation map, rel.~$D$} 
\xymatrix{
\left(\int_{X,\,Y:1\Cat\downarrow D}1\Fun_D(X,Y)\right) \times_{1\Cat\downarrow D} \left(\int_{X:1\Cat\downarrow D}X\right) \ar[rr]^-{\rm{ev}}\ar[dr] &&  \int_{Y:1\Cat\downarrow D}Y \ar[dl] \\
  & (1\Cat\downarrow D) \times D
}\end{align*}
which forms a transformation of co-Cartesian fibrations in the $Y:1\Cat\downarrow D$ variable, and restricts for fixed $X$ and $Y$ to the evident relative evaluation map.
\end{para}

\begin{para}[Extended Yoneda evaluation map]
\label{CAT_EVAL_YON}
\label{CAT_EVAL_SCART}

Our application for the universal evaluation map is to constructing an \emph{extended Yoneda} 2-cell which, in the case of Cartesian fibrations has the form
\begin{align*}
\alignhead{Extended Yoneda evaluation map}
x:D^\op,\ F:1\PSh(D) \quad 
  \return_{1\Cat} \quad \phi:1\PSh_D(\rm y_Dx,\,F) \quad 
    \return_{1\Cat} \quad \phi(\iden_x):F(x)
\end{align*}
where $1\PSh_D(-,-)$ means the 1-category of natural transformations of $1\Cat$-valued presheaves on $D$. As with the universal evaluation map, the subtlety requiring special attention is the naturality in $x$. Although this case is not used directly in the construction of the bivariant version, it has the same form and so by way of illustration I will explain it here.

First, by restricting  the universal evaluation map rel.~$D$ \eqref{CAT_EVAL_REL} to the relevant subcategory, we get a version for Cartesian fibrations:
\begin{align*}
\alignhead{Universal evaluation map, $D$-Cart} 
 \xymatrix{
\left(\int_{X,\,Y:\Cart_{D}}\Cart_{D}(X,Y)\right) \times_{\Cart_{D}} \left(\int_{X:\Cart_{D}}X\right) \ar[rr]^-{\rm{ev}}\ar[dr] &&  \int_{Y:\Cart_{D}}Y \ar[dl] \\
  & \Cart_{D} \times D
}\end{align*}
This transformation is co-Cartesian in the first variable and Cartesian in the second.

We will recover a Grothendieck integrated version of the Yoneda evaluation map by restricting the $X$ variable along a functor $\rm{free}:D\rightarrow \Cart_D$ which is defined to be the partner of the Yoneda embedding under Grothendieck integration. By \ref{CAT_GROT_PROPS}, for fixed $x:D$ this functor returns a free Cartesian fibration $D\downarrow x$ on $\{x\}$.

More precisely, the restriction goes along a composite of three maps, the first being the diagonal section $D\rightarrow D^{\Delta^1}$:
\begin{align*}
\int_{x:D,\,E:\Cart_D} \Cart_D(\rm{free}(x),E) \quad &
  \rightarrow \quad \left(\int_{x:D,\,E:\Cart_D}\Cart_D(\rm{free}(x),E)\right) \times_D D^{\Delta^1} \\
  \phi \quad & \return \quad (\phi,\,[\iden_x:x \rightarrow x]).
\intertext{ 
By \cite[\S4.1]{AF_fibrations}, we may identify $\int_{x:D}\rm{free}(x) = D^{\Delta^1}$ with projection on $x$ going to target projection on the right.
}
  & \acong \quad \left(\int_{x:D,\,E:\Cart_D}\Cart_D(\rm{free}(x),E)\right) \times_D \left(\int_{x:D}\rm{free}(x) \right)  \\
  & \return\quad (\phi, x) 
\intertext{
Finally, we compose with restriction in the $x$ variable along $\rm{free}:D\rightarrow \Cart_D$ and apply the universal evaluation map:
}
  & \rightarrow \quad \left(\int_{X,\,E:\Cart_D}\Cart_D(X,E)\right) \times_{\Cart_D} \left(\int_{X:\Cart_D}X\right) \\
  & \stackrel{\rm{ev}}{\rightarrow} \quad \int_{E:1\Cart_D}E.
\end{align*}
All maps are transformations of twisted Cartesian fibrations over $\Cart_D\times D$. Thus, the Grothendieck construction for twisted Cartesian fibrations (see \S\ref{BIV_TWIST} for more detailed comments on this) exchanges the map just constructed for a natural transformation of 1-functors of $D^\op\times 1\PSh(D)$. Inspecting the formulas for each step, we see that it does indeed have the expected behaviour $\phi\mapsto \phi(\iden_x)$ for fixed $x$.
\end{para}

\begin{remark}[Holistic construction of the composition law]

If we only cared about the evaluation map as a map of fibrations in \emph{spaces} over $X,Y$, we could easily bypass Grothendieck integration entirely, and hence issues with co/contravariance, in favour of an holistic expression
\[ 1\Cat^{\Delta^1}\times_{1\Cat}1\Cat^{\Delta^1}\cong 1\Cat^{\Delta^2} \stackrel{\sigma_{02}}{\longrightarrow} 1\Cat^{\Delta^1} \]
of the composition law in $1\Cat$, using the fact that $\Delta^2\cong \Delta^1\sqcup_{\Delta^0}\Delta^1$ in $1\Cat$. The identification of the restriction of this map to the fibre over $\rm{pt},\,X,\,Y$ with the counit of the Cartesian structure is straightforward, at least compared to the argument for Lemma \ref{CAT_EVAL_EVAL}.

With some additional theoretical development, this strategy ought to apply even to the full 1-categorical composition law by identifying $\int_{X,Y}1\Fun(X,Y)$ with $\rm{Lax}(\Delta^1,1\Cat)$, the 1-category of 1-functors $\Delta^1\rightarrow 1\Cat$ and \emph{lax} natural transformations.

One expects that lax natural transformations define a 2-category structure on $2\Cat$, adjoint to the Gray tensor product (and hence distinct from the standard enrichment of $2\Cat$). Provided that the representation of $\Delta^2$ as a pushout remains a 2-colimit in this 2-category, we may recover a universal composition law as
\[  \rm{Lax}(\Delta^1, 1\Cat) \times_{1\Cat} \rm{Lax}(\Delta^1,1\Cat) \cong \rm{Lax}(\Delta^2,1\Cat) \longrightarrow \rm{Lax}(\Delta^1, 1\Cat). \]
Of course, making this work would necessitate a presumably lengthy digression into the properties of the Gray tensor product and the accompanying 2-category of 2-categories.% We defer this to later works.
\end{remark}

\subsection{Adjoint functors and adjunctions in 2-categories}
\label{CAT_ADJ}

We discuss here the comparison between the unit-counit formulation of adjunctions, which is the one that can be internalised to any 2-category, to the traditional formulation in terms of mapping objects which works in $1\Cat$.

\begin{para}[Adjunctions of $1$-categories]
\label{CAT_ADJ_1CAT}

An adjunction between 1-categories is defined in \cite[Def.\ 5.2.2.1]{HTT} as a cospan $C\sqcup D\subset M$ plus a bi-Cartesian fibration $M\rightarrow\Delta^1$.\footnote{Modulo substituting the notion of Cartesian fibration used there with a homotopy invariant version \ref{CAT_FIB_CART_HTT}.} It is clear from this definition that the data of the adjunction is equivalent to that of just one or the other adjoint.

Via \cite[Lemma 4.1]{AF_fibrations}, the adjunction data is equivalent to that of a bimodule $C^\op\times D\rightarrow\bf{Spc}$ satisfying representability conditions. Hence, the same data may be expressed by providing two functors $L:C\rightarrow D$, $R:D\rightarrow C$ together with a natural equivalence of bimodules
\[ C(-,R-)\cong D(L-,-). \]
The process of passing through this equivalence is called passing to \emph{adjuncts}.
\end{para}

\begin{para}[Units and counits]
\label{CAT_ADJ_UNIT}

The notion of a unit transformation inducing an adjunction is introduced in \cite[Def.\ 5.2.2.7]{HTT}. The condition provided there can be related to the classical unit-counit formulation of adjunctions by a standard argument:

\begin{lemma0}Let $L:C\rightarrow D$ and $R:D\rightarrow C$ be two functors.
The condition of a map $e:\iden_C\rightarrow RL$ to be a unit is equivalent to the existence of a counit $\epsilon:LR\rightarrow\iden_D$ with which it satisfies the triangle identities up to homotopy.
\end{lemma0}
\begin{proof}The argument is standard from classical category theory; I repeat it here simply to reaffirm that it is not corrupted by homotopical subtleties.
Suppose that $e$ is a unit, and write $\tilde e$ for the induced bimodule isomorphism. We obtain a map $D(-,-)\rightarrow C(R-,R-)\stackrel{\tilde e}{\cong}D(LR-,-)$ and hence, by the Yoneda lemma, a transformation $\epsilon:LR\rightarrow \iden_D$.

Now, the claim is that the induced bimodule map $\tilde\epsilon:C(-,R-)\rightarrow D(L-,-)$ is homotopy inverse to $\tilde e$ if and only if the triangle identities are satisfied. I'll present one half of the argument: it follows from the commutativity of the diagram
\[\xymatrix{ D(L-,-)\ar@/^2ex/[rrr]^{-\circ\epsilon L}\ar[r]_-R\ar[dr]_{\tilde e} & C(RL-,R-) \ar[r]_-L\ar[d]^{-\circ e} & D(LRL-,LR-) \ar[d]^{-\circ Le} \ar[r]_-{\epsilon\circ-} & D(LRL-,-)\ar[d]^{-\circ Le} \\
& C(-,R-)\ar[r]^-L\ar@/_2ex/[rr]_{\tilde\epsilon} & D(L-,LR-) \ar[r]^-{\epsilon\circ-} & D(L-,-) 
}\]
that it is a homotopy \emph{left} inverse if and only if the left triangle identity $L\rightarrow LRL\rightarrow L$ is satisfied. Here we identified the top row with $-\circ\epsilon L$ using the fact that
\[\xymatrix{LRLx \ar[r]\ar[d]_\epsilon & LRy \ar[d]^\epsilon \\ Lx \ar[r] & y }\]
is commutative for any $Lx\rightarrow y$ by naturality of $\epsilon$. The other half follows the same way, mutatis mutandis.
\end{proof}

\noindent
It follows now from \cite[Prop.\ 5.2.2.8]{HTT} that a functor $C\rightarrow D$ admits an adjoint if and only if its image in the homotopy $(2,2)$-category admits an adjoint in the sense of strict 2-category theory.
\end{para}

%\begin{remark}The desired facts can apparently also be deduced from the theory of \cite{RV_Adj} together with the results invoked cited in the introduction (\S1.1) of \emph{op. cit}. This deduction passes through yet more auxiliary notions which are not homotopy-invariant (adjunctions in the strict 2-category of quasi-categories), so I chose to provide my own direct argument.

%Analogues of the desired statements are also more transparent in the presentation of duality in \cite[\S4.6.2]{HA}. Adjunctions can be placed in a slightly generalised version of this context as follows: $1\Fun(C^\op\times D,\bf{Spc})$ and $1\Fun(D^\op\times C,\bf{Spc})$ are dual as presentable bimodules over $1\Fun(C^\op\times C,\bf{Spc})-1\Fun(D^\op\times D,\bf{Spc})$, and the data of an adjunction is a duality inside these pair of dual categories. In this setting, the unit of the adjunction is the unit of the endomorphism object in $1\Fun(C^\op\times C,\bf{Spc})$.\end{remark}

\begin{para}[Adjunctions in $2$-categories]

We define \emph{adjunctions} in a $2$ (or higher) category following \cite[Def.\ 12.1.1.4]{GR}; thus, an adjunction in an $(\infty,2)$-category is nothing more than an adjunction in the homotopy $(2,2)$-category. 

 It follows from classical 2-category theory that the data of the adjunction is equivalent to the data of either:
\begin{itemize}\item only the left adjoint,
\item only the right adjoint,
\item both adjoints and the unit,
\item both adjoints and the counit.
\end{itemize}
Other presentations of equivalent data involving higher cells are possible; see \cite{RV_Adj} for a full account. 

As we have already observed \ref{CAT_ADJ_UNIT}, an adjunction in the 2-category $1\Cat$ is the same data as an adjunction between 1-categories in the present sense. %(Note: it is not so obvious from the presentation in \cite[\S5.2.2]{HTT} that the last two types of data are equivalent to adjunctions.)
\end{para}

\begin{para}[Adjunctions and opposite]
\label{CAT_ADJ_OPP}

Adjunctions in a 2-category $K$ are exchanged with adjunctions in its opposites as follows: an adjunction $L:x\rightleftarrows y:R$, $L\dashv R$ in $K$ yields an adjunction $L^{\op_2}:x\leftrightarrows y:R^{\op_2}$, $R^{\op_2}\dashv L^{\op_2}$ in $K^{\op_2}$ and $R^{\op_1}:x\rightleftarrows y:L^{\op_1}$, $R^{\op_1}\dashv L^{\op_1}$ in $K^{\op_1}$. These facts can be deduced from the classical statements in $(2,2)$-category theory (by construction, our $\op_k$ operations extend the classical ones of the same name \ref{CAT_MODELS_UNIQUE}).

This observation is somewhat incidental to the main arguments of the paper: we only employ it in \ref{BIV_FUN_OPP}, which is itself only invoked in the examples \S\ref{EX}.
\end{para}

\begin{remark}[$n$-adjunctions]
\label{CAT_ADJ_nCAT}

The argument for lemma \ref{CAT_ADJ_UNIT} can also be used to show that an adjunction in $n\Cat$ is the same as an $n$-natural isomorphism $C(-,R-)\cong D(L-,-)$ of $n$-functors into $(n-1)\Cat$. We invoke this observation in \ref{EX_MON_nCAT} to construct certain $n$-categorical symmetric monoidal structures.
\end{remark}

\begin{lemma}[Restricting adjunctions to subcategories]
\label{CAT_ADJ_SUBCAT}

  Let $L:C\rightleftarrows D:R$ be an adjunction, and let $C_0\subseteq C$, $D_0\subseteq D$ be subcategories.
  Suppose the following conditions are satisfied:
  \begin{enumerate}
    \item $L$ maps objects of $C_0$ into $D_0$;
    \item $R$ maps objects of $D_0$ into $C_0$;
    \item for each $x:C_0$ and $y:D_0$, the adjunction isomorphism $C(x,Ry)\cong D(Lx,y)$ restricts to an isomorphism between $C_0(x,Ry)$ and $D_0(Lx, y)$.
  \end{enumerate}
  Then there is a unique adjunction $L_0\dashv R_0$ between $C_0$ and $D_0$ making the diagrams
  \[\xymatrix{
    C_0 \ar[r] \ar[d] & D_0 \ar[d] & C_0 \ar[d] & D_0 \ar[d] \ar[l] & C_0(L_0-,-) \ar@{=}[r] \ar[d] & D_0(-,R_0-) \ar[d] \\
    C \ar[r] & D & C & D \ar[l] & C(L-, -) \ar@{=}[r] & D(-,R-)
  }\]
  commutative.
  
\end{lemma}
\begin{proof}

  In light of the hypotheses and the definition of adjunction data \ref{CAT_ADJ_1CAT}, the uniqueness is clear, and for the existence the only criterion left to check is that $L$, resp.~$R$, maps morphisms of $C_0$ into $D_0$, resp.~$D_0$ into $C_0$. This is trivially true for identity morphisms. Now, the adjunction isomorphisms identify the induced maps
  \[
    L:C(x,y) \rightarrow D(Lx,Ly)\cong C(x,RLy) \qquad R:D(u,v) \rightarrow C(Ru,Rv)\cong D(LRu,v)
  \]
  with postcomposition with the unit, resp.~precomposition with the counit; hence, it is equivalent to show that the unit and counit factor through the category of endofunctors of $C_0$, resp.~$D_0$.
This follows by applying the adjunction isomorphism to identity morphisms.  \qedhere

\end{proof}

\section{Bivariance}
\label{BIV}

A complete 2-fold Segal space $\rm{Span}_{D,\bullet}$ of correspondences in $D$ is defined in \cite[\S7.1.2]{GR}. The advantage of their approach to the definition, as mentioned in the introduction, is that it is obviously well-defined, functorial, and commutes with limits. On the other hand, it is not so immediate how we construct functors out of $\rm{Span}_{D,\bullet}$.

In this section we will introduce two other perspectives on correspondences, each with their own advantages and drawbacks:
\begin{itemize}
\item The notion of a \emph{universal bivariant extension} (\S\ref{BIV_UNIV}), which by definition has a universal property and is therefore natural. However, their existence is not straightforward.

\item The notion of a \emph{category of correspondences} (\S\ref{BIV_LOC}). A category of correspondences is easy to recognise, but it is not immediate that they are unique or functorial. 

\end{itemize}
In \S\ref{BIV_EXT}, we provide a method to construct functors out of categories of correspondences. We can thereby deduce that they are unique up to isomorphism.

In \S\ref{UNIV} we will identify these two concepts, at the same time bringing \emph{existence} and \emph{recognition} to universal bivariant extensions and \emph{uniqueness} to correspondences.

\subsection{Base change}
\label{BIV_BC}

We begin with a discussion of markings with base change and the base change, or \emph{Beck-Chevalley}, condition in 2-categories.

\begin{defn}[Base change]
\label{BIV_BC_DEF}

A marking $S_D$ of a 1-category $D$ (Def.~\ref{CAT_FIB_MARK}) is said to \emph{have base change} if pullbacks of elements of $S_D$ along arbitrary maps are representable and belong to $S_D$. Say also that $S_D$ is a \emph{marking with base change}, and that the pair $(D,S_D)$ is a \emph{marked category with base change}.

If $(D_0,S_0)$ and $(D_1,S_1)$ are marked categories with base change, a marked functor $D_0\rightarrow D_1$ is said to \emph{preserve base change} if it preserves pullbacks of members of $S_0$. 

A natural transformation $\psi:F_0\rightarrow F_1$ between marked functors with base change $F_i:D_0\rightarrow D_1$ is \emph{base change exact} if for all maps $d_0\rightarrow d_1$ in $S_0$ the square 
\[\xymatrix{F_0d_0\ar[r]^\psi\ar[d] & F_1d_0\ar[d] \\ F_0d_1\ar[r]^\psi & F_1d_1 }\] is a pullback in $D_1$.

Marked categories with base change form a 2-subcategory $1\Cat^{+r}$ of $1\Cat^+$ with specification
\begin{enumerate}[label=\arabic*), start=0]
\item categories with base change pattern;
\item base change preserving functors;
\item base change exact natural transformations.
\end{enumerate}
\end{defn}

\begin{remark}
\label{BIV_BV_2CAT}

The trivial marking always has base change, and $1\Cat^{+r}$ is closed in $1\Cat^+$ under limits and hence tensoring over $1\Cat$. Using these facts it is possible to show that the adjoint enrichment of $1\Cat^{+r}$ agrees with its inherited structure as a 2-subcategory of $1\Cat^+$.
Since we won't use this fact, I omit the details of this argument.

The condition for a natural transformation to be base change exact emerges, via
\[ 1\Fun\left(\Delta^1,1\Fun^{+r}(X,Y) \right)\cong 1\Fun^{+r}(\Delta^1\times X,Y), \]
as the base change condition for the square
\[\xymatrix{
(0,x)\ar[r]\ar[d] & (1,x)\ar[d] \\
(0,y)\ar[r] & (1,y)
}\]
where $x\rightarrow y$ is in $S_X$.
\end{remark}

\begin{remark}[Collar change]
\label{BIV_BC_COLLAR}

A marking of a 1-category has \emph{collar change} if pushouts of elements of $S_D$ are representable in $S_D$. There are associated notions of marked functor with collar change and collar change exact natural transformation that assemble to form a 2-category $1\Cat^{+\ell}$. (The name `collar change' is inspired by bordism theory.)
\end{remark}

\begin{para}[Conjugate mapping]
\label{BIV_BC_MATE}

Let 
\[\xymatrix{x_{00}\ar@{}[dr]|\Leftarrow\ar[r]^{\bar f_!}\ar[d]_{\bar g_!} & x_{01}\ar[d]^{g_!} \\ x_{10}\ar[r]^{f_!} & x_{11}}\] 
be a lax commuting square in a 2-category $K$ (that is, a map $[1,1]^{\op_2}\rightarrow K$ in the notation of \cite[\S 10.3.4]{GR}), and suppose that both $g_!$ and $\tilde g_!$ admit right adjoints $g^!\!,\,\tilde g^!$. (In what follows, for squares oriented from top-left to bottom-right like this we will always arrange that the \emph{vertical} arrows are the ones that get adjoints.)

The 2-cell $\phi$ that exhibits the commutativity induces a \emph{(right) conjugate mapping} (or, as the Australians call it, a \emph{mate}) via an equivalence
\begin{align*}
&[\text{Beck-Chevalley conjugation}] \\
& K(x_{00},x_{11})(g_!\bar f_!,f_!\bar g_!)\cong  K(x_{10},x_{01})(\bar f_!\bar g^!,g^!f_!)
\end{align*}
represented graphically by the composite
\[\xymatrix{
x_{10}\ar@{=}[d]\ar[r]^{\bar g^!}\ar@{}[dr]|\epsilon & x_{00} \ar@{}[dr]|\Leftarrow\ar[r]^{\bar f_!}\ar[d]_{\bar g_!} & x_{01}\ar@{}[dr]|e \ar[d]^{g_!}\ar@{=}[r] & x_{01}\ar@{=}[d] \\
x_{10}\ar@{=}[r] & x_{10}\ar[r]_{f_!} & x_{11}\ar[r]_{g^!} & x_{01}
}\]
Textually, one way of representing this is as a composition 
\[K(x_{00},x_{11})( g_!\bar f_!, f_!\bar g_!) \cong K(x_{10},x_{11})(g_!\bar f_!\bar g^!,f_!)\cong K(x_{10},x_{01})(\bar f_!\bar g^!,g^!f_!) \]
of adjunct isomorphisms; another goes via $K(x_{00},x_{01})(\bar f_!,g^!f_!\bar g_!)$. (The identification between these two paths is given by associativity of composition.) Here we used the fact that postcomposition with $g_!\dashv g^!$ gives an adjunction on $K(-,x_{01})\rightleftarrows K(-,x_{11})$, and precomposition with $\bar g_!\dashv\bar g^!$ an adjunction on $K(x_{00},-)\rightleftarrows K(x_{10},-)$.
\end{para}

\begin{defn}[Beck-Chevalley condition]
\label{BIV_BC_CONDITION}

A commutative square is said to be \emph{(right) Beck-Chevalley} if its vertical arrows admit right adjoints and the (right) conjugate mapping is an isomorphism. (This condition is called \emph{right adjointability} in \cite[Def.\ 4.7.4.13]{HA}.)

It is said to be left Beck-Chevalley if its image in $K^{\op_2}$ is right Beck-Chevalley.
\end{defn}

\subsection{Bivariant functors}
\label{BIV_FUN}

This section introduces the basic definitions and syntax of bivariant functors, bivariant functors with base change, and the 2-categories they form.

\begin{defn}[Bivariance]
\label{BIV_FUN_DEF}

Let $D:1\Cat^{+r}$, $K:2\Cat$. A functor $H:D\rightarrow K$ is said to be \emph{right bivariant}, or \emph{right $S_D$-bivariant} if it is necessary to make $S_D$ explicit, if for each $f:x\rightarrow y$ in $S$, $f_!:=H(f)$ admits a right adjoint. We denote this adjoint $f^!$.

A right bivariant functor $H$ is said to have \emph{base change} if for each fibre square 
\[\xymatrix{x\times_zy \ar[r]^-{f}\ar[d]_g & x\ar[d]^g \\ y\ar[r]^f & z}\]
in $D$, with $g\in S_D$, the base change map $f_!g^!\rightarrow g^!f_!$ defined in \ref{BIV_BC_MATE} is an equivalence in $K(Hy,Hx)$. (By a continued abuse of notation, the same letter is used to denote a morphism and its pullback.)
If $K$ is not specified, bivariant functors take values in $1\Cat$.

A natural transformation $\Phi:H_1\rightarrow H_2$ is called a \emph{bivariant natural transformation} if it transforms right adjoints of $H_1$ of elements of $S_D$ into right adjoints. More precisely, for each $f:x\rightarrow y$ in $S_D$ the square \[\xymatrix{   
  H_1(x)\ar[r]^{\Phi(x)}\ar[d]_{f_!} &  H_2(x)\ar[d]^{f_!} \\  
  H_1(y)\ar[r]^{\Phi(y)} & \rm H_2(y) 
  }\] 
is a (right) Beck-Chevalley square.
\end{defn}

\begin{para}[Composition]
\label{BIV_FUN_COMP}

Bivariant functors (with base change) have the following naturality properties:
\begin{itemize}\item In $K:2\Cat$:
\begin{itemize}
\item if $G: K_1\rightarrow  K_2$ is a 2-functor and $H:D\rightarrow K_1$ is a bivariant functor (with base change) then $GH$ is bivariant (with base change). Indeed, any functor of 2-categories preserves adjunctions and hence the construction of mates.
\item If $\phi:G_0\rightarrow G_1$ is a 2-natural transformation, then the induced morphism $G_0H\rightarrow G_1H$ is a bivariant natural transformation; the commutativity of the relevant square is just a consequence of naturality.\end{itemize}
\item In $D:1\Cat^{+r}$:
\begin{itemize}\item If $F: D_0\rightarrow D_1$ is a marked functor (with base change), and $H: D_1\rightarrow K$ is bivariant (with base change), then the composite $HF$ is bivariant (with base change).
\item If $\psi:F_0\rightarrow F_1$ is a base change exact natural transformation of marked functors with base change, then $HF_0\rightarrow HF_1$ is a bivariant natural transformation. (Note: there is no version without base change.)
\end{itemize}
\end{itemize}
\end{para}

\begin{para}[The 2-category of bivariant functors]
\label{BIV_FUN_2CAT}

Right bivariant functors $D\rightarrow K$ with base change form a 2-subcategory $\Biv(D,K)$ of $2\Fun(D,K)$ with the following specification (cf.\ \ref{CAT_FURTHER_SPECDEF}):
\begin{enumerate}[label=\arabic*), start=0]
\item $S_D$-bivariant functors with base change;
\item $S_D$-bivariant natural transformations;
\item All 3-transfers.
\end{enumerate}
It is stable under colimits and limits (but as we do not need to use this fact, I do not provide a proof). In the case $ K=1\Cat$, we abbreviate this to $\Biv_D$. 

By the composition properties \ref{BIV_FUN_COMP} and specification in families \ref{CAT_FURTHER_SPEC}, bivariant functors define a sub-bifunctor
\begin{align*}
\alignhead{Bivariant functors, 1-natural}
  D:(1\Cat^{+r})^\op,\  K:2\Cat \quad \underset{1\Cat}{\return}\quad  \Biv(D,K):2\Cat 
\end{align*}
of the 2-functor category construction.\end{para}

\begin{remark}
\label{BIV_FUN_3NAT2CAT}

One expects the a more complete functorial description of $\Biv(-,-)$ to look like:
\begin{align*}
\alignhead{Bivariant functors, 3-natural (\dag)}
  D:(1\Cat^{+r})^\op,\  K:2\Cat \quad \underset{3\Cat}{\return}\quad  \Biv(D,K):2\Cat 
\end{align*}
This would be cut out as a 3-subfunctor of the mapping 2-category of $2\Cat$ (\ref{CAT_TENS_YON}). However, we haven't established a way to specify 2-subcategories in families indexed by a 3-category --- we'd need a 3-categorical Grothendieck construction for that.
\end{remark}

\begin{para}[Opposite notions]
\label{BIV_OP}
\label{BIV_FUN_OPP}

Our definitions of bivariant functor and the base change property each involved a choice of parity. The dual notions for a functor $H:D\rightarrow K$ are:
\begin{itemize}
\item $H$ is \emph{left-bivariant} if $H(f)$ has a \emph{left} adjoint for all $f\in S$;
\item (when $(D,S_D)$ is a marked category with collar change as in \ref{BIV_BC_COLLAR}) $H$ \emph{has collar change} if it takes pushout squares to (left or right, depending on whether it is left or right bivariant) Beck-Chevalley squares.
\end{itemize} 
Denote the categories of left, resp.~right bivariant functors by $\Biv^\mp$, resp.~$\Biv^\pm$ where, in the former case, we of course ask that natural transformations preserve the \emph{left} adjoints. The base, resp.~collar change properties are indicated with a further superscript $r$, resp.~$\ell$. Elsewhere in the paper, $\Biv$ is a shorthand for $\Biv^{\pm r}$.

By \ref{CAT_ADJ_OPP}, opposites intertwine the various notions in the following way:
\begin{enumerate}
\item[$\op_1$] $\Biv^\pm(D,K)=\Biv^\mp(D^{\op_1},K^{\op_1})^{\op_1}$ and $\Biv^{\pm r}(D,K)=\Biv^{\mp\ell}(D^{\op_1},K^{\op_1})^{\op_1}$ since passing to $\op_1$ exchanges left with right adjoints and pullback squares with pushout squares.
\item[$\op_2$] $\Biv^\pm(D,K)=\Biv^\mp(D,K^{\op_2})^{\op_2}$ and $\Biv^{\pm r}(D, K)=\Biv^{\mp r}(D,K^{\op_2})^{\op_2}$ because $\op_2$ exchanges left and right adjoints.
\end{enumerate}
\renewcommand{\arraystretch}{1.5}
This behaviour is summarised in the following table: 
\begin{center}
	\begin{tabular}{c|c|c|} & - & $\op_1$ \\ \hline
		- & $\Biv^{\pm r}$ & $\Biv^{\mp\ell}$ \\ \hline
		 $\op_2$ & $\Biv^{\mp r}$ & $\Biv^{\pm\ell}$ \\\hline 
	\end{tabular}
\end{center}
\vskip1ex
The example $D(-,\Z)$ from the introduction and many of the fundamental examples in \S\ref{EX_COEF} are $1\Cat^{\op_1}$-valued left bivariant functors with base change.
\end{para}

\begin{para}[Bivariant functors of products]
\label{BIV_FUN_PROD}
Let $D,\,E:1\Cat^+$. A functor $F$ with domain $D\times E$ is bivariant if and only if its restriction $F_d$ to each slice $\{d\}\times E$ and $D\times\{e\}$ is bivariant.

Now let $D,\,E:1\Cat^{+r}$. Then $D\times E$ has base change, and the condition for $F$ to have base change may be rephrased in terms of the following criteria:
\begin{enumerate}
\item for each $d:D$, $F_d$ has base change as a functor of $E$;
\item for each arrow $d\rightarrow d'$ in $D$, the attendant natural transformation $F_d\rightarrow F_{d'}$ is bivariant.
\end{enumerate}
together with similar conditions indexed over objects and morphisms of $E$. We meet a generalisation of this example in \S\ref{UNIV_CART}.
\end{para}

\begin{para}[Syntax]
\label{BIV_SYNTAX}

We expand the language of \ref{CAT_TYPE} to declare a functor as right (resp.~left) bivariant in a variable $\beta:D$ by postfixing $(\pm)$ (resp.~$(\mp)$), and indicating base (resp.~collar) change by further postfixing an $r$ (resp.~$\ell$). The symbol $(\pm)$ may only be postfixed to an object of $1\Cat^{+}$, $(\pm,\,r)$ to an object of $1\Cat^{+r}$, and so on. By the composition properties, any natural formula in $H:\Biv(D,K)$ is bivariant in whichever explicit variables are passed to $H$.

In formulas with multiple factors as input, applying one of these postfixes to a single factor signifies that the functor is bivariant (resp.~with base or collar change) with respect to a marking restricted to slices of that factor. In particular, the notations
\[  d:D\ (*),\ e:E\ (*) \qquad \Leftrightarrow \qquad (d,e):D\times E\ (*) \]
are equivalent, where $(*)$ stands for the same symbol --- one of $(\pm)$, $(\pm,\,r)$, $(\mp)$, and so on --- in each place. There is no analogue of the right-hand side when the postfixes of $D$ and $E$ differ. The left-hand side provides a convenient notation in this case.

We have the following currying rule:
\begin{align*}
\alignhead{Bivariant currying}
  c:C\ (*),\ d:D\ (**) \quad 
    &\return\quad	e:E \\ 
  \Rightarrow \hspace{6em} d:D\ (**) \quad
    &\return\quad c:C\ (*) \quad\return\quad e:E
\end{align*}
where $(*)$ and $(**)$ are possibly different postfixes. This reflects a hom-tensor equivalence.
\end{para}

\begin{eg}[Bivariant evaluation map]
\label{BIV_FUN_EV}

Restricting the evaluation 2-functor $2\Fun(D,K)\times D\rightarrow K$ from the Cartesian closed structure of $2\Cat$ to bivariant functors, gives a formula
\begin{align*} 
\alignhead{Bivariant evaluation 2-functor}
  H:\Biv(D,K),\ \beta:D\ (\pm,\, r) \quad \underset{2\Cat}{\return}  \quad H\beta:K
\end{align*}
in which $(\pm,\,r)$ signifies bivariance with respect to the union of the markings $\{H\}\times S_D\subseteq\Biv(D,K)\times D$ for $H:\Biv(D,K)$. After \ref{BIV_FUN_PROD}, bivariance can be checked on slices, but base change entails an additional condition indexed over natural transformations $H\rightarrow H'$.
\end{eg}

\begin{eg}[Trivial markings]
\label{BIV_FUN_TRIV}

Let $D$ be a 1-category considered with the trivial marking. This marking has both base and collar change. Every functor with domain $D$ is left/right bivariant with base/collar change.
\end{eg}

\begin{eg}[Bivariant functors into a 1-category]
A functor from a marked category into a 1-category is bivariant if and only if it maps all marked arrows to isomorphisms.
\end{eg}

\subsection{Universal bivariant extensions}
\label{BIV_UNIV}

We introduce here the notion of a universal bivariant extension and exhibit its behaviour in families. The family version of the universal property is needed for the monoidal version of the universal property of correspondences \S\ref{UNIV_MON}.

\begin{defn}[Families of bivariant functors]
\label{BIV_UNIV_FAM}

If $K, K^\prime:I\rightarrow 2\Cat$ are 1-functors, then we will write $2\Fun_{i:I}(K_i,K_i^\prime)$ for the 2-category of 2-natural transformations $K\rightarrow K^\prime$. 

If $D:I \rightarrow 1\Cat^{+r}$ is a 1-functor, write $\Biv_{i:I}(D_i,K_i^\prime)$ for the subcategory of $2\Fun_{i:I}(D_i,K_i^\prime)$ whose objects are, for each $i$, bivariant functors of $D_i$ and whose maps are, for each $i$, bivariant natural transformations. In other words, it is defined as a pullback:
\[\xymatrix{
\Biv_{i:I}(D_i,K^\prime_i) \ar[r]\ar[d] & \prod_{i:\Ob(I)}\Biv(D_i, K^\prime_i)\ar[d] \\ 
1\Fun_{i:I}(D_i,K_i^\prime) \ar[r] & \prod_{i:\Ob(I)}1\Fun(D_i, K^\prime_i)
}\]
Call the objects of $\Biv_{i:I}(D_i, K^\prime_i)$ \emph{$I$-indexed families} of bivariant functors $D_i\rightarrow K_i$.
\end{defn}

\begin{defn}[Universality]
\label{BIV_UNIV_DEF}

A family of bivariant functors $H:\Biv_{i:I}(D_i,K_i)$ is said to be an \emph{$n$-universal bivariant extension}, where $n:\{1,2,3\}$, if restriction along $H$ induces an equivalence of $(n-1)$-categories $2\Fun_{i:I}(K_i,-)\tilde\rightarrow\Biv_{i:I}(D_i,-)$.
\end{defn}

\begin{para}[1-category of bivariant extensions]
\label{BIV_UNIV_1CAT}

Write $\bf{Biv}$ for the full 1-subcategory of the pullback
\[\xymatrix{
	\bf{Biv} \ar@{}[r]|-{\subset} &	
		1\Cat^{+r}\times_{2\Cat}2\Cat^{\Delta^1}\ar[r]\ar[d] & 	
		2\Cat^{\Delta^1}\ar[d]^{\rm{src}} \\
 	&
 		1\Cat^{+r}\ar[r]^-{\text{forget}}_-{\text{marking}} & 2\Cat
}\]
spanned by the right bivariant functors with base change. Thus the objects of $\bf{Biv}$ are tuples $(D,K,H)$ where $D:1\Cat^{+r}$, $K:2\Cat$ and $H:D\rightarrow K$ is a right bivariant functor with base change, and the morphisms are commuting squares. 

The projection to $1\Cat^{+r}\times 2\Cat$ is a bifibration, in fact, a sub-bifibration of the pullback of the unit bifibration of $2\Cat$. By \cite[\S4]{AF_fibrations}, the latter integrates the 2-functor space $\Ob(2\Fun(-,-))$, so via this embedding, $\bf{Biv}$ integrates $\Ob(\Biv(-,-)):(1\Cat^{+r})^\op\times 2\Cat\rightarrow\Spc$.

Write $\bf{Biv}_D$ for the 1-category of \emph{bivariant extensions of $D$}, that is, the fibre of $\rm{src}:\bf{Biv}\rightarrow 1\Cat^{+r}$ over $D$. Equivalently, it is the integral of the covariant 1-functor $K\return \rm{Ob}(\Biv(D,K))$. This construction is contravariant in $D$, so if $D:I\rightarrow 1\Cat^{+r}$ is a 1-functor we may write $\int_{i:I}\bf{Biv}_{D_i}= I\times_{1\Cat^{+r}}\bf{Biv}$ and $\prod_{i:I}\bf{Biv}_{D_i}=1\Fun_I(I,\int_{i:I}\bf{Biv}_{D_i})$.
\end{para}

\begin{lemma}[Bootstrapping universality]
\label{BIV_UNIV_BOOT}

Let $D:I\rightarrow 1\Cat^{+r}$ be a 1-functor, and suppose that for each $i:I$ $D_i$ admits a 1-universal bivariant extension. The maps in the left-hand column
\[\xymatrix{
  \{3\text{-universal bivariant extensions of }D_i\} \ar[d] \ar@{}[r]|-\subset & \prod_{i:I}\bf{Biv}_{D_i} \ar@{=}[d]
\\
  \{1\text{-universal bivariant extensions of }D\} \ar[d] \ar@{}[r]|-\subset & \prod_{i:I}\bf{Biv}_{D_i} \ar[d]
\\
  \prod_{i:\Ob(I)}\{1\text{-universal bivariant extensions of }D_i\} \ar@{}[r]|-\subset & \prod_{i:\Ob(I)}\bf{Biv}_{D_i}
}\]
are equivalences of spaces (hence all contractible). In particular, any 1-universal bivariant extension is 3-universal.
\end{lemma}
\begin{proof}
The argument for 1-universality is typical. I include the details to illustrate the fact that it fails conspicuously for the 2- and 3-universal versions.

A 1-universal bivariant extension of $D_i$ is an initial object of $\bf{Biv}_{D_i}$. When this exists for each $i:I$, the full subcategory of $\bf{Biv}_D$ spanned by the 1-universal extensions for each $i:I$ projects isomorphically down to $I$. The inverse is a fully faithful left adjoint to the projection $\bf{Biv}_D\rightarrow I$.
 
Now consider the pullback square
\[\xymatrix{
 1\Fun(I, \bf{Biv}_D) \ar[r]\ar[d] 
   & 1\Fun(I,\bf{Biv}) \ar[d] \\
 1\Fun(I,I) \ar[r]  & 1\Fun(I, 1\Cat^{+r}). 
}\]
By definition, a 1-universal extension of the family is an initial object of the fibre over $D$ of the right-hand vertical arrow. Since $D$ lifts to $\iden_I$ on the bottom right, it is enough to find an initial object here. This is given by postcomposing with the left adjoint.

For $3$-universality, we now rely on bootstrapping the 1-universal statement by applying Lemma \ref{CAT_BOOT_SUBREP} to the 1-natural transformation $2\Fun(K,-)\rightarrow \Biv(D_i,-)$ of functors of $2\Cat$ (considered as 2-powered over itself).
\end{proof}

\begin{remark}[$n$-naturality of objects satisfying $n$-universal properties]
\label{BIV_UNIV_RMK}
Intuitively, one might expect to be able to assemble universal bivariant extensions into families indexed by a 2-category in the manner of the proof of Lemma \ref{BIV_UNIV_BOOT}. However, $\rm{Ob}(\Biv(D,-))$, hence also $\bf{Biv}_D$,2 is not a 2-functor of $1\Cat^{+r}$, so we would have to use a different fibration. We might try the integral of the 1-category-valued 2-functor $\Biv(D,-)$; however, its integral over $2\Cat$ is the \emph{lax} slice under $D$, in which a 2-universal extension  is \emph{not} initial. Evidently, some new ideas are needed to make this type of construction work in $n$-category theory for $n>1$.
\end{remark}

\begin{eg}[Trivial example]
\label{BIV_UNIV_TRIV}

As observed in Example \ref{BIV_FUN_TRIV}, if $D$ is trivially marked then every functor is bivariant; hence $D$ is its own universal bivariant extension.
\end{eg}

\subsection{Bi-Cartesian fibrations}
\label{BIV_BICART}

For bivariant functors valued in $1\Cat$, there is a Grothendieck construction generalising \S\ref{CAT_GROT}. When we moreover impose base change, there is a convenient theory of free fibrations, after \S\ref{CAT_FIB}. Aspects of this appear in \cite[\S4.7.4]{HA}. In this section we review this theory and show that the free Cartesian fibration on a co-Cartesian fibration is free bi-Cartesian with base change.

\begin{para}[Interplay between mates and Cartesian lifts]
\label{BIV_BICART_MATE}

Let $\phi:E\to E'$ be a co-Cartesian transformation of co-Cartesian fibrations over $\Delta^1$. Writing $f:0\to 1$ for the non-trivial morphism of $\Delta^1$, these data are captured via the functors they classify in the form of a commuting square
\[\xymatrix{
  E_0 \ar[r]^\phi \ar[d]_{f_!} & E'_0\ar[d]^{f_!} \\
  E_1 \ar[r]^\phi & E'_1
}\]
in $1\Cat$. Now let $z:E_1$ and suppose that $f$ admits Cartesian lifts ending at $z$ and $\phi z$. We would like a way to characterise the associated mate transformation $\phi f^!z\rightarrow f^!\phi z$ \eqref{BIV_BC_MATE} in terms of these lifts.
\begin{lemma0}
There is a commuting triangle
\[\xymatrix{
  \phi f^! z \ar[r]^{\rm{mate}} \ar[dr]_{\phi(f^z)} &
    f^!\phi z \ar[d]^{f^{\phi z}} &
      (\in E'_0) \\
  & \phi z &
      (\in E'_1)
}\]
in $E'$.
\end{lemma0}
\begin{proof}
Recall that the conjugate 2-cell for $f_!,\phi$ is obtained by chasing the boundary of the diagram:
\[\xymatrix{
  E_1 \ar[r]^{f^!} \ar@{=}[dr] &
    E_0 \ar[r]^\phi \ar[d]^{f_!} \ar@{}[dl]|(.35)\Leftarrow & 
      E'_0 \ar[d]_{f_!} \ar@{=}[dr] &
        \ar@{}[dl]|(.65)\Leftarrow \\
  & E_1\ar[r]^\phi	& 
      E'_1 \ar[r]_{f^!} &
        E'_0
}\]
whereby its action $\phi f^!z \to f^!\phi z$ on $z:E_1$ comes from following the top row of
\[\xymatrix{
  \phi f^!z \ar[dr]_{\text{(ii)}} \ar[r]^-{\text{(ii)}} \ar[drr]^(.7){\text{(iii)}} &
    f^!f_!\phi f^!z \ar@{=}[r]^-{\text{(i)}} \ar[d]^(.3){\text{(i)}}|(.47)\hole &
      f^!\phi f_!f^!z \ar[r]^-{\text{(i)}} \ar[d]^{\text{(i)}} & 
        f^!\phi z \ar[d]^{\text{(i)}} &
          (\in E'_x) \\
  & f_!\phi f^!z \ar@{=}[r]_-{\text{(iv)}} &
      \phi f_!f^!z \ar[r]_-{\text{(v)}} &
        \phi z &
          (\in E'_y)
}\]
constructed as follows:
\begin{enumerate}
\item All of the vertical arrows are Cartesian over $f$. The second and third arrows on the top row are obtained by pulling back the bottom row along $f^!$, whence commutativity.
\item The first horizontal arrow is obtained by factorising the co-Cartesian arrow over $f$ starting at $\phi f^!z$ through the Cartesian arrow ending at $f_!\phi f^!z$. This gives commutativity of the left-most triangle.
\item The second diagonal arrow $\phi(f_{f^!z}):\phi f^!z\rightarrow \phi f_!f^!z$ is co-Cartesian by the hypothesis that $\phi$ is a co-Cartesian transformation. This gives commutativity of the skewed triangle (since co-Cartesian arrows with fixed source are unique) and hence with the rest of the diagram.
\item The identifications in the middle come from the commutativity of the input square.
\item The lower-right horizontal arrow is obtained by applying $\phi$ to the counit map $f_!f^!z\rightarrow z$, which in turn is obtained by factorising the Cartesian arrow $f^z:f^!z\rightarrow z$ through the co-Cartesian arrow $f_{f^!z}:f^!z\rightarrow f_!f^!z$. 
  
In particular, the composite of the entire lower route equals $\phi(f^z)$. \qedhere
\end{enumerate}
\end{proof}
\end{para}

\begin{remark}

  The constructions i)-v) in the proof of Lemma \ref{BIV_BICART_MATE}, and hence the triangle in the statement, can be made natural in $z$. However, since we won't use this, I omit the details.

\end{remark}

\begin{defn}[Bi-Cartesian fibrations]
\label{BIV_BICART_DEF}

Let $(D , S_D):1\Cat^{+}$. A 1-functor $E\rightarrow D$ is said to be a \emph{bi-Cartesian fibration} if it is a $D$-co-Cartesian fibration and an $S_D$-Cartesian fibration. 
A functor between bi-Cartesian fibrations over $(D,S_D)$ is said to be a \emph{bi-Cartesian transformation} if it is a $D$-co-Cartesian transformation and an $S_D$-Cartesian transformation. We define $\biCart_{D \sep S_D}:=1\Cart_{D \sep S_D}\cap 1\coCart_{D \sep D}$ as a 1-subcategory of $1\Cat\downarrow D$. It inherits a mapping 1-category bifunctor
\[
  \biCart_{D \sep S_D}(-,-):\biCart_{D \sep S_D}^\op\times \biCart_{D \sep S_D} \rightarrow 1\Cat.
\]
Compare \cite[Def.\ 4.7.4.1]{HA}.

More generally, if $S_D^+,\,S_D^-$ are two markings of $D$, we define $(S_D^+,S_D^-)$-Cartesian fibrations as the intersection of $S_D^+$-co-Cartesian fibrations and $S_D^-$-Cartesian fibrations, and similarly for $(S_D^+,S_D^-)$-Cartesian transformations. Thus with this terminology, for a marked category $(D,S_D)$ the word `bi-Cartesian' is an abbreviation for `$(D,S_D)$-Cartesian'.
\end{defn}

\begin{prop}[Bivariant Grothendieck construction]
\label{BIV_BICART_GROT}

The covariant Grothendieck (1-)equivalence $1\coCart_D\cong 1\Fun(D,1\Cat)$ matches the subcategories $\Biv^\pm_D$ and $\biCart_D$.
\end{prop}
\begin{proof}
We will check on objects and morphisms. 

\begin{labelitems}
\item[Objects]  
\begin{comment}
If $p:E\rightarrow D$ has both co-Cartesian and Cartesian lifts for an arrow $f:x\rightarrow y$ in $D$, then the associated pullback functor is right adjoint to the pushforward functor with adjunction isomorphism given by following the bottom-left path around the square
\[\xymatrix{
  E_x(-,f^!z) \ar[r] \ar@{=}[d]
    & E_y(f_!-,f_!f^!z) \ar[d] \\
  E_f(-,z) \ar@{=}[r] 
    &  E_y(f_!-,z)
}\]
where the isomorphisms are provided by the universal properties of Cartesian, resp.\ co-Cartesian lifts. Here $E_f$ denotes the pullback of $E$ along the map $\Delta^1\rightarrow D$ representing $f$. Following instead the top-right path, we see that the counit of this adjunction is given by factorising the Cartesian arrow $f^z:f^!z\rightarrow z$ through the co-Cartesian $f_{f^!z}:f^!z\rightarrow f_!f^!z$. (See \ref{CAT_ADJ_1CAT} for a discussion of notions of adjunction.) In particular, the covariant functor classified by a bi-Cartesian fibration is right bivariant.
\end{comment}
Let $p:E\rightarrow D$ be a co-Cartesian fibration.
By Lurie's definition of adjunction, $f_!$ admits a right adjoint for every morphism $f:x\rightarrow y$ in $D$ if and only if $p$ is a locally Cartesian fibration. Hence, we are reduced to showing that in this case, $p$ is even a Cartesian fibration.

Suppose $f_!$ admits a right adjoint $(f_!)^*$ with counit $\epsilon$. Then for any $g:w\rightarrow x$ in $D$
we get a commutative diagram
\[\xymatrix{
  E_{g}(-,(f_!)^*z) \ar[r]^-{f_{(f_!)^*z}\circ-} \ar@{=}[d]
    & E_{fg}(-,f_!(f_!)^*z) \ar[r]^-{-\circ\epsilon} \ar[d]
      & E_{fg}(-,z) \\
  E_x(g_!-,(f_!)^*z) \ar@{=}[r] 
    &  E_y(f_!g_!-,z) \ar@{=}[ur] 
}\] 
where, as usual, the unfortunately rather Baroque symbol $f_{(f_!)^*z}$ stands for the co-Cartesian lift of $f$ starting at $(f_!)^*z$.
This map therefore satisfies the universal property of a Cartesian lift of $f$. (Note that $p$ really needs to be a co-Cartesian fibration at least over $D\downarrow x$: the argument fails if we only assume that $f$ has co-Cartesian lifts.)

\item[Morphisms]
Let $E,\,E':\biCart_D$, and let $\phi:E\rightarrow E^\prime$ be a co-Cartesian transformation. By lemma \ref{BIV_BICART_MATE} applied to each $\Delta^1\rightarrow D$, $\phi$ is a bi-Cartesian transformation if and only if the associated natural transformation is bivariant.\qedhere
\end{labelitems}
\end{proof}

\begin{remark}[Opposite version]
\label{BIV_BICART_OPP}

By a dual argument we can identify the $D,(S_D,D)$ category $\Cart_{D,(S_D,D)}$ of $(S_D,D)$-Cartesian fibrations over $D$, as a subcategory of $\Cart_D$, with $\Biv^{\mp\ell}(D^\op,1\Cat)$.
\end{remark}

\begin{defn}[Beck-Chevalley condition for fibrations]
\label{BIV_BICART_BC}

Let $p:E\rightarrow D$ be a 1-functor and let
\[\xymatrix{
w \ar[r]^{\tilde g} \ar[d]_{\tilde f} & x \ar[d]^f \\
y \ar[r]^g & z
}\]
be a square in $D$. Suppose that $f$, $\tilde f$ admit Cartesian lifts and $g$, $\tilde g$ admit co-Cartesian lifts to $E$. For each $u:E_y$ we obtain a natural map 
\[ \nu_u: \tilde g_!\tilde f^!u\rightarrow f^!g_!u \]
by factorising $\tilde f^!u\rightarrow u\rightarrow g_!u$ through the co-Cartesian arrow $\tilde f^!u\rightarrow \tilde g_!\tilde f^!u$ and the Cartesian arrow $f^!g_!u\rightarrow g_!u$. We say that $E$ satisfies the \emph{Beck-Chevalley condition} on this square if $\nu_u$ is invertible.
\end{defn}

\begin{lemma}
When $f$ also admits co-Cartesian lifts, $\nu_u$ is homotopic to the action on $u$ of the conjugate 2-cell obtained from the adjunction $f_!\dashv f^!$.
\end{lemma}
\begin{proof}
By lemma \ref{BIV_BICART_MATE} applied to $\phi=g_!$.
\end{proof}

\begin{defn}[Base change for fibrations]
\label{BIV_BICART_BASECHANGE}

Let $S_D^\pm$ be two markings of $D$, and suppose that fibre products of elements of $S_D^-$ with elements of $S_D^+$ are representable; call the triple $(D,S_D^\pm)$ a \emph{twice marked category with base change}.

Let $p:E\rightarrow D$ be an $(S_D^+,S_D^-)$-Cartesian fibration. We say that $E$ \emph{has base change} if it satisfies the Beck-Chevalley condition for each square formed as a fibre product of a member of $S_D^+$ by one of $S_D^-$. The full subcategory of $\biCart_{D\sep S_D}$, resp.~$\Cart_{D\sep (S_D^+,S_D^-)}$ whose objects have base change is denoted $\biCart_{D\sep S_D}^r$, resp.~$\Cart_{D\sep (S_D^+,S_D^-)}^r$.
\end{defn}

\begin{remark}
\label{BIV_BICART_CC}
The dual \emph{collar change condition} for a pair of markings with collar change is, of course, formed by replacing fibre product with pushout squares in the preceding definition.
\end{remark}

\begin{prop}
A bi-Cartesian (that is, $(D,S_D^-)$-Cartesian) fibration has base change if and only if the right bivariant functor it classifies has base change. Dually, an $(S_D^+,D)$-Cartesian fibration has base change if and only if the left bivariant functor it classifies has collar change.
\end{prop}

\begin{eg}
\label{BIV_BICART_EX}

When $D:1\Cat$ admits fibre products, the target projection $D^{\Delta^1}\rightarrow D$ is a bi-Cartesian fibration with base change.
\end{eg}

\subsection{Free bivariant fibrations}

In this section we show that the presence of base change permits a convenient theory of free fibrations, after \S\ref{CAT_FIB}: to wit, the free Cartesian fibration on a co-Cartesian fibration is free bi-Cartesian with base change.

\begin{defn}[Free fibrations]
\label{BIV_BICART_FREEDEF}

A bi-Cartesian fibration with base change $E\rightarrow D$ is said to be \emph{$n$-free} (as a bi-Cartesian fibration with base change) on a subcategory $E_0\subseteq E$ if restriction induces an equivalence of $(n-1)$-categories
\[ \biCart_D(E,F) \cong 1\Fun_D(E_0,F) \]
for any $F:\biCart_D$, where $n:\{1,2\}$. Compare Definition \ref{CAT_FIB_FREEDEF}. (We will see that as for Cartesian fibrations, $2$-free fibrations always exist, and that $1$-free and $2$-free are equivalent properties.)\end{defn}

\begin{prop}[Bivariant path category]
\label{BIV_BICART_FREE}

Let $(D;\,S_D)$ be a marked category with base change.  
The free co-Cartesian fibration adjunction \eqref{CAT_FIB_FREE} $1\Cat\downarrow D \rightleftarrows 1\coCart_D$ restricts to an adjunction between $\Cart_{D;\,S_D}$ and $\biCart^r_{D;\,S_D}$. 
In particular, a free co-Cartesian fibration on a free $S$-Cartesian fibration is free $S$-bi-Cartesian with base change.

Moreover, 1-free bi-Cartesian fibrations with base change are 2-free.
\end{prop}
\begin{proof}
We will use the criteria of Lemma \ref{CAT_ADJ_SUBCAT}. Clearly, the restriction of the right adjoint inclusion maps $\biCart^r_{D;\,S_D}$ into $\Cart_{D;\,S_D}$. For the rest, it is enough to check the following:
\begin{enumerate}
  \item For any $S_D$-Cartesian fibration $p:E\rightarrow D$, $E\downarrow D$ is $S_D$-Cartesian;
  \item the inclusion $j_E:E\rightarrow E\downarrow D$ is an $S_D$-Cartesian transformation; 
  \item a $p$-left Kan extension along $j_E$ of an $S_D$-Cartesian transformation remains $S_D$-Cartesian. 
\end{enumerate}
Indeed, the first condition implies that the left adjoint preserves objects, while the other two ensure that the adjunction equivalence --- which is given by (inverse) $p$-left Kan extension and restriction along $j_E$ --- restricts to an equivalence $\biCart^r_{D;\,S_D}(LE,LF)\cong \Cart_{D;\,S_D}(E,RLF)$.

\begin{lemma}[Objects]
\label{BIV_BICART_FREE_OB}
  
  Let $(D,S_D^\pm)$ be a twice marked category with base change, and let $p:E\rightarrow  D$ be $S_D^-$-Cartesian. The free $S_D^+$-co-Cartesian fibration $E\downarrow^\sharp (D,S_D^-) \stackrel{\rm{tar}}{\rightarrow} D$ is $(S_D^+, S_D^-)$-Cartesian with base change, and the canonical map $E\rightarrow E\downarrow^\sharp (D, S_D^-)$ is an $S_D^-$-Cartesian transformation.

\end{lemma}
\begin{proof}
\begin{labelitems}
\item[Bivariance]
That $E\downarrow^\sharp(D,S_D^+)$ is $S_D^-$-Cartesian follows by decomposing it as follows: 
\[\xymatrix{
	E\downarrow^\sharp(D,S_D^+) \ar[r]^-{\hat p} \ar[d] \ar@{}[dr]|(0.35){\lrcorner} & D^{\Delta^\sharp} \ar[r]^{\mathrm{tar}}
    \ar[d]^{\mathrm{src}} & D \\
  E \ar[r]_p & D
}\]
Since elements of $S_D^+$ admit pullbacks by elements of $S_D^-$, the target fibration $(D,S_D^+)^{\Delta^\sharp}\rightarrow D$ is $S_D^-$-Cartesian (with Cartesian arrows given by pullback squares); meanwhile, $\hat p$ is a pullback of the $S_D^-$-Cartesian fibration $p$ (with Cartesian arrows those having $p$-Cartesian image in $E$ \cite[Prop.~2.1.4.3.(2)]{HTT}).

\item[Cartesian arrows]
From the description of Cartesian arrows in this composite, we find that for $f:\alpha \rightarrow \beta$ in $S_D^-$, the pullback in $E\downarrow^\sharp (D, S_D^+)$ of an object $(x,\, px\rightarrow\beta)$ is represented by the diagram 
\[\xymatrix{
	\tilde f^!x\ar@{|->}[r]\ar[d]
		&  \alpha\times_\beta p x\ar[d]_{S_D^-}^{\tilde f} \ar[r]|-{S_D^+}
		& \alpha\ar[d]^f_{S_D^-} \\
	x\ar@{|->}[r] 	& px\ar[r]|{S_D^+} 	& \beta
}\]
where the right-hand square is a pullback in $D$ and the left is $p$-Cartesian. 

\item[Canonical map] In particular, a $p$-Cartesian arrow $f^x:f^!x\rightarrow x$ in $E$ (over $f:\alpha\rightarrow\beta$) maps, under the canonical map, to a diagram
  \[\xymatrix{
    \tilde f^!x\ar@{|->}[r]\ar[d]
      &  \alpha\ar[d]^f \ar@{=}[r] & \alpha\ar[d]^f \\
    x\ar@{|->}[r] 	& \beta \ar@{=}[r] 	& \beta
  }\]
  which is certainly Cartesian by the preceding criterion.

\item[Base change] Let $[g:\gamma\rightarrow\beta]\in S_D^+$, $y:E\downarrow^\sharp\{\gamma\}$. Using the above description, we check the Beck-Chevalley condition for the fibre product of $g$ and $f$ by observing that in the composite diagram
\[\xymatrix{
  \tilde{f}^!y\ar@{|->}[r]\ar[d] &  
    \alpha\times_\beta p y \ar[d]^{\tilde f} \ar[r] & 
    \alpha\times_\beta\gamma \ar[r]^-g \ar[d]^f & 
    \alpha\ar[d]^f \\
  y\ar@{|->}[r] & py \ar[r] & 
    \gamma\ar[r]^g & \beta
}\]
presenting $f^!g_!(y,\,py\rightarrow \gamma)$, the middle square is Cartesian, whence it is also presents $g_!f^!(y,\,py\rightarrow \gamma)$. \qedhere
\end{labelitems}
\end{proof}

\begin{lemma}[Morphisms]
\label{BIV_BICART_FREE_MOR}

Let $F\rightarrow D$ be $(S_D^+,S_D^-)$-Cartesian with base change and let $\phi:E\rightarrow F$ be an $S_D^-$-Cartesian transformation. Then the relative left Kan extension $\tilde\phi:E\downarrow D\rightarrow F$ of $\phi$ is an $S_D^-$-Cartesian transformation.
\end{lemma}
\begin{proof}
In the notation established above:
\begin{align*}
\tilde\phi[\tilde f^!x,\alpha\times_\beta px	\stackrel{\tilde g}{\rightarrow}\alpha] 
  & = \tilde\phi \tilde g_!\tilde f^!x 	& \text{formula for }g_! \\
  &= \tilde g_!\phi\tilde f^!x 	& \tilde\phi~\text{co-Cartesian} \\
  &= \tilde g_!\tilde f^!\phi x 	& \phi~\text{Cartesian} \\
  & = f^! g_!\phi x 				& \text{base change in }F \\
  &= f^!\tilde\phi g_!x 			& \tilde\phi~\text{co-Cartesian} \\
  &= f^!\tilde\phi[x,px\stackrel{g}{\rightarrow}\beta] 
									& \text{formula for }g_! &  \qedhere
\end{align*}
\end{proof}
 To complete the proof of Proposition \ref{BIV_BICART_FREE}: 1-free is equivalent to 2-free by the corresponding statements for $\coCart_D$ and $\Cart_{D,\,S_D}$.
\end{proof}

\begin{remark}[Collar change]

Dual statements hold for bivariant functors with the collar change property \ref{BIV_BICART_CC}: the free Cartesian fibration on a co-Cartesian fibration has the collar change property.
\end{remark}

\subsection{Twisted Cartesian fibrations}
\label{BIV_TWIST}

This short section concerns a generalisation of the notion of `bifibration' studied in \cite[\S2.4.7]{HTT} and \cite[\S4]{AF_fibrations}. To minimise the potential for confusion of this concept with that of bi-\emph{Cartesian} fibration studied in \S\ref{BIV_BICART}, we abandon the terminology of \emph{op.~cit.}\ in favour of the more suggestive \emph{twisted Cartesian}  fibration. We need a notion of functor classified by a \emph{twisted bi-Cartesian} fibration to construct the bivariant Yoneda embedding.

\begin{defn}
\label{BIV_TWIST_DEF}

Let $C,\,D:\Cat^+$. A 1-functor $p:E\rightarrow C\times D$ is said to be a \emph{twisted bi-Cartesian fibration} if the following conditions are satisfied:
\begin{enumerate}
\item The composite $E\rightarrow C$ is a $(S_C,C)$-Cartesian fibration and $p$ maps Cartesian and $S_C$-co-Cartesian arrows to $D$-slices in $C\times D$.
\item For each $c:C$, $p_c:E_c\rightarrow D$ is a $(D,S_D)$-Cartesian (alias bi-Cartesian) fibration.
\item For each $f:c'\rightarrow c$ in $C$, the pullback functor
\[\xymatrix{
E_c \ar[rr]^{f^!}\ar[dr] && E_{c'}\ar[dl] \\ & D
}\] 
is a transformation of bi-Cartesian fibrations.
\item For each $g:c\rightarrow c"$ in $S_C$, the pushforward functor
\[\xymatrix{
E_c \ar[rr]^{g_!}\ar[dr] && E_{c"}\ar[dl] \\ & D
}\] 
is a transformation of bi-Cartesian fibrations.
\end{enumerate}
\end{defn}

\begin{para}[Twisted fibre transport functor]
\label{BIV_TWIST_GROT}

 Let $p:E\rightarrow C\times D$ be a twisted bi-Cartesian fibration. By condition (i), $p$ is a morphism in the category of $(S_C,C)$-Cartesian fibrations over $C$. Thus, contravariant Grothendieck integration turns it into a natural transformation $\fibre_Cp:\fibre_CE \rightarrow \underline D$ of left bivariant functors $C^\op\rightarrow 1\Cat$, where $\underline D$ denotes the constant functor with value $D$.% \ref{BIV_BICART_OPP}. 

Now, running $\fibre_Cp$ through a little bit of currying:
\begin{align*}
  \Biv^\mp_{C^\op}(\fibre_CE, \underline{D}) \quad & 
    \cong\quad 1\Fun(\Delta^1, \Biv^\mp(C^\op,1\Cat) )_{\rm{src}=\fibre_CE,~\rm{tar}=\underline D} \\
  & \cong\quad \Biv^\mp\left(C^\op, 1\Fun(\Delta^1, 1\Cat)_{\rm{src}=E_c,~\rm{tar}=D} \right)
\end{align*}
we obtain a functor $C^\op\rightarrow 1\Cat\downarrow D$. (In these formulas, the subscripts mean that we restrict to a fibre of the source/target projection of the mapping category.) 

This functor maps $c:C$ to $p_c:E_c\rightarrow D$ and $f:c'\rightarrow c$ to $f^!:E_c\rightarrow E_{c'}$, so by conditions ii) and iii) of Def.\ \ref{BIV_TWIST_DEF} it factors through the subcategory of $(D,S_D)$-Cartesian fibrations. Condition iv) ensures that the left adjoints to members of $S_C$ also belong to $\biCart_D$. Finally, take the fibre transport functor \ref{BIV_BICART_GROT} to obtain a left bivariant functor $C^\op\rightarrow\rm{RBiv}(D,1\Cat)$. In the syntax of \ref{BIV_SYNTAX}:
\begin{align*}
\alignhead{Twisted bivariant fibre transport functor}
c:C^\op\ (\mp) \quad \return \quad d:D\ (\pm) \quad &\return \quad \fibre_CE(c,d):1\Cat \\
\Rightarrow\hspace{3em} c:C^\op\ (\mp),\ d:D\ (\pm) \quad &\return \quad \fibre_CE(c,d):1\Cat
\end{align*}
\end{para}

\begin{para}[Base change]
\label{BIV_TWIST_BC}

If $p:E\rightarrow C\times D$ is twisted bi-Cartesian and satisfies base change in each variable --- that is, $E\rightarrow C$ is $(S_C,C)$-Cartesian with base change and for each $c:C$, $p_c:E_c\rightarrow D$ is $(D,S_D)$-Cartesian with base change --- then by Proposition \ref{BIV_BICART_BC} the output is a bivariant bifunctor with base change in the second variable and collar change in the first. The formula is:
\begin{align*}
\alignhead{Twisted bivariant fibre transport functor, base change}
c:C^\op\ (\mp,\, \ell),\ d:D\ (\pm,\, r) \quad &\return \quad \fibre_CE(c,d):1\Cat
\end{align*} 
\end{para}

\begin{remark}[Symmetry]
\label{BIV_TWISTED_SYM}

Our definition and construction is asymmetric: it is optimised for taking fibres over $C$ first and $D$ second. The reader may well imagine a dual approach where we proceed in the opposite order. We would then be left with the task of showing that the two constructions agree. Ultimately, I would prefer a more symmetric definition and construction.

A symmetric definition runs as follows: let $S^+$, resp.~ $S^-$ be the set of morphisms in $C\times D$ that project to isomorphisms in $C$, resp.~$D$. Then one can show that $p:E\rightarrow C\times D$ is twisted Cartesian iff it is $(S^+,S^-)$-Cartesian with base change in the sense of Def.\ \ref{BIV_BICART_BC}. A symmetric construction of the fibre transport functor appears in the proof of \cite[Lemma 4.1]{AF_fibrations}, but the method only works in the stated form for twisted fibrations in spaces, not in 1-categories.
\end{remark}

\subsection{Bivariant Yoneda lemma}
\label{BIV_YON}

In this section we combine our previous efforts to achieve what was promised in the title of this paper: a bivariant Yoneda embedding, an associated notion of \emph{representable} bivariant functor, and a Yoneda lemma analogous to the extended Yoneda lemma of \S\ref{CAT_EVAL_YON}. From hereon in, marked categories and bivariant functors are assumed to have base change unless otherwise indicated.

\begin{para}[Spans]
\label{BIV_YON_SPANS}

Denote by $\Lambda$ the \emph{walking span}, which is the free category on the picture 
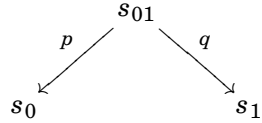
\[\xymatrix{&s_{01}\ar[dl]_p\ar[dr]^q \\ s_0&& s_1}\] 
considered as a marked category with $S_\Lambda=\{p\}$. A \emph{correspondence} or \emph{span} in $D:1\Cat^+$ is a marked functor $\Lambda\rightarrow D$, i.e.~ a diagram in $D$ of the above shape whose wrong-way map is marked. The 1-category of spans in $D$ is $D^\Lambda=1\Fun^+(\Lambda,D)$. The fibre $\Cor_D(x,y)$ of the projection $(\rm{pr}_0,\rm{pr}_1):D^\Lambda\rightarrow D\times D$ over $(x,y):D^2$ is called the \emph{category of spans from $x$ to $y$}. We will sometimes abbreviate the data of a span in $D$ by $(p,q)$.
\end{para}

\begin{para}[Free property of spans]
\label{BIV_YON_FREE}

By construction, the category  of spans in $ D$ fits into a pullback square 
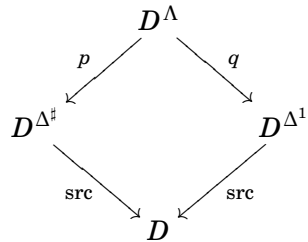
\[\xymatrix{ 
		& D^\Lambda\ar[dl]_p\ar[dr]^q \\
	D^{\Delta^\sharp}\ar[dr]_{\rm{src}} 	&& D^{\Delta^1}\ar[dl]^{\rm{src}} \\
		&  D
}\] 
where $D^{\Delta^\sharp}$ stands, as before, for the full subcategory of $D^{\Delta^1}$ spanned by the marked arrows \eqref{CAT_MARK_TRIV}. In the notation of \ref{CAT_FIB_SLICE}, $D^\Lambda=(D\downarrow^\sharp D) \downarrow D$ with the outer comma category taken on the source projection of $D\downarrow^\sharp D$.

\begin{cor0}[of Prop.\ \ref{BIV_BICART_FREE}] The second projection $\rm{pr}_1:D^\Lambda\rightarrow D$ is a free bi-Cartesian fibration with base change on the diagonal subcategory $D\subseteq D^\Lambda$. 
\end{cor0}
\noindent More generally, if $E:1\Cat\downarrow D$ then $E\times_DD^\Lambda \cong (D\downarrow^\sharp E) \downarrow D$ is free bi-Cartesian with base change on $E$, where the fibre product is taken over the $s_0$ (left) evaluation $D^\Lambda\rightarrow D$.
\end{para}

\begin{lemma}[Universal span]
\label{BIV_YON_UNIV}

The left and right target projections $(\rm{pr}_0,\rm{pr}_1):D^\Lambda\rightarrow D\times D$ set up the category of spans in $D$ as a twisted bi-Cartesian fibration with base change in both variables.
\end{lemma}
\begin{proof}
By another application of Proposition \ref{BIV_BICART_FREE}, $\rm{pr}_0:D^\Lambda\rightarrow D$ is (free) $(S_D,D)$-Cartesian with base change. The pullback along $x'\rightarrow x $, resp.~pushforward along $x\rightarrow x"$, as a functor of $[x\leftarrow w\rightarrow y]$ given by the constructions:
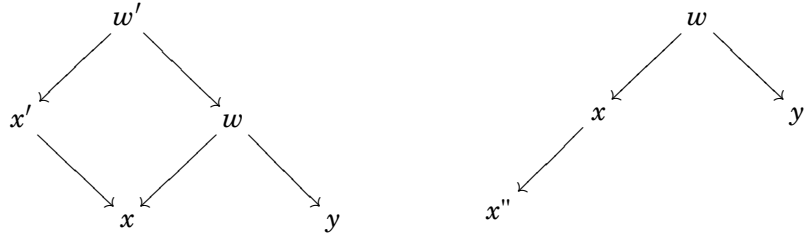
\[  \xymatrix{
		&w'\ar[dr]\ar[dl] %\ar@{}[dd]|(.2){\displaystyle\diamond}
	\\ x'\ar[dr] && w\ar[dl]\ar[dr] 
	\\ & x 	&& y
} \mskip90mu \xymatrix{ 
			&&w\ar[dr]\ar[dl] 
	\\ 	&x\ar[dl] && y 
	\\ x" 
}\]
where the diamond in the left diagram is a pullback. We must check ii)-iv) of Definition \ref{BIV_TWIST_DEF}.

\begin{labelitems}
\item[ii)] Let $x:D$. The projection $\rm{pr}_1:D^{{}_x\Lambda}\cong (D\downarrow^\sharp x)\downarrow D\rightarrow D$ is is free bi-Cartesian with base change, again by \ref{BIV_BICART_FREE}. This also confirms the base change criteria \ref{BIV_TWIST_BC}.

\item[iii), iv)] Inspecting the diagrams above, for $x'\rightarrow x$, resp.~ $x\stackrel{S_D}{\rightarrow} x"$, the pullback, resp.~ pushforward functor for $\rm{pr}_0$ is got by applying $-\downarrow D$ to pullback $D\downarrow^\sharp x\rightarrow D\downarrow^\sharp x'$, resp.~postcomposition $D\downarrow^\sharp x\rightarrow D\downarrow^\sharp x"$, which are themselves $S_D$-Cartesian transformations. By Lemma \ref{BIV_BICART_FREE_MOR}, these are therefore $S_D$-bi-Cartesian transformations.
\qedhere
\end{labelitems}
\end{proof}

\begin{para}[Yoneda embedding]
\label{BIV_YON_FUNCTOR}

The universal span being a twisted bi-Cartesian fibration with base change in both variables \eqref{BIV_YON_UNIV}, via \ref{BIV_TWIST_GROT} it classifies a bifunctor which through currying gives us the \emph{Yoneda embedding}:
\begin{align*}
  && x:D^\op\ (\mp,\ell),\ y:D\ (\pm,r) \quad \return_{1\Cat} \quad 
    &\Cor_D(x,y):1\Cat \\
\Rightarrow&& x:D^\op\ (\mp,\ell) \quad \return_{1\Cat} \quad
    & \rm y_D^\mp(x):\Biv_D \mskip150mu \\
\alignhead{Bivariant op-Yoneda embedding}
\Rightarrow&& x:D\ (\pm,r) \quad \return_{1\Cat} \quad
    & \rm y_D^\mp(x):\Biv_D^\op \mskip150mu %\\
%\Rightarrow && & \rm{y}_D^\mp:D\rightarrow \Biv_D^\op. 
\end{align*}
We will see in \S\ref{BIV_LOC} that unlike the usual case, $\rm y^\mp_D$ is not fully faithful, though it is monic.
\end{para}

\begin{remark}[Conjugate formulas]

Currying in the opposite order gives us:
\begin{align*}
\alignhead{Bivariant Yoneda embedding}
 y:D\ (\pm,r) \quad\return\quad \rm{y}_D^\pm(y):\Biv^{\mp,\ell}_{D^\op}.
\end{align*}
This Yoneda embedding is just a little less convenient to use within our policy of prioritising right bivariant functors with base change over their opposite versions.
\end{remark}

\begin{para}[Evaluation map for bivariant functors]
\label{BIV_YON_EVAL}
We now turn to the definition of the bivariant Yoneda evaluation mapping:
\begin{align*}
\alignhead{Bivariant Yoneda evaluation map}
&&x:D\ (\pm,r),\ H:\Biv_D \quad 
  &\return_{1\Cat} \quad \phi:\Biv_D(\rm y^\mp_Dx,H) \quad 
    \return_{1\Cat} \quad \phi(\iden_x):H(x) \\
\Rightarrow&& x:D\ (\pm,r),\ H:\Biv_D \quad
  &\return_{1\Cat} \quad \rm{ev}:\Biv_D(\rm y^\mp_Dx,H) \quad \underset{1\Cat}{\rightarrow} \quad H
\end{align*}
which proceeds along the lines of \ref{CAT_EVAL_YON}. Thus, we use the ordinary Grothendieck construction to pull the formula over to the category of co-Cartesian fibrations over $D\times\Biv_D$, and the bivariant Grothendieck construction \eqref{BIV_BICART_GROT} (in the version that provides a 2-equivalence) 
to exchange bivariant functors with bi-Cartesian fibrations:
\begin{align*} \int_{x:D}\Biv_D(\rm y^\mp_Dx,E) \quad 
   &\cong\quad  \int_{x:D}\biCart_D(\rm{free}^\mp_Dx, E)
\intertext{\indent Now, in terms of co-Cartesian fibrations over $D\times \biCart_D$:}
\alignhead{Representable bi-Cartesian fibration evaluation functor}
  \int_{x:D,\, E:\biCart_D}\biCart_D(\rm{free}^\mp_Dx, E) \quad &
    \to \quad  \left(\int_{x:D,\, E:\biCart_D} \biCart_D(\rm{free}^\mp_Dx, E) \right) \times_D D^{\Lambda} 
 \\
   & \acong \quad \left(\int_{x:D,\, E:\biCart_D} \biCart_D(\rm{free}^\mp_Dx, E) \right) \times_D \left( \int_{x:D}\rm{free}^\mp x \right) 
 \\
  &\rightarrow\quad \left(\int_{X,\, E:\biCart_D} \biCart_D(X, E) \right) \times_D \left( \int_{X:\biCart_D}X \right)  
\\
  &\stackrel{\rm{ev}}{\to}\quad \int_{E:\biCart_D} E;
\end{align*}
the last line is the restriction of the universal evaluation map \eqref{CAT_EVAL_REL} to the subcategory of bi-Cartesian fibrations.
Inspecting the formulas, for fixed $x$ and $E$ this map can be identified with evaluation $\phi\mapsto \phi(x)$ at the tautological element $x:\rm{free}^\mp_Dx$.
\end{para}

%NOTE. This argument may appear less rigorous than some of the others. Actually all that is missing is that the above diagrams are functors of $[x\leftarrow w \rightarrow y]$ (fixed $x$). Proof that they are actaully pullback/pushforward functors is standard enough to omit.

\begin{lemma}[Bivariant Yoneda lemma]
\label{BIV_YON_LEMMA}

The bivariant Yoneda evaluation mapping is an equivalence of bivariant functors.
\end{lemma}
\begin{proof}
For fixed $x$, $H$, the evaluation map is isomorphic to the expected evaluation on the identity span at $x$. Now apply the 2-free property of $D^\Lambda$, Prop.~\ref{BIV_BICART_FREE}.
\end{proof}

\begin{defn}[Representable bivariant functors]
\label{BIV_YON_REP}

A bivariant functor $D\rightarrow1\Cat$ is said to be \emph{representable} if it is isomorphic to one of the form $\Cor_\rm D(x,-)$. The \emph{Yoneda image} $\rm y^\mp(D)$ of $D$ is the full 2-subcategory of $\Biv_D$ spanned by the representable bivariant functors.
\end{defn}

\begin{remark}

The dual notion of universal cospan yields a universal family of bivariant functors with the opposite parities as follows:
\[ x:D\ (\pm,\ell),\ y:D^\op\ (\mp,r) \quad \return_{1\Cat} \quad 
    \rm{Bord}_{D\sep S_D}(x,y):1\Cat; \]
here $\rm{Bord}_{D\sep S_D}(x,y)$ is the category of marked functors $\Lambda^\op\rightarrow D$ with endpoints at $x$, $y$.
\end{remark}

\subsection{Local representation of spans}
\label{BIV_LOC}

A bivariant functor $D\rightarrow K$ induces, for each $x,y:D$, a 1-functor $\Cor_D(x,y)\rightarrow K(Hx,Hy)$ that sends $(p,q)$ to $q_!p^!$. It is this \emph{local representation} of spans in $D$ on morphisms in $K$ to which we now turn.

\begin{para}[Local representation of spans]
\label{BIV_LOC_DEF}

Let $H:D\rightarrow K$ be bivariant. For each $\alpha:K$ we can define a new $1\Cat$-valued bivariant functor $H^\dagger\alpha:\Biv_D$ by the formula
\begin{align*}
&& y:D\ (\pm,r),\ H:\Biv(D,K),\ \alpha:K \quad &
    \return\quad K(\alpha, Hy)
    \intertext{whence}
\alignhead{Bivariant Yoneda pullback}
&& H:\Biv(D,K),\ \alpha:K \quad &
    \return\quad H^\dagger\alpha:\Biv_D. \mskip120mu
 \end{align*}
Fixing $x:D$, there is now a canonical element $\iden_{Hx}:H^\dagger Hx$ corresponding to the identity of $K(Hx,Hx)$. Passing this through the bivariant Yoneda isomorphism \ref{BIV_YON_LEMMA}:
\[ H^\dagger Hx(x) \cong \Biv_D( \rm y^\mp_Dx, H^\dagger Hx) \] 
then yields a \emph{local representation of spans}
\begin{align*}
\alignhead{Local representation of spans}
 y: D\ (\pm), \quad
    \return &\quad H^\dagger Hx (y) \\
   & \quad H_!:\Cor_D(x,-)\rightarrow K(Hx,H-) 
\end{align*}

Let us calculate the value of $H_!$ on a fixed span $(p,q)$ based at $x:D$. By the definition \eqref{BIV_YON_EVAL} of the Yoneda evaluation map, $H_!$ corresponds via Grothendieck integration to the unique map of bi-Cartesian fibrations
  \[  \rm{free}^\mp_D(x) \rightarrow \int_{y:D}K(Hx,Hy) \]
that sends the canonical element $x:\rm{free}^\mp_D(x)$ to $\iden_{Hx}:\int_{y:D}K(Hx,Hy)$. Hence, its value at a span $[x\stackrel{p}{\leftarrow} w\stackrel{q}{\rightarrow} y]$ is calculated by pulling $\iden_{Hx}$ along $p$ then pushing along $q$. Now, because this fibration arises by Grothendieck integration of a functor, pushforward and pullback along a morphism $r$ are computed by the image of $r$ and its right adjoint with respect to that functor. In $\int_y K(Hx,Hy)$ these are the functors of postcomposition with $H(r)=r_!$, resp.~with its right adjoint $r^!$ .
\[\xymatrix{
	K(Hx,Hx) \ar[r]^-{p^!\circ-}\ar@{|->}[d] 	
		& K(Hx,Hw) \ar[r]^-{q_!\circ-} \ar@{|->}[d] 
		& K(Hx,Hy) \ar@{|->}[d] 
	\\
	x	& w \ar[l]_p \ar[r]^q	& y
}\] 
Applying this to the canonical element $x$, we find $H_!(p,q)=\iden_x\circ q_!\circ p^!$. Hence we are justified in making $H_!$ anonymous with the sentence:
  \begin{align*}
\alignhead{Local representation of spans (anonymous format)}
 y: D\ (\pm) \quad &
    \return_{1\Cat}\quad (p,q):\Cor_D(x,y)\quad\return_{1\Cat} \quad  q_!p^!:K(Hx,Hy).
\end{align*}\end{para}

\begin{para}[Functoriality in $K$ of the local representation]
\label{BIV_LOC_FUNC}

The construction of the local representation can be made to play well with a 2-functor $G:K\rightarrow K'$:
\begin{enumerate}
\item $G$ induces a natural transformation of bifunctors $K(H-,H-)\rightarrow K^\prime(GH-,GH-)$ and hence a bivariant natural transformation $H^\dagger Hx\rightarrow (GH)^\dagger GHx$.
\item Using the naturality of Yoneda evaluation \eqref{BIV_YON_EVAL}, postcomposition with $G$
\[  \Biv(\rm y^\mp_D x, H) \rightarrow \Biv(\rm y^\mp_D x, GH) \]
sends $H_!$ to $(GH)_!$, i.e.~we get a commuting triangle
\[\xymatrix{
  \Cor_D(x,y) \ar[r]\ar[dr] & K(Hx,Hy)\ar[d] \\ & K'(GHx,GHy)
}\]
functorial in $y$.
\end{enumerate} 
It is possible to say even more: for each $x,y:D$ the local representation forms a cone over the 1-functor 
\[ \bf{Biv}_D\rightarrow 1\Cat, \qquad [H:D\rightarrow K]\return K(Hx,Hy) \] 
where $\bf{Biv}_D$, as in \ref{BIV_UNIV_1CAT}, stands for the 1-category of bivariant extensions of $D$. As the derivation is a little involved, I omit a complete formulation of this last statement.
\end{para}

\begin{defn}[Correspondences]
\label{BIV_LOC_CORR}

Let $D:1\Cat^+$. A bivariant functor $H:D\rightarrow  K$ is said to exhibit $K$ as a \emph{2-category of correspondences of $D$} if it is essentially surjective and induces, via \ref{BIV_LOC_DEF}, an equivalence of categories \[ \Cor_D(x,y)\tilde\rightarrow K(Hx,Hy)\]for each $x,y:D$, or what is the same thing, an equivalence of $1\Cat$-valued bifunctors in $x,y$.

If $\Cor_D:\bf{Biv}_D$ is a category of correspondences for $D$, then the expression $\Cor_D(x,y)$ may then unambiguously be parsed either following \ref{BIV_YON_SPANS} or as the category of morphisms from $x$ to $y$ in $\Cor_D$. Hence, denoting any fixed category of correspondences as $\Cor_D$ cannot cause confusion in this way.
\end{defn}

\begin{remark}[Segal 2-category of spans]
We expect that the Segal 2-category $\rm{Span}_{D,\bullet}$ defined in \cite[\S12]{HigherSegal} is a 2-category of correspondences for $D$. To check this, we must calculate the induced representation
\[ \Cor_D(x,y) \rightarrow \rm{Span}_D(x,y) \]
which we expect to be the identity (these two categories are, in fact, defined in exactly the same way). The induced representation is easiest to characterise as a 1-natural transformation in $y$. Unfortunately, the construction of these mapping categories as a 1-functor of the codomain is not so elementary in the Segal model, and so I will not attempt a rigorous calculation. (On the other hand, if we accept \cite[Thm.~8.1.1.9]{GR}, the claim is an a posteriori consequence of the universal property of $\rm{Span}_{D,\bullet}$.)
\end{remark}

\begin{cor}
\label{BIV_LOC_IMAGE}

The bivariant Yoneda image is a 2-category of correspondences.\end{cor}
\begin{proof}
Apply the bivariant Yoneda lemma \ref{BIV_YON_LEMMA} to $H=\Cor_D(-,y)$ for each $y:D$.
\end{proof}

\begin{eg}[Functors between correspondence categories]
\label{BIV_LOC_EX}

Let $F:D\rightarrow E$ be a marked functor with base change, and let $\Cor_E$ be a category of correspondences for $E$. Then $D\rightarrow\Cor_E$ is bivariant, and so for each $x,y:D$ the construction \ref{BIV_LOC_DEF} provides a 1-functor
\[ \Cor_D(x,y)\rightarrow\Cor_E(Fx,Fy).\]
By construction, this is the fibre over $y$ of the unique map of bi-Cartesian fibrations
\[\xymatrix{
\{x\}\times_DD^\Lambda \ar[r]\ar[d] & \{Fx\}\times_EE^\Lambda \ar[d] \\
D\ar[r] & E
}\]
sending $x$ to $Fx$. By unicity, this can be identified with the map induced by functoriality of $(-)^\Lambda$, which is exactly the functoriality of correspondences one might expect.
\end{eg}

\subsection{Spans on a bivariant functor}\label{BIV_EXT}

The inclusion $D\subseteq\rm y^\mp(D)$ induces, by \ref{BIV_FUN_COMP}, a functor of restriction $2\Fun(\rm y^\mp(D),K)\rightarrow\Biv(D,K)$.
The goal of this section is to describe a right inverse to this restriction, the functor of \emph{span extension}. This extension is a `global' analogue of the representation of spans obtained in \ref{BIV_LOC}.

\begin{thm}[Extension]
\label{BIV_EXT_THM}

Let $D:1\Cat^{+r},~K:2\Cat$. There is a 2-functor 
\[ \Spanext:\Biv(D,K) \underset{2\Cat}{\rightarrow} 2\Fun(\rm y^\mp(D),K) \]
of \emph{span extension} to the Yoneda image. Its underlying 1-functor is a section of the 1-functor of restriction along $D\subseteq \rm y^\mp(D)$:
\[ H:\Biv(D,K) \quad \return_{1\Cat} \quad \Spanext(H)\circ \rm y^\mp_D \cong H. \] 
For each $x$, $y:D$, and $H:\Biv(D,K)$, $\Spanext(H):\rm y^\mp(D)(\rm y^\mp_Dx,\rm y^\mp_Dy)\rightarrow K(Hx,Hy)$ is equivalent, via the bivariant Yoneda lemma, to the local representation $H_!$ induced by $H$.
\end{thm}

\noindent Of course, $\Spanext$ should be a section of the 2-functor of restriction, and indeed in the conditional context of \S\ref{UNIV} this is a consequence of the universal property (Thm.~\ref{UNIV_EXT_THM}). The stated version is all I could prove unconditionally. The bottleneck is that we only have defined the bivariant Yoneda evaluation map \eqref{BIV_YON_EVAL} as a 1-natural transformation.

Before proceeding to the proof, record a corollary:

\begin{cor}[Weak unicity of correspondences]\label{BIV_EXT_UNIQUE}
Let $D\rightarrow K$ be a 2-category of correspondences of $ D$. There is a 2-equivalence $K\cong \rm y^\mp(D)$ under $D$.
\end{cor}
\begin{proof}
Apply Theorem \ref{BIV_EXT_THM} to obtain an extension $\rm y^\mp(D)\rightarrow K$ that induces the local representation of spans on each mapping space. Since $K$ is a category of correspondences, this functor is fully faithful and essentially surjective.
\end{proof}

\begin{proof}[Proof of Thm.~\ref{BIV_EXT_THM}] The construction occupies the remainder of this section.

\begin{para}[Outline of construction]\label{BIV_EXT_OUTLINE}Starting from a bivariant functor $H:D\rightarrow K$, we will construct a diagram, natural in $H$:
\begin{align*}
\alignhead{Span extension square}
\xymatrix{
  D \ar[d]_{\rm y^\pm} \ar[r]^H 
    \ar@{}[dr]|(.3)\Leftarrow|(.7)\Leftarrow &
    K \ar[d]^{\rm y^2} \ar[dl]|{H^\dagger} \\ 
  \Biv^\op_D \ar[r]_-{H_*} &
    1\PSh K
}\end{align*}
where $H^\dagger$ is obtained by restriction along $H$, and $H_*$ by restriction along $H^\dagger$. More precisely, the induced functors $H^\dagger,H_*$ are defined by formulas 
\begin{align}
\label{EQN_ADJ} \alpha:K,\ \beta:D \qquad H^\dagger\alpha(\beta) &
  = K(\alpha,H\beta) \\
\label{EQN_FWD} F:\Biv_D^\op,\ \alpha:K \qquad H_*F(\alpha) &
  = \Biv_D(F,H^\dagger\alpha).
\end{align}
The composite 2-cell that intertwines the morphisms of the outer square can be identified with an inverse to the Yoneda evaluation mapping \eqref{BIV_YON_EVAL}. By the Yoneda lemma, the outer square is therefore commutative. (We don't actually prove that the Yoneda intertwiner is the composite of the two 2-cells pictured; all we need is the commutativity of the outer square.) 

In particular, $H_*$ maps $\rm y^\mp(D)$ into $K\subset 1\PSh K$. Hence we get an extension mapping 
\[ \rm{Span}:\Biv(D,K) \underset{2\Cat}{\rightarrow} 2\Fun\left(\Biv_D^\op,1\PSh K\right) \] 
which, by the 2-Yoneda lemma for $K$ \eqref{CAT_TENS_YON}, factors through the full subcategory of functors that send the objects of $D$ into $K$. Composing with the restriction 
\[ 2\Fun(D,K)\times_{2\Fun\left(D,1\PSh K\right)} 2\Fun\left(\Biv_D^\op,1\PSh K\right) \longrightarrow 2\Fun(\rm y^\mp(D),K),\]
we obtain the desired extension functor. The fact that this is a right inverse follows from the commutativity 2-cell of the span extension square.\end{para}

\begin{remark}[Opposite version]The same argument could be carried out equally with the covariant Yoneda mapping of $D$: \[\xymatrix{ D\ar[d]_{\rm y}\ar@{}[dr]|(.3)\Rightarrow|(.7)\Rightarrow \ar[r]^H & K\ar[d]^{\rm y}\ar[dl]|{H^\dagger} \\ \Biv_D\ar[r]_-{H_*} & 1\Fun(K,1\Cat)^\op }\] defined by the formulas $H^\dagger\alpha(\beta):= K(H\beta,\alpha)$ and $H_*G(\alpha)=\Biv_D(G,H^\dagger\alpha)$.\end{remark}

\begin{remark}[Natural version]It is somewhat perverse that this argument involves extending $H$ to a morphism from a completion into a cocompletion. It is also possible to approach this using a completion (or cocompletion) on both sides, that is, replacing the bottom-left term by $2\Fun(K,1\Cat)^\op$. In this formulation, the bottom arrow would be a left adjoint \[H_!:\Biv_D\rightarrow2\Fun(K,1\Cat) \] to the restriction functor. 

A typical approach to constructing such a diagram would be as follows:
\begin{enumerate}\item The restriction functor into $\Biv_\rm D\subseteq 1\Cat^\rm D$ preserves weighted 2-limits and 2-colimits, since this inclusion is closed under weighted 2-limits and 2-colimits.
\item By an adjoint 2-functor theorem and the fact that $1\Cat^K$ is \emph{2-presentable} (whatever this means, 1-presheaf categories should be an example), it admits a left adjoint.
\item The commutativity of the outer square follows again from the bivariant and the 2-Yoneda lemmata.\end{enumerate}
This approach may be more natural in the sense that it fits into a general covariance pattern of 2-presheaves, but it also entails additional foundations beyond what we currently have available in 2-category theory --- particularly, an adjoint 2-functor theorem --- so I have not attempted to flesh out the details.

This machinery would also help in making the construction natural in $K$.
\end{remark}

\begin{para}[Yoneda adjoint]
\label{BIV_EXT_BACK}

The 2-functor of \emph{Yoneda pullback} 
\[ (-)^\dagger:\rm{Biv}(D,K) \underset{2\Cat}{\rightarrow} 2\Fun(K,\rm{Biv}_D^\op)^\op \]
 with the formula \eqref{EQN_ADJ} is obtained by repeated currying:
\begin{align*} 
	&& \beta:D~(\pm^r),~ \alpha:K^\op,~ H : \Biv(D,K) \quad 
		&\return  \quad K(\alpha,H\beta) :1\Cat  
\\
\Rightarrow	&& \alpha: K^\op,~ H:\Biv(D,K) \quad 
		&\return\quad H^\dagger\alpha:\Biv_D 
\\
\Rightarrow &&  H:\Biv(D, K) \quad 
		& \return \quad H^\dagger: K^\op\underset{2\Cat}{\rightarrow}\Biv_D
\\
\Rightarrow &&  H:\Biv(D, K) \quad 
		& \return \quad H^\dagger:2\Fun(K,\rm{Biv}_D^\op)^\op. &\mathstrut
\end{align*}
All formulas in this list are 2-functors. Some remarks on how they are obtained:
\begin{itemize}
\item The first line is the composite of the bivariant evaluation 2-functor \eqref{BIV_SYNTAX} with the mapping space 2-functor \eqref{CAT_TENS_MAP}.
\item The second and third lines are derived by bivariant currying followed by usual 2-categorical currying \eqref{BIV_SYNTAX}.
\item The last line passes through the 2-equivalence $2\Fun(K^{\op_1},\Biv_D)\cong 2\Fun\left(K,\Biv_D^{\op_1}\right)^{\op_1}$ which is got by following the chain of equivalences of spaces
\small
\[\xymatrix{
  2\Fun(-, 2\Fun(K^\op, \Biv_D)) \ar@{=}[d]\ar@{-->}[r] & 
    2\Fun\left((-)^\op, 2\Fun(K,\Biv_D^\op) \right) \ar@{=}[d] \ar@{=}[r] &
      2\Fun\left( -, 2\Fun\left(K, \Biv_D^\op\right)^\op \right) 
\\
  2\Fun(- \times K^\op,\Biv_D) \ar@{=}[r] & 
    2\Fun((-)^\op \times K, \Biv_D^\op)
}\]
\normalsize 
where the vertical identifications are from Cartesian closure of $2\Cat$ and the horizontal ones are the action of $\op_1$ as a  1-autoequivalence of $2\Cat$.
\end{itemize}
\end{para}

\begin{remark}[2-categorical Kan property]
\label{BIV_EXT_KAN}

The local representation of spans \ref{BIV_LOC_DEF} defines a 1-natural transformation $\hat\epsilon_H:H^\dagger H\rightarrow\rm y_D^\mp$, to wit:
\begin{align*} 
	&&   \beta:D~(\pm),~ \alpha:D^\op,~H : \Biv(D,K) \quad 
		& \return_{1\Cat} \quad \Cor_D(\alpha,\beta)\underset{1\Cat}{\rightarrow} K(H\alpha,H\beta) && \relax
\\
\Rightarrow && \alpha:D^\op,~H : \Biv(D,K) \quad
		&\return_{1\Cat} \quad \rm y_D^\pm \alpha \underset{\Biv_D}{\longrightarrow} H^\dagger H\alpha 
\\
\Rightarrow && H : \Biv(D,K) \quad
		& \return_{1\Cat} \quad \hat\epsilon_H:2\Fun\left(D,\Biv_D^\op\right)\left(H^\dagger H,\rm y_D^\pm\right).
\end{align*}
This 2-cell ought to exhibit $H^\dagger$ as a 2-categorical left Kan extension of $\rm y^\mp$ along $H$, whatever that means. However, we don't need to actually invoke this formula directly.
\end{remark}

\begin{para}[Pushforward]
\label{BIV_EXT_FWD}

The second 2-functor \eqref{EQN_FWD} 
\[2\Fun\left(K,\Biv_D^\op\right)^\op\rightarrow 2\Fun\left(\Biv_D^\op,1\PSh K\right)\] 
is obtained by the same logic as \ref{BIV_EXT_BACK} applied to 
\begin{align*}
	&& \alpha:K^\op,~F:\Biv_D^\op,\ G:K^\op\rightarrow\Biv_D \quad 
  		&\return \quad  \Biv_D(F,G\alpha):1\Cat &&\mathstrut \\
\Rightarrow && F:\Biv_D^\op,\ G:K^\op\rightarrow\Biv_D \quad 
		&\return\quad G^\dagger F:K^\op\rightarrow1\Cat \\
\Rightarrow && G:K^\op\rightarrow\Biv_D \quad 
		&\return\quad G^\dagger:\Biv_D\rightarrow 1\PSh K 
\end{align*}
The inference rules used to derive these formulas are much the same as those invoked in \ref{BIV_EXT_BACK}. Write $H_*$ instead of $({}^\dagger H)^\dagger$.

Again, there is a 1-natural transformation $\hat e_G:\rm y^2_K\rightarrow G^\dagger G$ whose derivation, parallel to that of $\hat\epsilon_H$ (\ref{BIV_EXT_KAN}), I omit. \end{para}

\begin{para}[Commutativity 2-cell]

Composing \ref{BIV_EXT_BACK} with \ref{BIV_EXT_FWD} we obtain a 2-functor $\Spanext:\Biv(D,K)\rightarrow 2\Fun(\Biv_D^\op, 1\PSh(K))$ with the formula:
\begin{align*}
\alignhead{Span extension, (completed)}
&  H:\Biv(D,K) \quad\return_{2\Cat} \quad H_*:2\Fun(\Biv_D^\op,1\PSh K) \\ 
\Rightarrow\quad & H:\Biv(D,K) \quad\return_{2\Cat} \quad F:\Biv_D^\op,~ \alpha:K^\op \quad\return_{2\Cat} \quad \Biv_D(F,K(\alpha,H-))
\end{align*} 
Now, substituting $F$ via $\rm y^\mp_D$ and applying the bivariant Yoneda lemma \ref{BIV_YON_LEMMA}, we obtain:
\begin{align*}
\alignhead{Span extension, commutativity}
  H:\Biv(D,K) \quad\return_{1\Cat} \quad \beta:D~(\pm,r),~ \alpha:K^\op \quad\return_{2\Cat}\quad \rm{ev}:\Biv_D\left(\rm y^\mp_D\beta,\,K(\alpha,H-)\right) \cong K(\alpha, H\beta).
\end{align*}
where the reader will notice that the level of functoriality in $H$ has reduced to 1, because this is all that the Yoneda evaluation makes available. In particular, for each $H$ the restriction of $H_*$ to $\rm y^\mp(D)$ has image in $K$. In diagrammatic terms, this is providing us with the large commuting triangle in the diagram of 1-categories
\[\xymatrix{
  \Biv(D,K) \ar@{>->}[d]_{\rm y^2_K\circ-} \ar[drr]_(.3){\Spanext} &
    2\Fun(\rm y^\mp(D), K) \ar@{>->}[d]|(.48)\hole \ar[l]_-{-\circ\rm y^\mp_D}
     \\
  \Biv(D,1\PSh (K))
  & 2\Fun(\rm y^\mp(D),1\PSh(K)) \ar[l] &
    2\Fun(\Biv_D^\op, 1\PSh(K)) \ar[l]
}\]
wherein the left-hand square is a pullback, and so $\Spanext$ lifts uniquely to a section of $(-\circ\rm y^\mp_D)$. 
\end{para}

\begin{para}[Local picture]
\label{BIV_EXT_LOC}

It remains to compute the action of $H_*$ on mapping categories for fixed $x,y:D$. This proceeds by precomposing the 2-natural transformation
\[ z:\rm y^\mp(D) \quad \return_{2\Cat} \quad \rm y^\mp(D)(x,z) \underset{1\Cat}{\rightarrow} K(H_*x,H_*z) \]
(justified in \ref{CAT_TENS_YON_NAT}) with $\rm y^\mp_D$  to obtain \eqref{BIV_FUN_COMP} a bivariant natural transformation
\[ y:D\ (\pm,r) \quad \return \quad \rm y^\mp(D)(x,y) \underset{1\Cat}{\rightarrow} K(Hx,Hy) \]
that sends $\iden_x$ to $\iden_{Hx}$. Since the local representation of spans is uniquely characterised by this property, this completes the proof.
\qedhere
\end{para}
\end{proof}

\subsection{Omake: composition of correspondences}
\label{BIV_COMP}

In this section we relate composition of correspondences by taking fibre product of the kernels to composition in the image of a bivariant functor. 

In particular, if $K$ is a category of correspondences in the sense of Def.~\ref{BIV_LOC_CORR}, then under the resulting identification of $K^{\Delta^1}$ with $\rm{Span}_{D,1}$, composition in $K$ can be identified with the composition law of the Segal object $\rm{Span}_{D,\bullet}$ studied in \cite{HigherSegal,GR}.

\begin{para}[Local action by spans]
\label{BIV_LOC_ACT}

In the case $K=1\Cat$, we can curry the local representation of spans to obtain a local \emph{action}
\begin{align*}
\alignhead{Local action by spans}
  \alpha_H:\Cor_D(x,y)\times Hx\rightarrow Hy
\end{align*}  
for each $x,y:D$. We will compare its action on the Yoneda embedding with composition of spans and with composition in the image of a bivariant functor.
\end{para}

\begin{para}[Action of spans on themselves versus composition of spans]
\label{BIV_COMP_ACTION}

We begin by identifying composition of spans with the local representation of spans by the bivariant Yoneda embedding itself:
\begin{align*}
&& (q,p):\Cor_D(x,y) \quad 
	&\return\quad q_!p^!:\rm{y}^\mp_D y\rightarrow \rm{y}^\mp_D x 
\\
\Rightarrow && z:D~(\pm),~(q,p):\Cor_D(x,y) \quad
	&\return\quad q_!p^!:\Cor_D(y,z)\rightarrow\Cor_D(x,z) 
\\ 
\Rightarrow && (q^\prime,p^\prime):\Cor_D(y,z),~ (q,p):\Cor_D(x,y)\quad
	&\return\quad q_!p^!(q^\prime,p^\prime):\Cor_D(x,z)
\end{align*}
(note the arrow reversal in the first line because the target of $\rm y^\mp_D$ is $\Biv_D^\op$ rather than $\Biv_D$). We know --- cf.~the discussion in the proof of Lemma \ref{BIV_YON_UNIV} --- that $q_!p^!(q^\prime,p^\prime)$ is computed by the roof of the pullback diagram
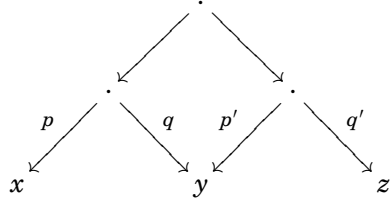
\[\xymatrix{
&& \cdot \ar[dr]\ar[dl] \\
& \cdot \ar[dr]^q\ar[dl]_p && \cdot \ar[dr]^{q^\prime}\ar[dl]_{p^\prime} \\
x && y && z
}\]
so that $q_!p^!(p',q')$ is indeed the composite of the two spans.
\end{para}

\begin{para}[Action of spans on maps in $K$ versus composition in $K$]
\label{BIV_COMP_TRIANGLE}

Let $x:D$. The 2-functor $K\rightarrow 1\Cat$ corepresented by $x$ yields a commuting triangle
\[\xymatrix{
D\ar[r]\ar[dr] 	& K\ar[d]^{K(Hx,-)} \\
		& 1\Cat
}\]
from which we obtain, by functoriality of the local representation of spans \eqref{BIV_LOC_FUNC}, a triangle
\[\xymatrix{
\Cor_D(y,z)\ar[r]\ar[dr] & K(Hy,Hz) \ar[d] \\
& 1\Fun[K(Hx,Hy),K(Hx,Hz)]
}\]
where the vertical arrow is the composition law in the 2-category $K$.

Pulling $K(Hx,Hy)$ out of the target, recover
\[\xymatrix{\Cor_D(y,z)\times K(Hx,Hy)\ar[r]^-\alpha\ar[d]_{(H,\,\iden)}&  K(Hx,Hz) \\
 K(Hy,Hz)\times K(Hx, Hy)\ar[ur]_-\circ }\] 
in which the upper arrow is the local action by correspondences on the functor $K(H-,Hz)$ and the lower is composition in $K$.
\end{para}

\begin{prop}[Composition]
\label{BIV_COMP_PROP}

Let $H:D\rightarrow K$ be a bivariant functor. 
The local action by spans intertwines composition of spans by fibre product with composition in $K$, i.e.\ there is a commuting diagram
\[\xymatrix{ \Cor_D(y,z)\times\Cor_D(x,y)\ar[r]^-{\text{f.p.}}\ar[d] & \Cor_D(x,z)\ar[d] \\
K(Hy,Hz)\times K(Hx,Hy)\ar[r]^-\circ & K(Hx,Hz)  }\]
of 1-categories. 
In particular, if $H$ exhibits $K$ is a 2-category of correspondences for $D$, then composition in $K$ can be identified with composition of spans.
\end{prop}
\begin{proof}
Applying the local action to the local representation $\Cor_D(x,-)\rightarrow K(Hx,H-)$ yields a commuting square
\[\xymatrix{ \Cor_D(y,z)\times\Cor_D(x,y)\ar[r]^-{\alpha_{\rm yx}}\ar[d] & \Cor_D(x,z)\ar[d] \\
\Cor_D(y,z)\times K(Hx,Hy)\ar[r]^-{\alpha_{\rm yHx}} & K(Hx,Hz)  }\]
where, by \ref{BIV_COMP_ACTION}, the top arrow is identified with the fibre product composite of spans. Now stack this square on top of the commuting triangle from \ref{BIV_COMP_TRIANGLE}.
\end{proof}

\begin{remark}
Because the action of spans on themselves \eqref{BIV_COMP_ACTION} is itself a morphism of bivariant functors, we can apply the naturality of the local action to obtain a square 
\[\xymatrix{
 \Cor_D(y,z)\times \Cor_D(x,y)\times\Cor_D(w,x) \ar[r]\ar[d]  & \Cor_D(x,z)\times\Cor_D(w,x) \ar[d] \\ 
 \Cor_D(y,z)\times \Cor_D(w,y) \ar[r] & \Cor_D(w,z)
 }\]
 exhibiting the associativity of composition. Continuing in this way, one imagines that the entire coherently associative structure of the Yoneda image can be reconstructed, and compared to fibre products of spans, explicitly. I will not expand on this further here.
\end{remark}

\begin{para}[Notation]
\label{BIV_LOC_NAME}

If $S_D^+$ and $S_D^-$ are two markings of $D$ with base change, then we can write $\Cor(D,S_D^+,S_D^-)$ or $\Cor_{D\sep (S_D^+,S_D^-)}$ for the 2-subcategory of $\Cor_{D\sep S_D^-}$ whose morphisms are those spans $\cdot\stackrel{S^-_D}{\leftarrow} \cdot \stackrel{S_D^+}{\rightarrow} \cdot$
whose right-way part belongs to $S_D^+$ (and whose wrong-way part belongs to $S_D^-$). 
Proposition \ref{BIV_COMP_PROP} confirms that this is indeed closed under composition.
Note that if $S_D^+ \subsetneq D$, then $D$ does not actually map into $\Cor(D,S_D^+,S_D^-)$.
\end{para}

\section{Universal property (\dag)}\label{UNIV}

We wish to prove that the extension constructed in \S\ref{BIV_EXT} of a bivariant functor of $D$ to a 2-functor of $\Cor_D$ is unique, hence obtaining a universal mapping property for $\Cor_D$ in the form of Theorem \ref{UNIV_EXT_THM}. This is the same as proving the naturality of this extension for change of target category $K$. Sadly, this statement appears not to be quite within the scope of 2-category theory as it has currently been established, and so the arguments of this section are \emph{hypothetical} --- more precisely, conditional on a special case of the 2-dimensional Grothendieck construction discussed below in \S\ref{UNIV_GROT}, which we use to study universal families of correspondence 2-categories.

%The main difficulty is that we have few methods at our disposal for studying (1-)universal families of 2-categories. 2-natural transformations. Via the Grothendieck construction, we can transform this problem --- at least for 2-functors into $1\Cat$ --- into a problem of constructing 2-functors over a base, for which the methods, although still scarce, turn out to be sufficient. 

A reminder: so far, we have only shown that bivariant extensions exist, and that correspondences are unique up to equivalence, but not up to unique equivalence.

\subsection{2-categorical Grothendieck integration}
\label{UNIV_GROT}

In fact, we only need to know the construction for fibrations in 2-categories over a 1-category.

\begin{para}[2-Cartesian fibrations]
\label{UNIV_GROT_CART}

The notion of 2-Cartesian fibration of 2-categories and 2-Cartesian transformations are defined in \cite[Def.\ 11.1.1.2]{GR}, and the dual notions of 2-co-Cartesian fibrations in \cite[(11.1.3.1)]{GR}. Since we only use the Grothendieck construction for fibrations in 2-categories whose base is a 1-category, I reproduce this special case of the definition here: a functor $E\rightarrow D$ is a 2-co-Cartesian fibration if for each $f:x\rightarrow y$ in $D$ and $e:E_x$ there is a \emph{2-co-Cartesian lift} of $f$ with target $e$, that is, a morphism $f_e:e\rightarrow f_!e$ over $f$ such that the induced square
\[\xymatrix{
 E(f_!e,-) \ar[r]^-{-\circ f_e}\ar[d] & E(e,-)\ar[d] \\
 D(y,p-)\ar[r]^-{-\circ f} & D(x,p-)
 }\]
is a pullback of 1-categories. The category of 2-Cartesian fibrations and 2-Cartesian transformations over $D:2\Cat$ forms a subcategory $2\Cart_D$ of the slice $2\Cat\downarrow  D$ of 2-categories over $D$.\end{para}

\begin{hyp}
\label{UNIV_GROT_HYP}

There exists a natural equivalence of the form:
\begin{align*}
&[\text{Grothendieck construction, 2-categorical}]   \\
&  \rm D:1\Cat^\op \quad\underset{1\Cat}{\return} \quad \textstyle\int:2\!\PSh(\rm D)\underset{1\Cat}{\cong}2\Cart_\rm D:\prod.
\end{align*}
that restricts to the identity functor of $2\Cat$ over $D=\bf 1$.
\end{hyp}

\begin{remark}
When using this hypothesis, the 1-categorical Grothendieck construction invoked throughout the paper should be the restriction of this one to fibrations in 1-categories.%As I mentioned before, I believe that this is unique, so it should be automatic.
\end{remark}

\begin{para}[Epistemological status of 2-dimensional Grothendieck integration]
\label{UNIV_GROT_EPIST}

The formula \ref{UNIV_GROT_HYP} has exactly the same form form as \cite[Thm.\ 11.1.1.8(b)]{GR}. Moreover, the equivalence given restricts to $\iden_{2\Cat}$ over $\rm D=\bf 1$, and to the 1-categorical Grothendieck construction \ref{CAT_GROT_INT} over $1\Cat$. As with all statements in \emph{op.~cit}, it can at best be regarded as conditional on the list of unproved statements \cite[\S10.0.4]{GR}.

%A form of 2-categorical Grothendieck construction also appears in \cite[Prop.\ 3.4.18]{Lurie_goodwillie}. However, this seems to be targeted towards 2-functors valued in $1\Cat$ (rather than 1-functors valued in $2\Cat$, which is what we need). In any case, %However, its functoriality in $D$ is only established for morphisms satisfying certain conditions, among which exponentability. Since functors from $\Delta^0$ and $\Delta^1$ are not generally exponentiable, it is not possible to use the arguments of \ref{CAT_GROT_PROPS} to conclude that Lurie's construction really does `what we think it does'. Undoubtedly, the same conclusions can be extracted directly from the nuts and bolts of the (intricate) construction, but it is beyond the scope of the present paper to do this.

%Furthermore, it is not straightforward to relate one side the equivalence given to the notion of 2-Cartesian fibration used here and in \cite{GR}; hence, Lurie's result does not seem to be convenient for use in the present context.
\end{para}

\begin{comment}
A somewhat incidental consequence of this hypothesis is that it allows us to construct subcategories by specification 2-functorially. This can be used, for example, to define $\Biv(-,-)$ as a 2-bifunctor of $1\Cat^{+r}\times 2\Cat$.

\begin{prop}[2-subfunctors $(\dagger)$]
\label{UNIV_GROT_SUBFUNCTOR}

Let $F: D\rightarrow2\Cat$ be a 2-functor. Suppose that for each $x: D$ we are given a 2-subcategory $G(x)\subseteq F(x)$, such that for each $x$, $y:D$ the action 1-functor
\[
    D(x,y) \times G(x) \subseteq D(x,y)\times F(x) \rightarrow F(y)
\]
  has image in $G(y)$.
  
Then there is a unique 2-functor $G:\rm D\rightarrow 2\Cat$ and 2-natural transformation $G\rightarrow F$ inducing the inclusion $G(x)\subseteq F(x)$ for each $x:\rm D$.

\end{prop}
\end{comment}

\subsection{Spans on a family of bivariant functors}
\label{UNIV_EXT}

We invoke our hypothetical 2-dimensional Grothendieck construction \ref{UNIV_GROT_HYP} in the proof of Theorem \ref{UNIV_EXT_THM} to construct a \emph{universal fibration} over $\bf{Biv}_D$. In this section and from now on, $\Cor_D$ stands for a 2-category of correspondences for $D$, which may as well be the bivariant Yoneda image of $D$.

\begin{prop}[Span extension in families]
\label{UNIV_EXT_FAMILY}
Let $D:1\Cat^{+r}$, $I:1\Cat$, and $K:2\Cart_I$. Suppose given a commuting triangle
\[\xymatrix{
D\times I\ar[rr]^{H}\ar[dr]_{\rm{pr_I}} && K\ar[dl]^\pi \\
&I
}\]
where $H$ preserves Cartesian arrows and is right bivariant with base change. Then this triangle can be embedded in a commuting diagram in $2\Cat$:
\[\xymatrix{
  D\times I \ar[dr]|-{i\times \iden_I} \ar[rr]^H \ar[ddr] & &
    K\ar[ddl]^\pi \\
  & \Cor_D\times I\ar[d]|{\mathstrut\rm{pr}_I}\ar[ur]|-{H'} & \\
&I
}\]
where $H'$ is a 2-Cartesian transformation over $I$ whose local action $\Cor_D(x,y)\rightarrow K_i(Hx,Hy)$ for each $i:I$ and $x$, $y:D$ agrees with the local action of spans induced by $H_i$.
\end{prop}
\begin{proof}
We factor $H$ through its span extension $\Spanext(H):\Cor_{D\times I}\rightarrow K$ (\S\ref{BIV_EXT}). This fits into a diagram
\[\xymatrix{
  D\times I \ar[dr] \ar[rr]^H \ar[ddr] & &
    K\ar[ddl]^\pi \\
  & \Cor_{D\times I}\ar[d]|{\pi \circ H'}\ar[ur]^(.4){\Spanext(H)} & \\
&I
}\]
whose outer triangle commutes by hypothesis, the top because $\Spanext(H)$ factorises $H$, and whose lower right triangle commutes by construction. The lower left triangle is the composite of the other three. 

We get a diagram of the promised form by finding a projection $\Cor_{D\times I}\rightarrow\Cor_D$ which, together with $\pi\circ H'$, decomposes the domain as a product. This is provided by Lemma \ref{UNIV_EXT_PROD}. That $H'$ is a Cartesian transformation is Lemma \ref{UNIV_EXT_CART}, and the formula for the local action follows from its unicity \eqref{BIV_LOC} as a natural transformation of functors of $y:D$ that sends $\id_X$ to $\id_{Hx}$.
\end{proof}

\begin{para}[Product structure]
\label{UNIV_EXT_PROJ}

The evident maps from $D\times I$ to $\Cor_D$ and $I=\Cor_I$ (recall that $S_I=\{\text{isos}\}$) are both bivariant, and so the work of \S\ref{BIV_EXT} gives us span extensions
\[ \Spanext(\rm{pr}'_D):\Cor_{D\times I}\rightarrow \Cor_D \qquad \Spanext(\rm{pr}_I):\Cor_{D\times I}\rightarrow I \]
where $\rm{pr}'_D$ is the composite of projection on $D$ with the embedding $D\subseteq\Cor_D$.
These exhibit $\Cor_{D\times I}$ as a product $\Cor_D\times I$. Indeed, by recognition of products (Lemma \ref{CAT_FURTHER_LIMIT}), it is enough to show that the projections induce decompositions:
\begin{enumerate}[label=\arabic*), start=0]
\item $\Ob(\Cor_{D\times I})\cong\Ob(\Cor_D)\times\Ob(I)$;
\item $\Cor_{D\times I}((x,i),(y,j))\cong\Cor_D(x,y)\times I(i,j)$ for each $x,y:\Cor_D$ and $i,j:I$.
\end{enumerate}
The first statement is clear, since both sides are equal to $\Ob(D\times I)=\Ob(D)\times\Ob(I)$. For the second, the map we are looking at is the local representation of correspondences which, following example \ref{BIV_LOC_EX}, is the fibre of the map
\[  (D\times I)^\Lambda \rightarrow D^\Lambda\times I^\Lambda\cong D^\Lambda\times I \]
given by functoriality of $(-)^\Lambda$ over the given source and target. This map is invertible because exponentiation commutes with limits.
\end{para}

\begin{lemma}
\label{UNIV_EXT_PROD}

In the context of Proposition \ref{UNIV_EXT_FAMILY}, let $H':\Cor_{D\times I}\rightarrow K$ be an extension of $H$ that acts on mapping 1-categories through the local representation of spans. The induced 2-functor 
\[ \left(\Spanext(\rm{pr}'_D),\, \pi\circ H'\right):\Cor_{D\times I}\rightarrow \Cor_D\times I \] 
is an equivalence of 2-categories.
\end{lemma}
\begin{proof}
Again, use recognition of products (Lemma \ref{CAT_FURTHER_LIMIT}) to reduce to separate questions about spaces of objects and mapping 1-categories. The former statement is immediate, as the space of objects of $\Cor_{D\times I}$ is the same as that of $D\times I$. 

For the latter, the compatibility \eqref{BIV_LOC_FUNC} of the local representation of spans with change of target category gives a commuting triangle
\[\xymatrix{
\Cor_{D\times I}((x,i),(y,j))\ar[r]\ar[d] & K(H(x,i),H(y,j))\ar[dl] \\
I(i,j)
}\]
where the vertical map is the local projection described in \ref{UNIV_EXT_PROJ}; hence, by \emph{loc.~cit}, the local action \[ \Cor_{D\times I}((x,i),(y,j)) \rightarrow \Cor_D(x,y) \times I(i,j) \] is a 1-equivalence.
\end{proof}

\begin{remark}
Note that we did not actually invoke the product decomposition of $\Cor_{D\times I}$ obtained in \ref{UNIV_EXT_PROJ} directly, but only the local decomposition of its mapping categories. The projection $\Spanext(\rm{pr}_I):\Cor_{D\times I}\rightarrow I$ obtained there is not \emph{a priori} identified with $\pi\circ\Spanext(H)$. However, such an identification does follow \emph{a posteriori} from unicity, which we establish below.
\end{remark}

\begin{lemma}[Spans on a Cartesian family]
\label{UNIV_EXT_CART}

Suppose that $K\rightarrow I$ is a 2-Cartesian fibration and $D\times I\rightarrow K$ is a 2-Cartesian transformation. Then $H'$ is also a 2-Cartesian transformation.
\end{lemma}
\begin{proof}
Since $\Cor_{D\times I}$ is a product by Lemma \ref{UNIV_EXT_PROD}, we are just checking that all arrows in $I$-slices, that is, of the form $(x,i)\stackrel{\iden_x\times\phi}{\longrightarrow} (x,j)$, are mapped to 2-Cartesian arrows in $K$. All such arrows are in the image of $D\times I$, so this follows from the fact that $H'$ is an extension of $H$, which is a 2-Cartesian transformation by hypothesis.
\end{proof}

\begin{thm}[Universal property of correspondences (\dag)]
\label{UNIV_EXT_THM}

A 2-category of correspondences is a 3-universal bivariant extension. There is a unique 1-functor 
\[ \Cor:1\Cat^{+r}\rightarrow2\Cat \]
equipped with a natural transformation $h:\iota\rightarrow\Cor$, where $\iota:1\Cat^{+r}\rightarrow 2\Cat$ forgets the marking, such that $h_D:D\rightarrow\Cor_D$ is a category of correspondences for each $D:1\Cat^{+r}$.
\end{thm}

\begin{proof}
Applying 2-dimensional (contravariant) Grothendieck integration to the tautological 1-natural transformation
\[ [H:D\rightarrow K]:\bf{Biv}_D \quad\return\quad H:D \underset{2\Cat}{\rightarrow}  K
\]
of functors $\bf{Biv}_D\rightarrow 2\Cat$
(recall \eqref{BIV_UNIV_1CAT} that we consider $\bf{Biv}_D$ as a 1-category), we obtain a 2-co-Cartesian transformation:
\[\xymatrix{
D\times\bf{Biv}_D^\op\ar[rr]\ar[dr]	&& \int_{(H,K):\bf{Biv}_D^\op}K\ar[dl] \\
	& \bf{Biv}_D^\op
}\]
Plugging this triangle into Proposition \ref{UNIV_EXT_FAMILY} we embed this in an extended diagram
\[\xymatrix{
  D\times\bf{Biv}_D^\op\ar[r]\ar[dr]	& 
    \Cor_D\times\bf{Biv}_D^\op \ar[d]|{\rm{pr}_{\bf{Biv}_D^\op}} \ar[r] & 
    \int_{(H,K):\bf{Biv}_D^\op}K\ar[dl] \\
  & \bf{Biv}_D^\op
}\]
of 2-Cartesian transformations. Moreover, by the formula for the local action on slices of the functor $H'$ produced by \ref{UNIV_EXT_FAMILY}, the right horizontal map restricts to the identity over $K=\Cor_D$.

 We now reverse the Grothendieck construction to obtain a pointwise bivariant natural transformation from the constant functor $\bf{Biv}_D\rightarrow \bf{Biv}_D$ with value $D\rightarrow\Cor_D$ into the identity functor which is equivalent to the identity on $D\rightarrow\Cor_D$. This exhibits $ D\rightarrow\Cor_D$ as an initial object of $\bf{Biv}_D$, proving 1-universality. Lemma \ref{BIV_UNIV_BOOT} then gives us $\Cor$ as a 1-functor and 3-universality.
\end{proof}

\begin{cor}[Unicity of correspondences]
\label{UNIV_EXT_UNIQUE}

Categories of correspondences are unique.
\end{cor}
\begin{proof}Because they are universal.\end{proof}

\subsection{Correspondences and Cartesian fibrations}
\label{UNIV_CART}

In this section we study a generalisation of the families of bivariant functors studied in \S\ref{UNIV_EXT} in which the source category $D$ is also allowed to vary. We will apply this in the proof of Theorem \ref{UNIV_MON_LAX}, an oplax monoidal generalisation of \ref{UNIV_EXT_THM}. 

\begin{prop}[Integral of a family of markings with base change]
\label{UNIV_CART_MARK}

Let $D:I\rightarrow 1\Cat^{+r}$ be a 1-functor, and mark the contravariant Grothendieck integral $p:\int_{I^\op}D\rightarrow I^\op$ with the union of the markings of the fibres. This marking has base change, and the inclusion of each fibre preserves base change.
\end{prop}
\begin{proof}
If $u:i\rightarrow j$ is a morphism in $I$ and $f:x\rightarrow y$ belongs to $S_{D_j}$, then it follows from the Cartesian properties that
\[\xymatrix{
  u^!x\ar[r]^-{u^x} \ar[d]_{u^!f}	& x\ar[d] \\
  u^!y\ar[r]^-{u^y}		& y
}\]
is a fibre product in $\int_{I^\op}D$. That is, the pullback $u^!f$ of $f$ belongs to $S_{D_i}$.

It remains to observe that a pullback of an element of $S_{D_i}$ in its fibre remains a pullback in the total space of the fibration. Let $x\times_yz:D_i$ be such a pullback, and let $w$ be a cone in $\int_{I^\op}D$ over $x\rightarrow y\leftarrow z$. All projections from $w$ map to the same morphism $u$ in the base. Pulling back along $u$ we identify the two squares
\[\xymatrix{
\Map_u(w,x\times_yz) \ar[r]\ar[d] & \Map_u(w,x)\ar[d] 
  & \Map_{pw}(w,u^!(x\times_yz)) \ar[r]\ar[d] & \Map_{pw}(w,u^!x)\ar[d] \\
\Map_u(w,z)\ar[r]			& \Map_u(w,z)\ar@{}[ur]|\cong
  & \Map_{pw}(w,u^!z)\ar[r]			& \Map_{pw}(w,u^!z)
}\]
and since $u^!$ preserves pullbacks of elements of $S$, the right square is a pullback, so the left square is too. The result follows now from $\Map(w,x\times_yz)=\coprod_{u:pw\rightarrow px}\Map_u(w,x\times_yz)$.
\end{proof}

\begin{remark}

The second half of the proof may also be seen as an application of \cite[Prop.~4.3.1.10]{HTT}.\footnote{Thanks to an anonymous referee for drawing this to my attention.} 
More precisely, \emph{loc.~cit}.~tells us that pullbacks in the fibres remain \emph{$p$-limits}.
We can argue that these are moreover \emph{limits} in $\int_{I^\op}D$ using the facts that the index category of pullbacks has contractible nerve and that the diagrams we are interested project to a point in $I^\op$.

\end{remark}

\begin{prop}[Integral of a family of bivariant functors with base change]
\label{UNIV_CART_BIV}

Let $H:D\rightarrow K$ be an $I$-indexed family of right bivariant functors with base change. Then the contravariant Grothendieck integral $\int_{I^\op}H:\int_{I^\op}D\rightarrow \int_{I^\op}K$ is right bivariant with base change.
\end{prop}
\begin{proof}

Because marked arrows of $\int_{I^\op}D$ are contained in fibres of $p$, bivariance of $\int_{I^\op}H$ can itself be checked on fibres, as can base change along a map contained in a fibre (which by Prop.~\ref{UNIV_CART_MARK} remains a base change in $\int_{I^\op}D$). It remains to check the Beck-Chevalley condition for Cartesian squares of the form
\[\xymatrix{
u^!x \ar[r]^{u^x} \ar[d] & x \ar[d]^f \\
u^!y \ar[r]^{u^y} & y
}\] where $u:i\rightarrow j$ in $I^\op$, $[f:x\rightarrow y]\in S$, and the horizontal arrows are Cartesian over $u$. Explicitly, we are checking that the composite 2-cell implied by the diagram
\[\xymatrix{
H_ju^!y \ar[r] \ar@{=}[dr] & H_ju^!x \ar[r] \ar[d] & H_ix \ar[d] \ar@{=}[dr] \\
& H_ju^!y \ar[r] & H_iy \ar[r] 
  \ar@{}[llu]|(.7)\Leftarrow & H_ix \ar@{}[llu]|(.3)\Leftarrow
}\]
is invertible.

We will embed the left and middle cells of the above diagram into one of the form $J\times \Delta^1\rightarrow K$, where $J$ is a 2-category indexing the left-hand triangle (several choices for $J$ exist; for example, we could take $J=[1,1]^\op$ as in \ref{BIV_BC_MATE}).
We achieve this by identifying $u^!H_i\cong H_j u^!:D_i\rightarrow K_j$, using that $H$ is a natural transformation, and applying a lemma:

\begin{lemma}

  Let $E\rightarrow D$ be a 2-Cartesian fibration over a 1-category $D$, $f:x\rightarrow y$ a morphism in $D$, and $\phi:J\rightarrow E_x$ a 2-functor.
  Then there is a unique 2-functor $\tilde\phi:J\times\Delta^1\rightarrow E$ that for all $j:J$, $\tilde\phi$ sends $\{j\}\times \Delta^1$ to a 2-Cartesian arrow in $E$ starting at $\phi(j)$.

\end{lemma}
\begin{proof}

  First, by pulling back along the map $\Delta^1\rightarrow D$ represented by $f$, we may assume $D=\Delta^1$ and $f$ is the unique non-invertible morphism $0\rightarrow 1$. 
  Applying the Grothendieck construction, since $0$ is initial in $\Delta^1$, there is a unique morphism from the constant functor $\Delta^1\rightarrow 2\Cat$ with value $J$ into $\fibre_E$ that restricts to $\phi$ at $0$.
  Passing back to the fibrations side, we find that similarly there is a unique 2-Cartesian transformation $J\times \Delta^1\rightarrow E$
  restricting to $\phi$ over $0$.
\end{proof}

Applying the lemma to $J$ as above, we find a diagram of the form
\[\xymatrix{
  u^!H_iy \ar[r] \ar@{=}[dr] \ar[rrd]|(.3)\Leftarrow^(.7){u^{H_iy}} & 
    u^!H_ix \ar@{..>}[d]|\hole \ar[rrd] 
\\
  & u^!H_iy \ar[rrd] &
      H_iy \ar[r]|{f^!} \ar@{=}[dr] &
        H_ix \ar[d] |{f_!} \ar@{=}[dr] 
\\
   &&
      & H_iy \ar[r]|{f^!} &
          H_ix; \ar@{}[llu]|(.3)\Leftarrow|(.7)\Leftarrow
}\]
moreover, using the uniqueness clause, we can identify the `rear' square panel with the central commuting square of the diagram defining the conjugate 2-cell.

Now, by the zig zag identities for the adjunction $f_!\dashv f^!$, the endomorphism of $f^!:K_j(H_jy,H_jx)$ implied by the bottom-right diamond is equivalent to the identity; hence, the same is true of the conjugate transformation which is obtained by precomposing with $u^{H_jy}$.
\end{proof}

\begin{lemma}[Spans on a Cartesian fibration]
\label{UNIV_CART_CART}

The induced projection $\Spanext(p):\Cor(\int_{I^\op} D)\rightarrow {I^\op}$  is a 2-Cartesian fibration with 2-Cartesian arrows the images of those in $\int_{I^\op}D$.
\end{lemma}
\begin{proof}
If $u:x\rightarrow y$ is a Cartesian arrow in $\int_{I^\op}D$, then we wish to prove that
\[\xymatrix{
  \Cor(w,x)\ar[r]^{u\circ-}\ar[d] & \Cor(w,y) \ar[d] \\
  I(pw,px) \ar[r] & I(pw,py)
}\]
is a pullback in $1\Cat$ for each $w:D$, where $\Cor_D(-,-)\rightarrow I(p-,p-)$ is the induced representation, which by Example \ref{BIV_LOC_EX} is induced by functoriality of $(-)^\Lambda$.

Observe that because $I^\Lambda\cong I^{\Delta^1}$ the induced representation $D^\Lambda\rightarrow I^\Lambda= I^{\Delta^1}$ factors through projection $D^\Lambda\rightarrow D^{\Delta^1}$ on the right-way map. We may therefore factor the above square as:
\[\xymatrix{
  \{w\} \times_D D^{\Delta^\sharp} \times_D D^{\Delta^1} \times_D \{x\} \ar[d]\ar[r] &
    \{w\} \times_D D^{\Delta^\sharp} \times_D D^{\Delta^1} \times_D \{y\} \ar[d] 
\\
  D^{\Delta^1} \times_D \{x\}   \ar[r] \ar[d] &
    D^{\Delta^1} \times_D \{y\}   \ar[d] 
\\
  I(pw,px)\ar[r] & I(pw,py)  
}\]
Now, the upper square is Cartesian because it is obtained by pullback along $\rm{tar}:\{w\}\times_D D^{\Delta^\sharp}\rightarrow D$, and the lower is Cartesian because $u$ is a Cartesian arrow in $D$.
\end{proof}

\begin{prop}[Correspondences commute with Grothendieck integration]
\label{UNIV_CART_GROT}

The contravariant integral of the universal bivariant extension
\[ \textstyle \int_{I^\op}h: \int_{I^\op}D \rightarrow \int_{I^\op}\Cor_D \]
is itself a universal bivariant extension.

If $g:\int_{I^\op} D\rightarrow K$ is a bivariant functor over $I^\op$, and $g$ preserves the set of arrows Cartesian over some fixed arrow $\phi:I$, then so too does the span extension $\int_{I^\op}\Cor_D\rightarrow K$ of $g$.
\end{prop}
\begin{proof}

By Proposition \ref{UNIV_CART_BIV}, $\int_{I^\op}h$ is bivariant, and so we deduce from the universal property of correspondences (Thm.~\ref{UNIV_EXT_THM}) a commuting diagram
\[\xymatrix{
D_i \ar[d]\ar[r] & \Cor(D_i) \ar@{=}[r] \ar[d] & \Cor(D_i) \ar[d] \\
\int_{I^\op}D \ar[r]\ar[dr] & \Cor(\int_{I^\op}D) \ar[r]^-{\Spanext{\int_{I^\op}h}}\ar[d] & \int_{I^\op}\Cor_D \ar[dl]
\\
 & {I^\op}
}\] 
(for any $i:I$). By the top-right square, the restriction of $\Spanext(\int_{I^\op}h)$ to each fibre is an isomorphism. Since by Lemma \ref{UNIV_CART_CART}, $\Cor(\int_{I^\op}D)$ is 2-Cartesian fibred over $I$, we deduce that $\Spanext(\int_{I^\op}h)$ is an equivalence. The last statement follows because the Cartesian arrows in $\int_{I^\op}\Cor_D$ are exactly the images of Cartesian arrows in $D$.
\end{proof}

\begin{remark}
The results of this section admit dual statements for co-Cartesian fibrations in marked categories with collar change and right bivariant functors with collar change.
\end{remark}

\subsection{Monoidal structure}
\label{UNIV_MON}

To upgrade Theorem \ref{UNIV_EXT_THM} to a monoidal version we just need it to work in families indexed by the commutative operad. We laid the groundwork for this in \S\ref{BIV_UNIV}.

\begin{para}[Commutative monoids]
\label{UNIV_MON_DEF}

Let $\Fin$ denote the category of finite sets. When necessary, we consider $\Fin$ as a marked category with base change by marking all morphisms. A commutative monoid in a 1-category $C$ is the data of a product-preserving 1-functor $\lie c_1\Cor(\bf{Fin})\rightarrow C$. (Note that $\lie c_1\Cor(\bf{Fin})$ is 2-truncated, hence we consider it here as a $(2,1)$-category.) The 1-category of commutative monoids in $C$ is denoted $\CMon(C)$.

(This is not the same definition as the one that appears in \cite{HA}, but is equivalent via a natural embedding of Segal's category $\Gamma$ into $\lie c_1\Cor(\bf{Fin})$ \cite[Prop.~C.1]{bachmann2020norms}.)\footnote{Thanks to an anonymous referee for indicating the reference to me.}

\end{para}

\begin{para}[Enrichment, tensoring, and powering of commutative monoids]
\label{UNIV_MON_EN}

Let $K$ be an $n$-category with finite products. The 1-category $\CMon(K)$ of commutative monoids in $K$ inherits an $n$-categorical enrichment as a full $n$-subcategory of $n\Fun(\lie c_1\Cor(\Fin), K)$.

Similarly, if $K$ is $n$-powered over $(n-1)\Cat$ \eqref{CAT_BOOT_EX}, and this powering preserves finite products (which implies that they are actually limits in the $n$-categorical sense as in \ref{EX_MON_nCAT}), then $\CMon(K)$ too inherits an $n$-powering in the same manner. 

In the case $K=n\Cat$ we write $n\Fun^\otimes(-,-)$ for mapping spaces in $\CMon(n\Cat)$. The $(n+1)$-functor $n\Fun^\otimes(A,-):n\Cat^\otimes\rightarrow n\Cat$ is $n$powered over $n\Cat$ by the argument of \ref{CAT_BOOT_EX}. This makes the bootstrap argument of Lemma \ref{CAT_BOOT_SUBREP} available in the monoidal setting.
\end{para}

\begin{para}[Marked symmetric monoidal categories and base change]

A commutative monoid in $1\Cat^{+}$ consists of the data of a symmetric monoidal category $(D, \otimes)$ equipped with a marking $S_D$ of $D$ which is stable under tensor product.

It is further a commutative monoid in the subcategory $1\Cat^{+r}$ if $S_D$ has base change and pullbacks of members of $S_D$ are preserved by tensor product. This condition is satisfied, for example, in Cartesian monoidal structures (with any marking) --- compare example \ref{EX_MON_SPANS}.

A \emph{symmetric monoidal bivariant functor} (with base change) $D\rightarrow K$ is a symmetric monoidal functor $(D,\otimes)\rightarrow (K,\otimes)$ whose underlying functor is bivariant (with base change). A \emph{symmetric monoidal base change exact natural transformation} is a symmetric monoidal natural transformation (between symmetric monoidal bivariant functors) whose underlying natural transformation is base change exact. The subcategory of $2\Fun^\otimes(D,K)$ spanned by bivariant functors with base change and base change exact natural transformations is denoted $\Biv^\otimes(D, K)$.
\end{para}

\begin{defn}[Monoidal universality]
\label{UNIV_MON_UNIV_DEF}

A symmetric monoidal functor $D\rightarrow K$ is said to be an $n$-\emph{universal monoidal bivariant extension}, where $n:\{1,2,3\}$, if it induces an $n$-natural equivalence
\[ 2\Fun^\otimes(K, K') \tilde\rightarrow \Biv^\otimes(D, K') \]
for any $K':2\Cat$. In other words, $H$ is universal if it is universal as a $\lie c_1\Cor(\Fin)$-indexed family of bivariant extensions in the sense of Def.~\ref{BIV_UNIV_DEF}.
\end{defn}

\begin{lemma}[Correspondences commutes with limits]
\label{UNIV_MON_PROD}

The 1-functor $\Cor: 1\Cat^{+r} \rightarrow 2\Cat$ preserves limits.
\end{lemma}
\begin{proof}
Let $D:I \rightarrow 1\Cat^{+r}$ be a 1-functor. Since $\Ob (\lim_{i:I} D_i)=\lim_{i:I}(\Ob(D_i)) = \lim_{i:I}(\Ob(\Cor_{D_i}))$, it will suffice to check the statement on mapping spaces. That is, for $x,y:\lim_{i:I}D_i$, we have
\[ \Cor\left(\lim_{i:I}D_i\right) (x, y) = \lim_{i:I}\Cor_{D_i}(\bar x, \bar y) \]
where $\bar x$ etc.\ denotes the relevant projection of $x$. This follows from the fact that these 1-categories are naturally fibres of the universal fibration $D_i^\Lambda \rightarrow D_i\times D_i$, and this construction, being an exponential, commutes with limits.
\end{proof}

\begin{thm}[Monoidal universality]
\label{UNIV_MON_THM}

Let $D:\rm{CMon}(1\Cat^{+r})$ and let $h:D\rightarrow\Cor_D$ be a 2-category of correspondences. There is a unique symmetric monoidal structure on $\Cor_D$ and extension of $h$ to a symmetric monoidal functor $(D,\otimes)\rightarrow(\Cor_D,\otimes)$, and this functor is a 3-universal symmetric monoidal bivariant extension.

\smallskip\noindent
There is a unique 1-functor
\[ \Cor^\otimes:\rm{CMon}(1\Cat^{+r}) \rightarrow \rm{CMon}(2\Cat) \]
equipped with a symmetric monoidal natural transformation $h$ from the forget-the-marking functor such that for each $D:\rm{CMon}(1\Cat^{+r})$, $h_D:D\rightarrow \Cor_D$ is a category of correspondences. Moreover, $h$ is a 3-universal symmetric monoidal bivariant extension.
\end{thm}
\begin{proof}
First, 3-universal symmetric monoidal extensions exist by Theorem \ref{UNIV_EXT_THM}, Lemma \ref{UNIV_MON_PROD}, and Lemma \ref{BIV_UNIV_BOOT}. For the uniqueness of the monoidal structure, we must prove that the fibre of
\[ D^{\times-}\downarrow2\Fun^\times(\Cor(\Fin),2\Cat) \quad\rightarrow \quad D^{\times-}\downarrow2\Fun^\times(\Fin^\op,2\Cat) \quad\cong \quad D\downarrow2\Cat \]
over $h$ is contractible. Because $h:D\rightarrow\Cor_D$ is universal, it has no endomorphisms, i.e.~$\{h\}\subset D\downarrow2\Cat$ is a full subcategory. The fibre over $h$ is therefore also a full subcategory of $D^{\times-}\downarrow2\Fun(\Cor(\Fin),2\Cat)$. By Lemmas \ref{UNIV_MON_PROD} and \ref{BIV_UNIV_BOOT}, it is exactly the subcategory consisting of 3-universal symmetric monoidal extensions. 

The second statement follows in much the same way, but applied to the universal family over $\rm{CMon}(1\Cat^{+r})\times\Cor(\Fin)$.
\end{proof}

\begin{para}[Extension: oplax symmetric monoidal functors]
\label{UNIV_MON_LAXDEF}

An \emph{oplax} symmetric monoidal 2-functor between monoidal 2-categories $(C,\otimes)$ and $(D,\otimes)$ is defined, using the contravariant Grothendieck construction, to be a commutative triangle
\[\xymatrix{
\int_{\Cor(\Fin)}C^\otimes \ar[rr] \ar[dr] && \int_{\Cor(\Fin)}D^\otimes \ar[dl] \\
& \Cor(\Fin)^\op
}\]
in $2\Cat$ whose restriction to the subcategory $\Fin\subset\Cor(\Fin)^\op$ of inert morphisms is a 2-Cartesian transformation. (Presumably, this notion forms part of a theory of 2-co-operads into which I dare not delve further here.)

An oplax symmetric monoidal 2-functor $h:D\rightarrow K$ is an $n$-\emph{universal} oplax symmetric monoidal bivariant extension if restriction induces an $(n-1)$-equivalence
\[ 2\Fun^{\rm{oplax-}\otimes}(K,K') \tilde\rightarrow \Biv^{\rm{oplax-}\otimes}(D,K') \]
for any symmetric monoidal 2-category $K'$. Here, of course, $\Biv^{\rm{oplax-}\otimes}$ is the category of oplax symmetric monoidal functors whose underlying functor is bivariant.
\end{para}

\begin{thm}
\label{UNIV_MON_LAX}

The universal symmetric monoidal bivariant extension is also a universal oplax symmetric monoidal bivariant extension.
\end{thm}
\begin{proof}
This is a special case of Proposition \ref{UNIV_CART_GROT} with $I=\Cor(\Fin)$ and $D=D^\otimes$.
\end{proof}

\begin{eg}[Spans operad]
\label{EX_MON_SPANS}

Suppose that $D$ is finitely complete. Then \cite[\S2.4.1]{HA} endows $D$ with a Cartesian symmetric monoidal structure $(D,\times)$. Because products commute with other limits, equipping it with the maximal marking makes it an object of $\rm{CMon}(1\Cat^{+r})$.

Applying Theorem \ref{UNIV_MON_THM}, we obtain a symmetric monoidal structure on $\Cor_D$ compatible with the inclusion $(D,\times)\rightarrow(\Cor_D,\otimes)$.  
Beware that the induced operation is \emph{not} usually the product in $\Cor_D$. Indeed, since $\Cor_D$ is equivalent to its opposite, coproducts are products, and in typical examples where these exist (like $\bf{Fin}$) it is given by the coproduct in the underlying category. See also \cite[\S2]{Toen_operad} for more details.

This monoidal structure is discussed in \cite{Haugseng}, where it is shown that all objects are actually \emph{self-dual} with evaluation and coevaluation each given by the correspondence 
\[\xymatrix{ 
& X\ar [dr]\ar[dl]_\delta \\ X\times X && \mathrm{pt} 
}\]
considered in the appropriate direction.
\end{eg}

\section{Examples}\label{EX}

Here we finally come to some non-trivial examples of bivariant functors to which the constructions of \S\S\ref{BIV_EXT}, \ref{UNIV_EXT}, and \ref{UNIV_MON} can be applied. The first example, \S\ref{EX_MON}, provides an alternative presentation of a basic construction internal to higher algebra and is therefore perhaps mainly of interest to specialists. The second set \S\ref{EX_COEF}, by contrast, come from geometry and cohomology theory. This fertile source of bivariant constructions provides the blueprint for further development of the theory presented in this paper.

The constructions of \S\ref{EX_MON} and the first half of \S\ref{EX_COEF} do not depend on the conditional results of \S\ref{UNIV}. The universal property comes into play when we \emph{classify} Cartesian monoidal structure (Prop.~\ref{EX_MON_CHAR}) and when we need span extension to work in families for the monoidal versions of categories of coefficients (\ref{EX_COEF_OPLAXCART} onwards).

\subsection{Cartesian and co-Cartesian symmetric monoidal structures}
 \label{EX_MON}

The notions of \emph{Cartesian} and \emph{co-Cartesian} symmetric monoidal structures \cite[\S2.4.1]{HA} are fundamental to symmetric monoidal $\infty$-category theory since, as far as I know, all other examples of symmetric monoidal structures are built out of these. 

The span extension theorem provides a quick and clean method for constructing  these structures, and generalises easily to classify Cartesian monoidal structures on $n$-categories. Combining it with the monoidal universal property, we obtain an alternative characterisation \eqref{EX_MON_CHAR} of when a symmetric monoidal structure is (co-)Cartesian.

\begin{para}
\label{EX_MON_CONSTRUCT}

Since $1\Cat$ admits finite products, it is equivalent (via right Kan extension) to the category $\Fun^\times(\Fin^\op,1\Cat)$ of product preserving functors into $1\Cat$ from $\Fin^\op$, the free Lawvere theory on one object.
This universal property of $\Fin^\op$ can be deduced, for example, from the fact that $\Fin\subset\Spc$ is generated under finite coproducts by $\mathrm{pt}$, and considering it as a $P^K_R(\mathrm{pt})$ in the notation of \cite[\S5.3.6]{HTT} with $R=\emptyset$, $K=\Fin$.
The construction associates to $C:1\Cat$ a functor $C^\times$ which takes:
\begin{itemize}
\item a finite set $I$ to the category $C^I$;
\item a map $\phi:I\rightarrow J$ to a 1-functor $\phi^{-1}:X^J\rightarrow X^I$ which on objects behaves as $(x_i)_{j:J} \return (x_{\phi i})_{i:I}$.
\end{itemize}
Now if $C$ admits finite products, resp.~coproducts, then for all $\phi$ the functor $\phi^{-1}$ admits a right, resp.~left adjoint given by
\begin{align*}
  \phi_*(x_i)_{i:I} & = \textstyle(\prod_{\phi i=j}x_i)_{j:J} \\
  \phi_!(x_i)_{i:I} & = \textstyle(\coprod_{\phi i=j}x_i)_{j:J}
\end{align*}
In other words, in this case the functor $C^\times:\Fin^\op\rightarrow1\Cat$ is right, resp.~left bivariant.

In either case, the pushforwards commute with pullbacks of finite sets, that is pushouts in $\Fin^\op$, so $C^\times$ has collar change. Re-orienting by taking opposites following the rules of \ref{BIV_FUN_OPP}, we deduce two right bivariant functors with base change --- in the finite products case, apply $\op_1\op_2$ to get
\begin{align*}
  C^\times_*: \Fin &\rightarrow 1\Cat^{\op_1\op_2} \\
  \Rightarrow\qquad (C,\times):\Cor(\Fin) & \underset{1\Cat}{\rightarrow} 1\Cat \mskip90mu
\end{align*}
where in the second line we composed with an anti-auto-equivalence $\Cor(\Fin)\cong\Cor(\Fin)^{\op_1}$ and suppressed the $\op_2$ because we restricted to 1-cores.  Similarly, in the finite coproducts case, applying $\op_1$:
\begin{align*}
  C^\times_!: \Fin &\rightarrow 1\Cat^{\op_1} \\
  \Rightarrow\qquad(C,\sqcup):\Cor(\Fin) & \underset{1\Cat}{\rightarrow} 1\Cat. \mskip90mu
\end{align*}
It is straightforward to check that these monoidal structures are indeed Cartesian, resp.~co-Cartesian. For example, in the Cartesian case the unit morphism is right adjoint to $C\rightarrow \rm{pt}$, so it is a final object as required. I omit the remaining details.
\end{para}

\begin{para}[Classification of (co-)Cartesian structures]

The results of \cite[\S2.4]{HA} also provide uniqueness statements for Cartesian and co-Cartesian symmetric monoidal structures, when they exist. Unfortunately, our results in \S\ref{UNIV_MON} do not immediately allow us to recover the same statements, because they only classify 2-functors out of $\Cor(\Fin)$, not 1-functors, which is what we actually want. For the moment, then, we do not have anything new to say about proving this statement.

Shifting perspective, however we can combine our results with those of \cite{HA} to obtain a novel characterisation of these structures:
\end{para}

\begin{prop}[Reformulation of (co-)Cartesian condition (\dag)]
\label{EX_MON_CHAR}

A symmetric monoidal structure on a category $C$ is Cartesian, resp.~co-Cartesian, if and only if it extends to a symmetric monoidal 2-functor of $\Cor(\Fin)^{\op_2}$, resp.~$\Cor(\Fin)$.
\end{prop}
\begin{proof}
We will argue for the co-Cartesian case. We have seen that $C^\times:\Fin^\op\rightarrow 1\Cat$ is right bivariant with base change if and only if $C$ admits finite coproducts. This holds if and only if it extends to a 2-functor $\Cor(\Fin)\rightarrow 1\Cat$; the `only if' clause is provided by span extension and `if' is because $\Fin\rightarrow\Cor(\Fin)$ is bivariant. Moreover, in this case the extension is unique (Thm.~\ref{UNIV_MON_THM}) and its restriction to $\lie c_1\Cor(\Fin)$ is a co-Cartesian monoidal structure.

On the other hand, \cite[Cor.~2.4.1.9]{HA} tells us that under exactly the same condition, $C$ admits a unique co-Cartesian monoidal structure. By uniqueness, our symmetric monoidal 2-functor extends this one.
\end{proof}

\begin{para}[Cartesian symmetric monoidal $n$-categories]
\label{EX_MON_nCAT}

The construction of \eqref{EX_MON_CONSTRUCT} extends without change to put `Cartesian' symmetric monoidal structures on $n$-categories with finite products (or coproducts), provided that these are interpreted suitably as $n$-limits, i.e. by the condition
\[  C(-,x\times y) \underset{(n-1)\Cat}{\cong} C(-,x)\times C(-,y). \]
The existence of $n$-categorical finite products is equivalent to the existence of $n$-categorical right adjoints to the reindexing maps $\phi^{-1}$ --- see aside \ref{CAT_ADJ_nCAT}. 
\end{para}

\subsection{Categories of coefficients}\label{EX_COEF}

Generally speaking, the phrase `categories of coefficient systems' refers to an assignment to a class of geometric objects $X$ of a stable, triangulated, or Abelian category $\sh A(X)$, with adjoint pushforward and pullback functors $f^*\dashv f_*$ associated to morphisms $f$, which satisfy a base change isomorphism. The base change isomorphism usually holds only under some restrictions on one or other of the maps defining the fibre product.

These data form part of a more sophisticated system popularly known as the \emph{six functor formalism}. It is discussed in the context of correspondences in \cite[Part III, Intro.]{GR}.

\begin{eg}[Locally compact Hausdorff spaces]
\label{EX_COEF_LCH}

Let $\bf{LCH}$ denote the category of locally compact Hausdorff spaces, marked with the class of proper morphisms, and $\sh A(X)$ one of the following:
\begin{itemize}
\item the Abelian category of Abelian sheaves on $X$;
\item derived $\infty$-category of cohomologically bounded below complexes of Abelian sheaves on $X$;
\item the category of sheaves of spaces on $X$;
\end{itemize}
considered as a contravariant functor of $\bf{LCH}$.

By the proper base change theorem \cite[Thm.\ 20.18.2]{stacksproject} (for Abelian sheaves), \cite[Thm.\ 7.3.1.16]{HTT} (for spaces), $\sh A$ is:
\begin{itemize}

\item right bivariant with base change, considered as a functor $\bf{LCH}\rightarrow 1\Cat^{\op_1}$;

\item left bivariant with collar change, considered as a functor $\bf{LCH}^\op \rightarrow 1\Cat$.

\end{itemize}
Following the first interpretation and applying $\op_1$, we obtain a 2-functor 
\begin{align*}
&&\Spanext(\sh A):\Cor(\bf{LCH},\,\rm{all},\,\rm{proper}) & \rightarrow 1\Cat^{\op_1}
\\
\Rightarrow && \Spanext(\sh A):\Cor(\bf{LCH},\,\rm{proper},\, \rm{all}) & \rightarrow 1\Cat  \mskip150mu
\end{align*}  
extending $\sh A$ (see \ref{BIV_LOC_NAME} for notation).
\end{eg}

\begin{eg}[\'Etale coefficient systems]
\label{EX_COEF_ET}

Let $\bf{Sch}_k$ be the category of quasi-compact, quasi-separated schemes over a field $k$ of characteristic $p$, marked with the set of proper morphisms, and $\sh A(X)$ the derived $\infty$-category of complexes of \'etale torsion sheaves on $X$.
Then just as in Ex.~\ref{EX_COEF_LCH}, by the proper base change theorem in \'etale cohomology \cite[Thm.~57.87.11]{stacksproject} $\sh A$ is right bivariant with base change if considered as a functor into $1\Cat^\op$.

Dually, mark $\bf{Sch}_k$ now with the class of \emph{smooth} morphisms and let $\sh A_{\hat p}\subseteq\sh A$ be the full subcategory of sheaves with torsion prime to $p$, considered as a covariant functor. Then the smooth base change theorem \cite[Thm.~57.85.2]{stacksproject} tells us that $\sh A_{\hat p}:\bf{Sch}_k\rightarrow 1\Cat$ is a right bivariant functor with base change, hence defines a $2$-functor
\[ \Spanext(\sh A_{\hat p}): \Cor(\bf{Sch}_k,\,\rm{all},\,\rm{smooth}) \rightarrow 1\Cat. \]
\end{eg}

\begin{eg}[Quasi-coherent cohomology]
\label{EX_COEF_QC}

Let $\bf{Sch}$ be the category of all qcqs schemes, $S_\bf{Sch}$ the set of flat morphisms, and $\rm{QC}(X)$ either the Abelian category of quasi-coherent $\sh O_X$-modules or the derived $\infty$-category of quasi-coherent complexes, considered as a covariant functor of $X$.
Then using flat base change in quasi-coherent cohomology \cite[Lemma 30.5.2]{stacksproject}, $\rm{QC}$ is left bivariant with base change. Hence, we obtain a 2-functor
\[ \Spanext(\rm{QC}):\Cor(\bf{Sch},\,\rm{all},\,\rm{flat})^{\op_2} \rightarrow 1\Cat. \]

A version for derived stacks with perfect morphisms \cite[Prop.\ 3.10]{IntegralTransforms} yields a 2-functor
\[ \Spanext(\rm{QC}): \Cor(\bf{Stk},~\rm{perfect},~\rm{all}) \longrightarrow 1\Cat \]
where $\bf{Stk}$ denotes the category of derived stacks. 
This admits a dual version using ind-coherent sheaves desribed in \cite[\S5]{GR}.
\end{eg}

\begin{para}[Upgrading the target 2-category]
\label{EX_COEF_UPGRADE}

Most of these examples naturally take values not in $1\Cat$, but in some upgraded version of the same --- Abelian categories, dg-categores, and their presentable or compactly generated analogues. For Theorem \ref{UNIV_EXT_THM} to give us 2-functors from correspondences into such an upgraded target, the pullback-pushforward adjunctions must be adjunctions in that 2-category.

If the target category $K$ is a subcategory of $1\Cat$ which is fully faithful on 1-cells, meaning that $K(x,y)$ is a full subcategory of $1\Cat(x,y)$ for all $x,y:K$, then any adjunction in $1\Cat$ between morphisms in $K$ lifts uniquely to an adjunction in $K$. The Beck-Chevalley condition is also inherited by the subcategory. This logic applies to the examples:
\begin{itemize}
\item Presentable categories and colimit-preserving functors;
\item Compactly generated categories and colimit and compact object preserving functors;
\item Abelian categories and additive, left exact, right exact, or exact functors;
\item Grothendieck Abelian categories and colimit and compact object preserving functors.
\end{itemize}

For the examples taking values in complexes, one expects an upgrade to a functor taking values in $\bf{dgCat}$. Since this is not a subcategory of $1\Cat$, but rather of the 2-category of categories enriched in complexes, the construction of the adjunction requires more attention. The Beck-Chevalley condition, however, can still be inferred from the image in $1\Cat$ by extensionality of natural transformations \eqref{CAT_TENS_EXT}.
\end{para}

\begin{eg}[Abelian sheaves, revisited]
\label{EX_COEF_ABREV}

Returning to Example \ref{EX_COEF_LCH}: the pullback and pushforward functors of Abelian sheaves are both left exact functors of Abelian categories, and so they yield adjunctions in the 2-category $\bf{AbCat}^\rm{lex}$ of Abelian categories and left exact functors. Hence, the spans extension upgrades to a 2-functor
\[ \Spanext(\sh A):\Cor(\bf{LCH},\,\rm{proper},\, \rm{all}) \longrightarrow \bf{AbCat}^\rm{lex} .\]
\end{eg}

\begin{para}[Oplax monoidal version, Cartesian case]
\label{EX_COEF_OPLAXCART}

All of the preceding examples can be made into oplax monoidal functors. The case where the target is $1\Cat$ is easiest to explain: by (the opposite of) \cite[Prop.~2.4.3.8]{HA}, for any 1-categories with finite products $C$, $D$ the forgetful map
\[ \Fun^{\rm{oplax}\otimes}((C,\times),(D,\times))\rightarrow 1\Fun(C,D)\]
is an equivalence, so the examples of bivariant functors in \ref{EX_COEF_LCH}--\ref{EX_COEF_QC} admit a unique such enhancement. 

Now, considering $1\Cat$ as a symmetric monoidal 2-category with the Cartesian structure discussed in \ref{EX_MON_nCAT} and applying Theorem \ref{UNIV_MON_LAX}, we obtain oplax symmetric monoidal functors
\begin{align*}
 \Spanext(\sh A_{\hat p}): (\Cor(\bf{Sch}_k,\,\rm{all},\,\rm{smooth}),\otimes) &\rightarrow (1\Cat,\times)
\\ 
\Spanext(\rm{QC}):(\Cor(\bf{Sch},\,\rm{all},\,\rm{flat}),\otimes)^{\op_2} &\rightarrow (1\Cat,\times) \end{align*} 
where $\otimes$ is the symmetric monoidal structure on correspondences induced by Cartesian product in the base category \eqref{EX_MON_SPANS}. Here we used that $\op_2$ preserves (op)lax monoidal functors.
(The case of \ref{EX_COEF_LCH} is a bit trickier since it starts with a contravariant functor, so effectively the monoidal structure is Cartesian on one side but co-Cartesian on the other.)
\end{para}

\begin{remark}
At least the case of quasi-coherent sheaves admits a more powerful \emph{strongly} monoidal version where we must be more careful about the target category. For example, \cite[Thm.\ 4.7]{IntegralTransforms} ought to provide a strongly monoidal 2-functor into a 2-category of presentable dg-categories with the presentable tensor product. I will defer further discussion of this idea to a later work.
\end{remark}

\printbibliography
\end{document}